\documentclass[preprint,12pt,1p]{elsarticle}                     

\usepackage{amsmath,amssymb}
\usepackage{rotating}
\usepackage{amsthm}
\usepackage{multirow}
\usepackage[all]{xy}
\usepackage{shuffle}
\usepackage[labelfont = bf , justification = centering , singlelinecheck = false]{caption}
\DeclareTextSymbol{\degre}{OT1}{23}

\def	\N	{	\ensuremath	{\mathbb	{N}	}	}
\def	\Q	{	\ensuremath	{\mathbb	{Q}	}	}
\def	\Z	{	\ensuremath	{\mathbb	{Z}	}	}
\def	\R	{	\ensuremath	{\mathbb	{R}	}	}
\def	\C	{	\ensuremath	{\mathbb	{C}	}	}

\newcommand	{\seq}[1]	{	\ensuremath	{	\underline{\mathbf{#1}\!}\,
										}
						}

\newcommand	{\crochet}[2]	{	\ensuremath	{	[\![
														\, #1 \, ; \, #2 \,
												]\!]
											}
							}

\def \p	{	\ensuremath	{\bullet}	}

\def \sh	{	\text{sh}	}
\def \ch	{	\text{ch}	}

\def \re	{	\ensuremath	{	\Re e \,}	}
\def \im	{	\ensuremath	{	\Im m \,}	}

\def \Symmetrality	{Symmetr\textbf{\textit {\underline {a}}}lity }
\def \Symmetrelity	{Symmetr\textbf{\textit {\underline {e}}}lity }
\def \Symmetrility	{Symmetr\textbf{\textit {\underline {i}}}lity }

\def \symmetrality	{symmetr\textbf{\textit {\underline {a}}}lity }
\def \symmetrelity	{symmetr\textbf{\textit {\underline {e}}}lity }

\def \symmetral	{symmetr\textbf{\textit {\underline {a}}}l }
\def \symmetrel	{symmetr\textbf{\textit {\underline {e}}}l }
\def \symmetril	{symmetr\textbf{\textit {\underline {i}}}l }

\gdef\stuffle{\;%
  \setlength{\unitlength}{0.0125cm}%
  \begin{picture}(20,10)(220,580) 
  \thinlines 
  \put(220,592){\line( 0,-1){ 10}} 
  \put(220,582){\line( 1, 0){ 20}} 
  \put(240,582){\line( 0, 1){ 10}} 
  \put(230,592){\line( 0,-1){ 10}} 
  \put(225,587){\line( 1, 0){ 10}} 
  \end{picture}\; 
}

\newtheorem{Definition}{Definition}
\newtheorem{Theorem}{Theorem}
\newtheorem*{Theorem*}{Theorem}
\newtheorem{Property}{Property}
\newtheorem{Lemma}{Lemma}
\newtheorem{Conjecture}{Conjecture}
\newtheorem{Corollary}{Corollary }

\newproof{Proof}{Proof}

\begin{document}

\begin{frontmatter}

\title{The Algebra of Multitangent Functions}
\author{Olivier Bouillot}
\ead{olivier.bouillot@math.u-psud.fr}
\ead[url]{http://www.math.u-psud.fr/$\sim$bouillot/}
\address{	D\'epartement de Math\'ematiques	\\
			B\^atiment $425$					\\
			Facult\'e des Sciences d'Orsay		\\
			Universit\'e Paris-Sud $11$			\\
			F-91405 Orsay Cedex
		}

\begin{abstract}
	\indent
	Multizeta values are numbers appearing in many different contexts. Unfortunately, their
	arithmetics remains mostly out of reach.
	
	In this article, we define a functional analogue of the algebra of multizetas values, namely the algebra of multitangent functions, which are $1$-periodic
	functions defined by a process formally similar to multizeta values.
	
	We introduce here the fundamental notions of reduction into monotangent functions, projection onto multitangent functions and that of trifactorisation,
	giving a way of writing a multitangent function in terms of Hurwitz multizeta functions. This explains why the multitangent algebra
	is a functional analogue of the algebra of multizeta values. We then discuss the most important algebraic and analytic properties of these functions and their
	consequences on multizeta values, as well as their regularization in the divergent case.
	
	Each property of multitangents has a pendant on the side of multizeta values. This allows us to propose new conjectures, which have been checked up to the
	weight $18$.
\end{abstract}


\end{frontmatter}

\tableofcontents

\section	{Introduction}

\subsection{The Riemann zeta function at positive integers}
An interesting problem, but still unsolved and probably out of reach today, is to determine the polynomial relations over $\Q$ between the numbers $\zeta(2)$ , $\zeta(3)$ , $\zeta(4)$ , $\cdots$ , where the Riemann zeta function $\zeta$ can be defined by the convergent series
$$\zeta (s) = \displaystyle	{	\sum	_{n = 1}
										^{+ \infty}
										\frac	{1}	{n^s}
							}
$$ in the domain $\re s > 1$~.

Thanks to Euler, we know the classical formula for all even integers $s$:
$$\zeta (s) = \displaystyle	{	\frac	{(2 \pi)^s}	{2}
					\frac	{|B_s|}	{s!}
				} \ ,
$$
where the $B_s$'s are the Bernoulli numbers.
From this, one can see that $\Q[\zeta(2) , \zeta(4) , \zeta(6) , \cdots] = \Q[\pi^2]$~. Now, Lindemann's theorem on the transcendence of $\pi$ concludes the discussion for $s$ even, as the last ring is of transcendence degree $1$~.

Euler failed to give such a formula for $\zeta(3)$. Actually, the situation is quite more complicated concerning the values of the Riemann zeta function at odd integers. Essentially, nothing is known about their arithmetics. One had to wait the end of the twentieth century to see the first results:
\begin{enumerate}
	\item In $1979$, Roger Ap\'ery proved that $\zeta(3) \not \in \Q$ (see \cite{Apery})~;
	\item In $2000$, Tanguy Rivoal proved there are infinitely many numbers in the list $\zeta(3)$ , $\zeta(5)$ , $\zeta(7)$ , $\cdots$ which are irrational numbers (see \cite{Rivoal})~;
	\item in $2004$, Wadim Zudilin showed that there is at least one number in the list $\zeta(5)$ , $\cdots$ , $\zeta(11)$ which is irrational (see \cite{Zudilin2})~.
\end{enumerate}

One conjectures that each number $\zeta (s)$ , $s \geq 2$, is a transcendental number. To be more precise, the following conjecture is expected to hold:
\begin{Conjecture}	\label{main diophantian conjecture}
	The numbers $\pi$ , $\zeta(3)$ , $\zeta(5)$ , $\zeta(7)$ , $\cdots$ are algebraically independent over $\Q$~.
\end{Conjecture}

\subsection	{The multizeta values}

The notion of multizeta value has been introduced in order to study questions related to this conjecture. Multizeta values are a multidimensional generalization of the values of the Riemann zeta function $\zeta$ at positive integers, defined by:
\begin{equation}	\label{defMZV}
	\displaystyle	{	\mathcal{Z}e^{s_1, \cdots , s_r}
						=
						\sum	_{0 < n_r < \cdots < n_1}
								\frac	{1}
										{{n_1}^{s_1} \cdots {n_r}^{s_r}}
					} \ ,
\end{equation}
for any $r \geq 1$ and $(s_1, \cdots, s_r) \in (\N^*)^r$ with $s_1 \geq 2$.

The weight of the multizeta values $\mathcal{Z}e^{s_1, \cdots, s_r}$ is $s_1 + \cdots + s_r$~.

\bigskip

Their first introduction dates back to the year $1775$ when Euler studied in his famous article \cite{Euler} the case of length $2$. In this work, he proved numerous remarkable relations between these numbers, like $\mathcal{Z}e^{2,1} = \mathcal{Z}e^3$ or more generally:
$$	\forall p \in \N^*\ ,\ 
	\sum	_{k = 1}
		^{p - 1}
		\mathcal{Z}e^{p + 1 - k, k}
	=
	\mathcal{Z}e^{p + 1}\ .	
$$

Although they sporadically appeared in the mathematics as well as in the physics literature, 
we can say that they were forgotten during the XIX${}^{\text{th}}$ century and during most of
the XX${}^{\text{th}}$ century. In the last $70$'s, these numbers have been reintroduced by
Jean Ecalle in holomorphic dynamics under the name ``moule zeta\"ique'' (See ). He used them as auxiliary coefficients in order to construct
some geometrical and analytical objects, such as solutions of differential equations with specific
dynamical properties. Mathematicians have been definitely convinced of the interest of these numbers
by their many apparitions in different contexts during the late $80$'s. Finally, these numbers began
to be studied for themselves.

Today, multizeta values arise in many different areas like in:
\begin{enumerate}
	\item Number theory (search for relations between multizeta values, in order to study the hypothetical algebraic independence of values of Riemann's zeta function ; arithmetical dimorphy) : see \cite{Ecalle3}, \cite{Waldschmidt}, \cite{Zudilin} for example.
	\item Quantum groups, knot theory or mathematical physics (with the Drinfeld associator which has multizeta values as coefficients):
			see \cite{Broadhurst}, \cite{Broadhurst-Kreimer}, \cite{Ihara3} ou \cite{Kreimer}.
	\item Resurgence theory and analytical invariants (in many cases, these invariants are expressed in term of series of multizetas values) : see \cite{Bouillot} and \cite{Bouillot-Ecalle}.
	\item the study of Feynman diagrams : see \cite{Broadhurst}, \cite{Broadhurst-Kreimer} or \cite{Kreimer}.
	\item the study of $\mathbb{P}^1 - \{0 ; 1 ; \infty\}$ (through the Grothendieck-Ihara program):
			see \cite{Ihara1}, \cite{Ihara-Matsumoto}, \cite {Matsumoto} for example.
	\item the study of the ``absolute Galois group'': see \cite{Ihara3} for example.
\end{enumerate}

In regard of Conjecture $\ref{main diophantian conjecture}$, one of the most important questions is
the understanding of the relations between multizeta numbers. There are numerous relations between
these numbers, coming in particular from their representation as iterated series \eqref{defMZV} or as iterated
integrals. Indeed, it is now a well-known fact that multizeta values have a representation as iterated integrals
which can be seen in the following way.

If we consider the $1$-differential forms
$$\omega_0 = \displaystyle	{	\frac	{dt}	{t}	} \text{ and } \omega_1 = \displaystyle	{	\frac	{dt}	{1 - t}	} \ ,$$
the iterated integral \label{iterated integral representation}
\begin{equation}	\label{defMZViteratedIntegral}
	\displaystyle	{	\mathcal{W}a^{\alpha_1, \cdots \hspace{-0.01cm},\hspace{0.02cm} \alpha_r}
						=
						\int	_{0 < t_1 < \cdots < t_r < 1}
								\omega_{\alpha_1} \cdots \omega_{\alpha_r}
					}
\end{equation}
is well defined when $(\alpha_1 ; \cdots ; \alpha_r) \in \{0 ; 1\}^r$ satisfies $\alpha_1 = 1$
and $\alpha_r = 0$~.

It is easy to see that there is a relation between the ``functions'' $\mathcal{Z}e^\p$ and $\mathcal{W}a^\p$: 
$$	\mathcal{Z}e^{s_1, \cdots \hspace{-0.01cm},\hspace{0.02cm} s_r}
	=
	\mathcal{W}a^{1, 0^{[s_r - 1]}, \cdots \hspace{-0.01cm},\hspace{0.02cm}  1, 0^{[s_1 - 1]}} \  ,
$$
for all $r \geq 1$ and $(s_1 ; \cdots ; s_r) \in (\N^*)^r$ such that $s_1  \geq 2$~.

It is clear that the $\mathcal{W}a^{\alpha_1 , \cdots , \alpha_r}$'s span an algebra, while the $\mathcal{Z}e^{s_1, \cdots , s_r}$'s also
span an algebra. Among others, the product of two elements of one of these two algebras (which are usually called a quadratic relation)
is a particularly important relation. These two types of relations (one for each algebra) allow us to express a product
of two multizeta values as a $\Q$-linear combination of multizeta values in two different ways. One conjectures that these
two families (up to a regularization process) span all the other relations between these numbers
(see \cite{Waldschmidt} or \cite{Zudilin}). This conjecture, out of reach today, would in particular
show the absence of relations between multizeta values of different weights, and so the transcendence
of the numbers $\zeta(s)$ , $s \geq 2$~.

\subsection	{On multitangent functions}

In this article, we will present an algebra of functions, the algebra of multitangent functions,
which is in a certain sense a good analogue of the algebra of the multizeta values. 
Let us first mention two ideas underlying the definition of multitangent functions.
\\

First, the essential ideas leading to the explicit computation of $\zeta(2n)$ , 
where $n \in \N$, is a symmetrization of the set of summation, that is to say a transformation that
allows us to transform a sum over $\N$ into a sum over $\Z$~. By the same idea, we are able to compute numerous sums of the form
$	\displaystyle	{	\sum	_{m \in \N^*}
								\frac	{\omega^{m r}}	{m^r}
					}
$ , where $\omega$ is a root of unity.

Consequently, it is a natural idea to try to symmetrize the summation simplex of multizeta values.
\\

Next, some well-known ideas are interesting to stress out. One knows that working with numbers imposes a certain rigidity, while working with functions, which will be evaluated afterwards to a particular point, gives more flexibility. One also knows that working with periodic functions gives us access to a whole panel of methods.
\\

The simplest suggestion of a functional model of multizeta values is to consider the Hurwitz multizeta functions:
$$	z \longmapsto \mathcal{H}e_+^{s_1, \cdots , s_r} (z)
	 =
	\displaystyle	{	\sum	_{0 < n_r < \cdots < n_1}
								\frac	{1}
										{(n_1 + z)^{s_1} \cdots (n_r + z)^{s_r}}
					} \ ,
$$
for any $r \geq 1$ and $(s_1, \cdots, s_r) \in (\N^*)^r$ with $s_1 \geq 2$.
The advantage of these functions is to have a very simple link with the multizetas values:
$$	\mathcal{H}e_+^{s_1, \cdots , s_r} (0) = \mathcal{Z}e^{s_1, \cdots , s_r}   \ ,
$$
where $r \geq 1$ and $(s_1, \cdots, s_r) \in (\N^*)^r$ with $s_1 \geq 2$. 
\\

For the sequel of this article, we also define the functions
$$	z \longmapsto \mathcal{H}e_-^{s_1, \cdots , s_r} (z)
	 =
	\displaystyle	{	\sum	_{- \infty < n_r < \cdots < n_1 < 0}
								\frac	{1}
										{(n_1 + z)^{s_1} \cdots (n_r + z)^{s_r}}
					} \ ,
$$
for any $r \geq 1$ and $(s_1, \cdots, s_r) \in (\N^*)^r$ with $s_r \geq 2$.
\\

Unfortunately, this choice does not seem to be the best one, according to the previous remarks: these functions are not periodic and the summation set is not symmetric... So, we are led to modify the model by considering the functions:
\begin{equation}	\label{defMTGF}
	\displaystyle	{	z \longmapsto	\mathcal{T}e^{s_1 , \cdots , s_r} (z)
										=
										\sum	_{-\infty < n_r < \cdots < n_1 < +\infty}
												\frac	{1}
														{(n_1 + z)^{s_1} \cdots (n_r + z)^{s_r}}
					} \ ,
\end{equation}
for all $r \geq 1$ and $(s_1, \cdots, s_r) \in (\N^*)^r$ with $s_1 \geq 2$ and $s_r \geq 2$.

Obviously, these are  $1$-periodic functions and the set of summation is a symmetric set. Nevertheless, what is gained on one side is obviously lost on the other one: in spite of similar expressions, the link with multizeta values is not so clear. However, this link does exist and is actually stronger than the one with  Hurwitz multizeta functions (see \textsection \ref{reduction} and \textsection \ref{reduction2}) .

We are going to refer to these functions as ``\textit{multitangent functions}''. The prefix ``multi'' characterizes the summation set in more than one variable; the suffix ``tangent'' comes from the link between Einsenstein series and the cotangent function. A more representative name would have been ``multiple cotangent functions'' or ``multicotangent functions'', but we preferred to simplify it by forgetting the syllable ``co'', which does not alter its quintessence.

\bigskip

To the best of our knowledge, this family of functions had never been studied from the point of view of special functions, even if it is an interesting and completely natural mathematical object. There are, actually, three good reasons to study such a family of functions, in an algebraic as well as in an analytical way:
\begin{enumerate}
	\item The multitangent functions seem to have appeared for the first time in resurgence theory and holomorphic dynamics, in a book of Jean Ecalle (see \cite{Ecalle1}, vol. $2$ as well as \cite{Bouillot}, or the survey \cite{Bouillot-Ecalle}). Consequently, these functions have some direct applications.
	\item The multitangent functions are deeply linked to multizeta values, at least because of an evident formal similarity. In a naive approach, we can raise  the same questions as for multizeta values, but this time for multitangent functions.
	\item The multitangent functions are a multidimensional generalization of the Eisenstein series, which have been used by Eisenstein to develop his theory of trigonometric functions in his famous article of $1847$ (see \cite{Eisenstein} or \cite{Weil} for a modern approach). So, interesting facts may emerge from this generalization.
\end{enumerate}

\subsection	{Eisenstein series}

The series considered by Eisenstein are defined for all $z \in \C - \Z$ by:
$$	\displaystyle	{	\varepsilon_k (z) 
						 =
						\sum	_{m \in \Z}
								\frac	{1}	{(z + m)^k}
						\ ,
					}
$$
where $k \in \N^*$ and the Eisenstein summation process defined as the summation over $\N^*$ of the terms of index $m$ and $-m$ is used for $k = 1$ (see \cite{Weil}):
$$	\displaystyle	{	\varepsilon_1(z)	=	\sum	_{m \in \N^*}
														\left (
																\frac	{1}	{(z + m)^k}
																+
																\frac	{1}	{(z - m)^k}
														\right )
											=	\frac	{\pi}	{\tan (\pi z)} \ .
					}
$$

As Eisenstein himself said, ``the fundamental properties of these simply-periodic functions reveal themselves through consideration of a single identity'' (see \cite{Eisenstein}):
$$	\frac	{1}	{p^2 q^2}
	=
	\frac	{1}	{(p + q)^2}
	\left (
			\frac	{1}	{p^2}
			+
			\frac	{1}	{q^2}
	\right )
	+
	\frac	{2}	{(p + q)^3}
	\left (
			\frac	{1}	{p}
			+
			\frac	{1}	{q}
	\right )
	 .
$$

From this, he would obtain some identities, which are non trivial at a first sight, between these series.
About the ingenuity and the virtuosity of Eisenstein, Andr\'e Weil compared his work with one of the most
difficult works, even today, of the last period of creation of Beethoven: the Diabelli variations.
It is a work of art based from the most harmless theme which is and which, during the variations
following one another, will generate a prodigious and extremely rich musical universe which is full of
delicacy, but also at the same time full of pianos and compositional virtuosity. The parallel to show
the beauty of the results obtained by Eisenstein is crystal clear.

In his ``variations'', Eisenstein obtained, in particular, the following relations:
\begin{eqnarray}
	{\varepsilon_2}^2 (z)			&=&	\varepsilon_4 (z) + 4\zeta(2) \varepsilon_2 (z) \ .					\label{symetrelite eisenstein}
	\\
	\varepsilon_3 (z)\phantom{3}	&=&	\varepsilon_1 (z) \varepsilon_2 (z) \ .								\label{symetrelite eisenstein divergente 1}
	\\
	3\varepsilon_4 (z)				&=&	{\varepsilon_2 (z)}^2 + 2 \varepsilon_1 (z) \varepsilon_3 (z) \ .	\label{symetrelite eisenstein divergente 2}
\end{eqnarray}

Eisenstein also proved that each of his series is in fact a polynomial with real coefficients in $\varepsilon_1$~. In our study of the algebraic relations between multitangent functions, we will find another proof of the relations $(\ref{symetrelite eisenstein})$ , $(\ref{symetrelite eisenstein divergente 1})$ and $(\ref{symetrelite eisenstein divergente 2})$~. These are particular cases of more general relations: the relation (\ref{symetrelite eisenstein}) is a mix of what we will call the relations of \symmetrelity (see \S \ref{reminder on symmetrel mould} or Appendix \ref{AppendixSymmetrelity}) and of the reduction of multitangent function into monotangent functions (see \S \ref{reduction}), while relations (\ref{symetrelite eisenstein divergente 1}) and (\ref{symetrelite eisenstein divergente 2}) are the archetype of relations of \symmetrelity for divergent multitangent functions (see \S \ref{prolongement des multitangentes au cas divergent}).

Let us mention that although Weil preferred in \cite{Weil} the notation $\varepsilon_k$ in honour of
Eisenstein. From now on, we will systematically use the notation $\mathcal{T}e^s$ coming from multitangent
functions. Also, in connection with the name ``multitangent functions'', we shall name them ``monotangent
functions'' in order to mean that the sequence is of length one.

\subsection{Results proved in this article}

From the three fundamental reasons evoked before, we have initiated a complete study of multitangent functions. It uses intensively the mould notations and mould calculus developped by Jean Ecalle. When it will be necessary, the reader will be given explanations, otherwise he may referred himself to Appendix \ref{Elements de calcul moulien debut}.

The first definition used in mould calculus is what a sequence is. It is simply a list of element of a set. For example, $\text{seq}(\N^*)$ will denote in the sequel of this article the set of sequences of positive integers:
$$	\text{seq}(\N^*)
	=
	\displaystyle	{	\{ \emptyset \}
						\cup
						\bigcup	_{r \in \N^*}
								\{
									(s_1 ; \cdots ; s_r) \in (\N^*)^r
								\}
						\ .
					}
$$

We will also consider three subsets of $\text{seq}(\N^*)$:
$$	\begin{array}{lll}
		\mathcal{S}_b^\star	&=&	\{ (s_1 ; \cdots ; s_r) \in \text{seq}(\N^*) \,; s_1 \geq 2\} \ .
		\vspace{0.2cm}	\\
		\mathcal{S}_e^\star	&=&	\{ (s_1 ; \cdots ; s_r) \in \text{seq}(\N^*) \,; s_r \geq 2\} \ .
		\vspace{0.2cm}	\\
		\mathcal{S}_{b,e}^\star &=& \{ (s_1 ; \cdots ; s_r) \in \text{seq}(\N^*) \,; s_1 \geq 2 \text{ and } s_r \geq 2 \} \ .
	\end{array}
$$

The first important properties of multitangent functions (see \S \ref{definition des multitangentes et premieres proprietes} and \S \ref{algebraic properties}) are:
\begin{Property}	\label{property_intro}
	\begin{enumerate}
		\item	The mould $\mathcal{T}e^\p$ of multitangent function is a \symmetrel mould, that is, for all sequences $\seq{\pmb{\alpha}}$ and
				$\seq{\pmb{\beta}}$ in $\mathcal{S}^{\star}_{b,e}$, we have
				$$	\displaystyle	{	\mathcal{T}e^{\seq{\pmb{\alpha}}} (z) \mathcal{T}e^{\seq{\pmb{\beta}}} (z)
										=
										\sum	_{	\gamma
													\in
													sh\text{\textbf{\underline{\textit{e}}}}(\seq{\pmb{\alpha}}, \seq{\pmb{\beta}})
												}
												\mathcal{T}e^{\seq{\pmb{\gamma}}} (z)
									} \text { , for all } z \in \C - \Z\ ,
				$$
				where the set $sh\text{\textbf{\underline{\textit{e}}}}(\seq{\pmb{\alpha}}, \seq{\pmb{\beta}})$ is a finite subset of $\mathcal{S}^{\star}_{b,e}$.
		\item	There are many $\Q$-linear relations between of multitangent functions.
	\end{enumerate}
\end{Property}

In order to explain the terminology used in this property, let us mention here that a mould is a function with a variable number of variables, a \symmetral or \symmetrel mould is subject to a constraint which give a lot of equations to be satisfied. For example, a \symmetral mould $Ma^\p$ will satisfy
$$	\begin{array}{lll}
			Ma^{x} Ma^{y}
			&=&
			Ma^{x,y} + Ma^{y,x} \ ,
			\vspace{0.1cm}
			\\
			Ma^{x,y} Ma^{y}
			&=&
			Ma^{y, x, y} + 2Ma^{x, y, y} \ ,
	\end{array}
$$
while a \symmetrel mould will satisfy
$$	\begin{array}{lll}
			Me^{x} Me^{y}
			&=&
			Me^{x,y} + Me^{y,x} + Me^{x + y} ,
			\vspace{0.1cm}
			\\
			Me^{x,y} Me^{y}
			&=&
			Me^{y, x, y} + 2Me^{x, y, y} + Me^{x + y, y} + Me^{x, 2y} \ .
	\end{array}
$$
For instance, $\mathcal{W}a^\p$ (defined in \eqref{defMZViteratedIntegral}) is a \symmetral mould, while $\mathcal{Z}e^\p$ (defined in \eqref{defMZV}) is a \symmetrel one. The reader shall consult Appendix \ref{Elements de calcul moulien debut}, especially sections \ref{AppendixDefinition}, \ref{AppendixSymmetrality} and \ref{AppendixSymmetrelity} for the formal definitions of the notions of mould and of \symmetrality/\symmetrelity\!\!.
\\

In one word, the first point of the property \ref{property_intro} allows us to find more than one half of all the known algebraic relations between multizeta values (the relation of \symmetrelity and a few of double-shuffle relations), while the second point allows us to find conjecturally exactly the other algebraic relations between multizeta values (the relation of \symmetrality and the other double-shuffle relations).
\\

We will also see that each multitangent function has a simple expression in terms of multizeta values and monotangent functions. We will also determine that a sort of converse of this property is true: the algebra of multitangent functions is a module over the algebra of multizeta values. The first property is called the ``\textit{reduction into monotangent functions}'' (see \S \ref{reduction}), while the second property is called ``\textit{projection onto multitangent functions}'' (see \S \ref{projection}).

\begin{Theorem}	\textit{(Reduction into monotangent functions)}	\\
	\label{Theorem-Reduction_Intro}
	For all sequences $\seq{s} = (s_1 ; \cdots ; s_r) \in \text{seq} (\N^*)$, there exists an explicit family
	$	(z^{\seq{s}}_k)_{k \in \crochet{0}{M}}
		\in
		\left (
				Vect_\Q (\mathcal{Z}e^{\seq{\sigma}})_{\seq{\sigma} \in \mathcal{S}_b^{\star}}
		\right )
		^{M + 1}
	$, with
	$M = \displaystyle	{\max	_{i \in \crochet {1}{r}} s_i	}$, such that:
	$$	\mathcal{T}e^{\seq{s}} (z)
		=
		z^{\seq{s}}_0
		+
		\displaystyle	{	\sum	_{k = 1}
									^M
									z^{\seq{s}}_k \mathcal{T}e^k (z)
						}\text{ , where }z \in \C - \Z\ .
	$$
	Moreover, if $\seq{s} \in \mathcal{S}^\star_{b,e}$, then $z^{\seq{s}}_0 = z^{\seq{s}}_1 = 0$ .
\end{Theorem}

Let us notice that this theorem states that the multitangent $\mathcal{T}e^{1,2}$ for instance admit an expression
but is not yet defined since $(1,2) \not \in \mathcal{S}_{b,e}^\star$. This will be done in \S \ref{prolongement des multitangentes au cas divergent}.

From an algebraic point of view, let us define some algebras more or less related to the first point of the property $\ref{property_intro}$:
$$	\begin{array}{@{}lll}
		\mathcal{M}ZV_{CV} = \text{Vect}_{\Q}	\left (
														\mathcal{Z}e^{\seq{s}}
												\right )
												_{	\seq{s} \in \mathcal{S}^\star_b	}
		&
		\hspace{-1.5cm}\text{and}
		&
		\mathcal{M}ZV_{CV,p} = \text{Vect}_{\Q}	\left (
														\mathcal{Z}e^{\seq{s}}
												\right )
												_{	\seq{s} \in \mathcal{S}^\star_b
													\atop
													||\seq{s}|| = p
												}
		\ ,
		\\
		\\
		\mathcal{H}MZF_{CV,+} = \text{Vect}_{\Q}	\left (
															\mathcal{H}e_+^{\seq{s}}
													\right )
													_{	\seq{s} \in \mathcal{S}^\star_b	}
		&
		\hspace{-1.5cm}\text{and}
		&
		\mathcal{H}MZF_{CV,+,p} = \text{Vect}_{\Q}	\left (
															\mathcal{H}e_+^{\seq{s}}
													\right )
													_{	\seq{s} \in \mathcal{S}^\star_b
														\atop
														||\seq{s}|| = p
													}
		\ ,
		\\
		\\
		\mathcal{H}MZF_{CV,-} = \text{Vect}_{\Q}	\left (
															\mathcal{H}e_-^{\seq{s}}
													\right )
													_{	\seq{s} \in \mathcal{S}^\star_e	}
		&
		\hspace{-1.5cm}\text{and}
		&
		\mathcal{H}MZF_{CV,-,p} = \text{Vect}_{\Q}	\left (
															\mathcal{H}e_-^{\seq{s}}
													\right )
													_{	\seq{s} \in \mathcal{S}^\star_e
														\atop
														||\seq{s}|| = p
													}
		\ ,
		\\
		\\
		\mathcal{M}TGF_{CV} = \text{Vect}_{\Q}	\left (
														\mathcal{T}e^{\seq{s}}
												\right )
												_{	\seq{s} \in \mathcal{S}^\star_{b,e}	}
		&
		\hspace{-1.4cm}\text{and}
		&
		\mathcal{M}TGF_{CV,p} = \text{Vect}_{\Q}	\left (
															\mathcal{T}e^{\seq{s}}
													\right )
													_{	\seq{s} \in \mathcal{S}^\star_{b,e}
														\atop
														||\seq{s}|| = p
													}
		\ ,
		\\
		\\
		\mathcal{H}MZV_{CV,\pm} = \text{Vect}_{\Q}	\left (
															\mathcal{H}e_+^{\seq{s}^1} \mathcal{H}e_-^{\seq{s}^2}
													\right )
													_{	\seq{s}^1 \in \mathcal{S}^\star_{b}
														\atop
														\seq{s}^2 \in \mathcal{S}^\star_{e}
													} \ ,
	\end{array}
$$
where $p \in \N$, the weight of a sequence $\seq{s} = (s_1, \cdots , s_r) \in \N^\star$ is defined by: $$||\seq{s}|| = s_1 + \cdots + s_r \ .$$

Using this notation, we can state the following:

\begin{Theorem} \textit{(Projection onto multitangent functions)}	\\
	The following conjectural statements are equivalent:
	\begin{enumerate}
		\item For all $p \geq 2$ ,	$	\displaystyle	{	\mathcal{M}TGF_{CV , p} = \bigoplus	_{k = 2}
																								^{p}
																								\mathcal{M}ZV_{CV , p - k} \cdot \mathcal{T}e^k
														} \ .
									$
		\item $\mathcal{M}TGF_{CV}$ is a $\mathcal{M}ZV_{CV}$-module.
		\item For all sequences $\seq{\pmb{\sigma}} \in \mathcal{S}_e^\star$ ,
				$\mathcal{Z}e^{\seq{\pmb{\sigma}}} \mathcal{T}e^2 \in \mathcal{M}TGF_{CV,||\seq{\pmb{\sigma}}|| + 2}$ .
	\end{enumerate}
\end{Theorem}

We will see that the duality reduction/projection is a very important process (see \S \ref{algebraic properties}). In one sentence, we can sum up all the study by saying:
\begin{center}
	\textit	{	``the algebra of multitangent functions is a functional analogue of the algebra of multizeta values: each result on multizeta values has a
				translation in the algebra of multitangent functions, and conversely.''
			}
\end{center}
We can also sum up this study by the following diagram:

$$	\xymatrix	{		\mathcal{M}ZV_{CV}	\ar@{.>}@<4pt>[ddd]	^{\text{projection}}	&&&		\mathcal{H}MZF_{+,CV}	\ar[lll]	_{\text{evaluation at 0}}
																														\ar@{^(->}[ddd]
						\\
						\\
						\\
						\mathcal{M}TGF_{CV}	\ar[uuu]			^{reduction}
											\ar@{^(->}[rrr]		^{\text{trifactorization}}	&&&		\mathcal{H}MZF_{\pm, CV}
				}
$$

In this diagram, which will be constructed throughout the article as an evolutive one,
the trifactorization is an explicit expression of each multitangent function in terms of
Hurwitz multizeta functions. Using it, we will be able to regularize divergent multitangent
functions (see \S \ref{prolongement des multitangentes au cas divergent}), that is to say
multitangent functions depending on a sequence
$\seq{s} \in \text{seq} (\N^*) - \mathcal{S}_{b,e}^\star$ , as $\mathcal{T}e^{1,2}$ for instance. This explains that we allow
such sequences in Theorem~$\ref{Theorem-Reduction_Intro}$~.
\\

We will also give some analytical properties of the multitangent functions (see \S \ref{analytiques}),
such as their Fourier expansion or their upper bound on the half-plane, which would be useful for direct
applications. Finally, we will perform some explicit computations (see \S \ref{calculs explicite})
to obtain:
\begin{Property} \label{Property2Intro}
	Let $n \in \N^*$ and $k \in \N$.	\\
	Let us also set $E$ the floor function and define for $(k ; n) \in \N \times \N^*$ the functions $t_{k,n}$ by:
	$$	\forall x \in \R \ , \ 
		t_{k,n} (x)
		=
		\left \{
				\begin{array}{ll}
					\cos^{(n - 1)}(x)	&	\text{ , if }k \text{ is odd.}	\\
					\sin^{(n - 1)}(x)	&	\text{ , if }k \text{ is even.}	\\
				\end{array}
		\right.
	$$
	Then, we consider the moulds $sg^\p$ , $e^\p$ and $s^\p$ , with values in $\C$ and defined over the alphabet
	$\Omega = \{ 1 ; -1 \}$:
	$$	\begin{array}{l@{\hspace{1.5cm}}l@{\hspace{1.5cm}}l}
			sg^{\seq{\varepsilon}} =	\displaystyle	{	\prod	_{k = 1}
																	^n
																	\varepsilon_k
														}  ,
			&
			s^{\seq{\varepsilon}} =	\displaystyle	{	\sum	_{k = 1}
																^n
																\varepsilon_k
													}  ,
			&
			e^{\seq{\varepsilon}} = \displaystyle	{	\sum	_{k = 1}
																^n
																\varepsilon_k e^{(2k - 1) \frac{i \pi}{n}}
													}  .
		\end{array}
	$$
	Then, for all $z \in \C - \Z$ , we have:
	$$	\mathcal{T}e^{n^{[k]}}(z)
		=
		\displaystyle	{	\frac	{(-1)^{n - 1 + E ( \frac{kn + 1}{2} )} \pi^{kn}}
									{(kn)! (2 \sin (\pi z))^n}
							\sum	_{	\seq{\varepsilon}
										=
										(\varepsilon_1 ; \cdots ; \varepsilon_n) \in \Omega^n
									}
									sg^{\seq{\varepsilon}}
									(e^{\seq{\varepsilon}})^{kn}
									t_{kn,n} (s^{\seq{\varepsilon}} \pi z) \ ,
						}
	$$
	where $n^{[k]}$ represents the sequence consisting of $k$ repetitions of $n$.
\end{Property}

 \section*{Acknowledgement}
 The author would like to thank Jean Ecalle who has introduced to him multitangent functions as a family of special functions appearing in holomorphic dynamics and also to thank the referees for their valuable remarks. Without them, the paper would have been more difficult to read and two proofs would have been longer.
\section	{Definition of the multitangent functions and their first properties}
\label{definition des multitangentes et premieres proprietes}

Let us begin with a general lemma which immediately shows, if a certain condition holds, that
a mould defined as an iterated sum of holomorphic functions is a \symmetrel mould with values in the
algebra of holomorphic functions. This will give us the analytical definition of multitangent
functions, but this will also be useful to deal with the Hurwitz multizeta functions in the
sequel. In the case of multizeta values, it gives the well-known convergence criterion.

As a consequence of this lemma, we will obtain four elementary, but fundamental, properties of multitangent functions.
First of all, let us explain in details what a \symmetrel mould is.

\subsection	{Definition of a \symmetrel mould}
\label{reminder on symmetrel mould}
Let $(\Omega , \cdot)$ be an alphabet with a semi-group structure. The stuffle product of two words $P = p_1 \cdots p_r$ and $Q = q_1 \cdots q_s$ constructed over the alphabet $\Omega$ is denoted by $\stuffle$ and defined recursively by:
$$	\left \{
		\begin{array}{@{}l@{}}
			P \stuffle \varepsilon
			=
			\varepsilon \stuffle P
			=
			P \ ,
			\\
			P \stuffle Q
			=
			p_1 \big( p_2 \cdots p_r \stuffle Q \big)
			+
			q_1 \big( P \stuffle q_2 \cdots q_s \big)
			+
			(p_1 \cdot q_1) \big ( p_2 \cdots p_r \stuffle q_2 \cdots q_s \big )
			\ ,
		\end{array}
	\right.
$$
where $\varepsilon$ is the empty word. As an example, in $\text{seq}(\N)$, if $P = 1 \cdot 2$ and $Q = 3$, we have:
$P \stuffle Q = 1 \cdot 2 \cdot 3 + 1 \cdot 3 \cdot 2 + 3 \cdot 1 \cdot 2 + 1 \cdot 5 + 4 \cdot 2~.$
\\

Let us remind that the recursive definition of the stuffle product may also be:
$$	\left \{
		\begin{array}{@{}lll}
			P \stuffle \varepsilon
			&=&
			\varepsilon \stuffle P
			=
			P \ ,
			\\
			\\
			P \stuffle Q
			&=&
			\big( p_1 \cdots p_{r-1} \stuffle Q \big) p_r
			+
			\big( P \stuffle q_1 \cdots q_{s - 1} \big) q_s
			\\
			&&
			+
			\big ( p_1 \cdots p_{r - 1} \stuffle q_1 \cdots q_{s - 1} \big ) (p_r \cdot q_s)
			\ ,
		\end{array}
	\right.
$$

The multiset $sh\text{\textbf{\underline{\textit{e}}}}(\seq{\pmb{\alpha}} ; \seq{\pmb{\beta}})$, where $\seq{\pmb{\alpha}}$ and $\seq{\pmb{\beta}}$ are sequences in $\text{seq} (\Omega)$, is defined to be the set of all monomials that appear in the non-commutative polynomial $\seq{\pmb{\alpha}} \stuffle \seq{\pmb{\beta}}$, counted with their multiplicity.
\\

When $\Omega$ is an alphabet, which is also an additive semigroup and $\mathbb{A}$ an algebra, we define a \symmetrel mould $Me^\p$ to be a mould of
$\mathcal{M}_{\mathbb{A}}^\p (\Omega)$ which satisfies for all $(\seq{\pmb{\alpha}} ; \seq{\pmb{\beta}}) \in \big (\text{seq} (\Omega) \big )^2$:

$$	\left \{
			\begin{array}{l}
					\displaystyle	{	Me^{\seq{\pmb{\alpha}}}
										Me^{\seq{\pmb{\beta}}}
										=
										\sum	_{	\seq{\pmb{\gamma}}
													\in
													sh\text{\textbf{\underline{\textit{e}}}} (\seq{\pmb{\alpha}} \,  ; \,  \seq{\pmb{\beta}})
												}
												\hspace{-0.5cm}'  \hspace{0.2cm}
												Me^{\seq{\pmb{\gamma}}}
									} \ .
					\\
					Me^\emptyset = 1 \ .
			\end{array}
	\right.
$$

Here, the sum	$	\displaystyle	{	\hspace{-0.3cm}
										\sum	_{	\seq{\pmb{\gamma}}
													\in
													sh\text{\textbf{\underline{\textit{e}}}} (\seq{\pmb{\alpha}} \,  ;   \, \seq{\pmb{\beta}})
												}
												 \hspace{-0.5cm}'  \hspace{0.2cm}
												Me^{\underline{\pmb{\gamma}}}
									}
				$ is a shorthand for
$	\displaystyle	{	\sum	_{\seq{\pmb{\gamma}} \in \text{seq} (\Omega)}
								\text{mult}	\binom{\seq{\pmb{\alpha}} \,;\, \seq{\pmb{\beta}}}{\seq{\pmb{\gamma}}}
								Me^{\seq{\pmb{\gamma}}}
					}
$, where $	\text{mult}	\binom{\seq{\pmb{\alpha}} ; \seq{\pmb{\beta}}}{\seq{\pmb{\gamma}}}$ is the coefficient of the monomial $\seq{\pmb{\gamma}}$ in the product $\seq{\pmb{\alpha}} \stuffle \seq{\pmb{\beta}}$ and is equal to 
$\langle \seq{\pmb{\alpha}} \stuffle \seq{\pmb{\beta}} | \seq{\pmb{\gamma}} \rangle$~.
From now on, we also omit the prime:

$$	\displaystyle	{	Me^{\seq{\pmb{\alpha}}} Me^{\seq{\pmb{\beta}}}
						=
						\sum	_{\seq{\pmb{\gamma}} \in \text{seq} (\Omega)}
								\langle
										\seq{\pmb{\alpha}} \stuffle \seq{\pmb{\beta}} | \seq{\pmb{\gamma}}
								\rangle
								Me^{\seq{\pmb{\gamma}}}
						=
						\sum	_{\seq{\pmb{\gamma}} \in sh\text{\textbf{\underline{\textit{e}}}}(\seq{\pmb{\alpha}} , \seq{\pmb{\beta}})}
								Me^{\seq{\pmb{\gamma}}} \ .
					}
$$

For a mould, being \symmetrel imposes a strong rigidity. For example, if $(x ; y) \in \Omega^2$ and $Me^\p$ denote a \symmetrel mould, then we have necessarily:
$$	\begin{array}{lll}
			Me^{x} Me^{y}
			&=&
			Me^{x,y} + Me^{y,x} + Me^{x + y} .
			\vspace{0.1cm}
			\\
			Me^{x,y} Me^{y}
			&=&
			Me^{y, x, y} + 2Me^{x, y, y} + Me^{x + y, y} + Me^{x, 2y} \ .
	\end{array}
$$

The reader may refer to Appendix \ref{Elements de calcul moulien debut} for a summary of mould calculus.

\subsection	{A lemma on \symmetrel moulds}

This is a first version of this lemma, for classical sums, that is to say when the summation index varies from $N$ to $+ \infty$, when $N \in \N$:

\begin{Lemma}
	\label{definition_des_moules_symetrEls1}
	\textit{(Definition of symmetrel moulds, version $1$.)}
	\\
	Let $\mathcal{U}$ be an connected open subset of $\C$ on which a complex logarithm is well-defined, $(f_n)_{n \in \N}$ a sequence of non-vanishing
	holomorphic functions on $\mathcal{U}$ and $N \in \N$~.	\\
	We assume that for all compact subsets $K$ of $\mathcal{U}$, 
	$$	|| f_n ||_{\infty, K}
		\underset {n \longrightarrow + \infty} {=}
		\mathcal{O}	\left (
							\displaystyle	{	\frac	{1}	{n}	}
					\right ) \ .
	$$

	Then, for all sequences $\seq{s} \in \text{seq} (\C) - \{\emptyset\}$, of length $r$, satisfying
	\begin{equation}	\label{caracterisation}
		\left\{
			\begin{array}{l}
				\re (s_1) > 1 ,	\\
				\hspace{1cm} \vdots		\\
				\re (s_1 + \cdots + s_r) > r ,	\\
			\end{array}
		\right.
	\end{equation}
	we have:
	\begin{enumerate}
		\item	The function
			$	\begin{array}[t]{llll}
					Fe^{\seq{s}}_N :	&	\mathcal{U}	&	\longrightarrow	&	\C	\\
									&	z		&	\longmapsto		&	\displaystyle	{	\sum	_{N < n_r < \cdots < n_1 < + \infty}
																									(f_{n_1} (z))^{s_1} \cdots (f_{n_r} (z))^{s_r}
																						}
				\end{array}
			$ is well defined on $\mathcal{U}$.
		\item $Fe_N^{\seq{s}}$ is holomorphic on $\mathcal{U}$ and for all $z \in \mathcal{U}$:
			$$	\left (
						Fe_N^{\seq{s}}
				\right )
				' (z)
				=
				\displaystyle	{	\sum	_{N < n_r < \cdots < n_1 < + \infty}
											\left (
													\prod	_{i=1}
															^r
															(f_{n_i})^{s_i}
											\right )
											'
								} .
			$$
	\end{enumerate}

	Moreover, if we set $Fe_N^\emptyset = 1$, then $Fe_N^\p$ is a symmetrel mould defined on the set of sequences
	$\seq{s} \in \text{seq} (\C)$ satisfying (\ref{caracterisation}), with values in
	$\mathcal{H}(\mathcal{U})$~.
\end{Lemma}

Let us also notice that the last set, the set of sequences $\seq{s} \in \text{seq} (\C)$
satisfying (\ref{caracterisation}), is stable by stuffling
as can easily be seen by an induction, and 
$$	\seq{s} \in sh\text{\textbf{\underline{\textit{e}}}}(\seq{s}^1 ; \seq{s}^2)
	\Longrightarrow
	\left \{
			\begin{array}{l}
				l(\seq{s}) \leq l(\seq{s}_1) + l(\seq{s}_2) \ .	
				\vspace{0.1cm}
				\\
				||\seq{s}|| = ||\seq{s}^1|| + ||\seq{s}^2|| \ .
			\end{array}
	\right.
$$

The interest of this lemma is to give in one result an absolute convergence criterion for the iterated sum
as well as the \symmetrel character. So, from now on, each time we will consider a mould which
satisfies the hypothesis of this lemma and its second version, we will call it a \symmetrel mould
without further explanations.

In the following proof, we will just indicate the reason of the conditions imposed to obtain absolute
convergence of the series and the holomorphy of $Fe_N^{\seq{s}}$. Nevertheless, we will prove in detail
the \symmetrelity of $Fe_N^{\seq{s}}$ even if it is also elementary and a direct consequence of a
computation made by Michael Hoffman (see \cite{Hoffmann}, page $485$)~, since it is not so clear
for a lot of people it is the same one.

\begin{Proof}
	Points $1$ and $2$ can be proved simultaneously because the series defining $Fe_N^{\seq{s}}$
	is normally convergent on every compact subset of $\mathcal{U}$~. Thus, the classical theorem of
	Weierstrass for limit of sequences of holomorphic functions concludes the proof. Actually,
	if $K$ is a compact subset of $\mathcal{U}$~, there exists $M_K > 0$ such that for all $n \in \N$:
	$$	||f_n||_{\infty , K} \leq \displaystyle	{	\frac	{M_K}	{n + 1}	} \ .
	$$
	Besides, for $z \in K$, we can write $f_n (z) = r_n (z) e^{i \theta_n(z)}$ with 
	$r_n (z) > 0$ and $\theta_n (z) \in \rbrack - \pi ; \pi \rbrack$~. Thus:
	$	\left |
				f_n (z) ^i
		\right |
		= e^{- \theta_n (z)}
		\in \lbrack 0 ; e^\pi \rbrack
	$~.	\\
	In particular, for $s \in \C$, we obtain:	$	\displaystyle	{	\left |
																				f_n (z) ^s
																		\right |
																		\leq
																		\frac	{M_K{}^{\re s} e^{\pi \im s}}
																				{(n + 1)^{\re s}}
																	}
												$~.
	Therefore, there exists a constant $C > 0$ satisfying:
	$$	\begin{array}{@{}lll}
					\displaystyle	{	\sum	_{N < n_r < \cdots < n_1}
												\hspace{-0.2cm}
												||f_{n_1}^{s_1} \cdots f_{n_r}^{s_r}||_{\infty , K}
									}
					&\leq&
					\hspace{-0.2cm}
					\displaystyle	{	\sum	_{N < n_r < \cdots < n_1}
												\frac	{C}	{(n_1 + 1)^{\re s_1} \cdots (n_r + 1)^{\re s_r}}
									}
					\\
					\\
					&\leq&
					C \mathcal{Z}e^{\re s_1 , \cdots , \re s_r}
					<
					+ \infty \ .
				\end{array}
	$$

	We will prove the \symmetrelity of the mould $Fe_N^\p (z)$ for all $N \in \N$ by an induction process.
	To be precise, we will show the equality
	$$	\displaystyle	{	Fe_N^{\seq{s}^1} (z) Fe_N^{\seq{s}^2} (z)
							=
							\sum	_{\seq{\pmb{\gamma}} \in sh\text{\textbf{\underline{\textit{e}}}} (\seq{s}^1 ; \seq{s}^2)}
									Fe_N^{\underline{\gamma}} (z)
						} \ ,
	$$ with sequences $\seq{s}^1$ and $\seq{s}^2$ of $\text{seq} (\C)$ satisfying (\ref{caracterisation})~. The induction is over the integer
	$l(\seq{s}^1) + l(\seq{s}^2)$ .
	\\
	\\
	\hspace{0.15cm}
	Before starting\footnotemark, let us observe that, if $\seq{s} \in \text{seq} (\C)$ satisfies
	(\ref{caracterisation}), then, by definition of $Fe_N^\p$, we have:

	\footnotetext	{	Let us remind that if $\seq{s} = (s_1 , \cdots , s_r)$, the notation $\seq{s}^{\leq k}$ refers to the sequence $(s_1 , \cdots , s_k)$
						of the first $k$ terms of $\seq{s}$, while $\seq{s}^{< k}$ refers to the empty sequence when $k = 1$ or the sequence of the first
						$(k - 1)$ terms of $\seq{s}$ if $k \geq 2$~.	\\
						For this notation, see the annex on mould calculus.
					}
	$$	Fe_N^{\seq{s}}
		=
		\displaystyle	{	\sum	_{p > N}
									(f_p)^{s_r}
									Fe_p^{\seq{s}^{< r}}
						} \text{ , by setting } p = n_r \ .
	$$
		
	\noindent
	\underline{Anchor step:} Let $(\seq{u} ; \seq{v}) \in \left ( \text{seq}(\C) \right )^2$ satisfying (\ref{caracterisation}) and
	$l(\seq{u}) = l(\seq{v}) = 1$.
	\\
	\\
	Writing $\seq{u} = (u)$ and $\seq{v} = (v)$, we successively have, for $N \in \N$:
	$$	\begin{array}{@{}lll}
			Fe_N^{\seq{u}} Fe_N^{\seq{v}}
			&=&
			\displaystyle	{	\left (
										\sum	_{p > N}
												(f_p)^{u}
								\right )
								\left (
										\sum	_{q > N}
												(f_q)^{v}
								\right )
							}
			\\
			\vspace{-0.1cm}	\\
			&=&
			\displaystyle	{	\sum	_{p > q > N}
										(f_p)^{u}  (f_q)^{v}
								+
								\sum	_{p = q > N}
										(f_p)^{u}  (f_q)^{v}
								+
								\sum	_{q > p > N}
										(f_p)^{u}  (f_q)^{v}
							}
			\\
			\vspace{-0.1cm}	\\
			&=&
			\displaystyle	{	\sum	_{q > N}
										(f_q)^{v}	 Fe_q^{u} (z)
								+
								Fe_N^{u + v} (z)
								+
								\sum	_{p > N}
										(f_p)^{u}	 Fe_p^{v} (z)
							}
			\\
			\vspace{-0.1cm}	\\
			&=&
			Fe_N^{u , v} + Fe_N^{u + v} + Fe_N^{v , u}
			=   
			\displaystyle	{	\sum	_{\seq{w} \in sh\text{\textbf{\underline{\textit{e}}}} (\seq{u} ; \seq{v})}
										Fe_N^{\seq{w}} \ .
							}
		\end{array}
	$$

	\noindent
	\underline{Induction step:}	Let us suppose that there exists a positive integer $N \geq 2$ such that the result is proved for all sequences $\seq{u}$ and
 	$\seq{v}$ of $\text{seq} (\C)$	satisfying (\ref{caracterisation}) and $l(\seq{u}) + l(\seq{v}) = N$~.
	\\
	\\
	In the same way as for length $1$ and by the use of the induction hypothesis, if $\seq{u}$ and $\seq{v}$ are of length $k$ and $l$ respectively, we successively have:	\\
	\\
	$	\begin{array}{@{}ll@{}l}
			Fe_N^{\seq{u}} Fe_N^{\seq{v}}
			&=&
			\begin{array}[t]{l}
				\displaystyle	{	\hspace{-0.1cm}
									\sum	_{p > q > N}
											\hspace{-0.1cm}
											(f_p)^{u_k} 
											(f_q)^{v_l} 
											Fe_p^{{\seq{u}}^{{\leq k - 1}}}
											\hspace{-0.05cm}
											Fe_q^{{\seq{v}}^{{\leq l - 1}}}
									\hspace{-0.1cm}
									+
									\hspace{-0.3cm}
									\sum	_{n = p = q > N}
											\hspace{-0.15cm}
											(f_n)^{u_k + v_l}
											Fe_n^{{\seq{u}}^{{\leq k - 1}}}
											\hspace{-0.05cm}
											Fe_n^{{\seq{v}}^{{\leq l - 1}}}
								}
				\\
				\\
				\displaystyle	{	+
									\sum	_{q > p > N}
											(f_p)^{u_k} 
											(f_q)^{v_l} 
											Fe_p^{{\seq{u}}^{{\leq k - 1}}}
											Fe_q^{{\seq{v}}^{{\leq l - 1}}}
								}
			\end{array}
			\\
			\\
			&=&
			\begin{array}[t]{l}
				\displaystyle	{	\sum	_{q > N}
											(f_q)^{v_l} 
											Fe_q^{{\seq{v}}^{{\leq l - 1}}}
											Fe_q^{\seq{u}}
									+
									\sum	_{n > N}
											(f_n)^{u_k + v_l} 
											Fe_n^{{\seq{u}}^{{\leq k - 1}}} 
											Fe_n^{{\seq{v}}^{{\leq l - 1}}}
								}
				\\
				\\
				\displaystyle	{	+
									\sum	_{p > N}
											(f_p)^{u_k} 
											Fe_p^{{\seq{u}}^{{\leq k - 1}}}
											Fe_p^{\seq{v}}
								}
			\end{array}
			\\
			\\
			&=&
			\begin{array}[t]{l}
				\displaystyle	{	\hspace{-0.65cm}
									\sum	_{\seq{w} \in sh\text{\textbf{\underline{\textit{e}}}} (\seq{u}  ;  {\seq{v}}^{\leq l - 1})}
											\hspace{-0.2cm}
											\left (
												\displaystyle	{	\sum	_{q > N}
																			(f_q)^{v_l}  Fe_q^{\seq{w}}
																}
											\right )
									+
									\hspace{-0.5cm}
									\sum	_{\seq{w} \in sh\text{\textbf{\underline{\textit{e}}}} ({\seq{u}}^{\leq k - 1}  ;   {\seq{v}}^{\leq l - 1})}
											\hspace{-0.2cm}
											\left (
												\displaystyle	{	\sum	_{n > N}
																			(f_n)^{u_k + v_l}  Fe_n^{\seq{w}}
																}
											\right )
								}
				\\ 
				\\
				\displaystyle	{	+
									\hspace{-0.65cm}
									\sum	_{\seq{w} \in sh\text{\textbf{\underline{\textit{e}}}} ({\seq{u}}^{\leq k - 1}  ;  \seq{v})}
											\hspace{-0.2cm}
											\left (
												\displaystyle	{	\sum	_{p > N}
																			(f_p)^{u_k}  Fe_p^{\seq{w}}
																}
											\right )
								}
			\end{array}
		\end{array}
	$
	$	\begin{array}{@{}ll@{}l}
			\phantom{Fe_N^{\seq{u}} Fe_N^{\seq{v}}}
			&=&
			\begin{array}[t]{l}
				\displaystyle	{	\hspace{-0.5cm}
									\sum	_{\seq{w} \in sh\text{\textbf{\underline{\textit{e}}}} (\seq{u}  ;  {\seq{v}}^{\leq l - 1}) \cdot v_l}
											\hspace{-0.1cm}
											Fe_N^{\seq{w}}
									+
									\hspace{-0.5cm}
									\sum	_{\seq{w} \in sh\text{\textbf{\underline{\textit{e}}}} ({\seq{u}}^{\leq k - 1}  ;  {\seq{v}}^{\leq l - 1}) \cdot (u_k + v_l)}
											\hspace{-0.1cm}
											Fe_N^{\seq{w}}
									+
									\hspace{-0.5cm}
									\sum	_{\seq{w} \in sh\text{\textbf{\underline{\textit{e}}}} ({\seq{u}}^{\leq k - 1}  ;  \seq{v}) \cdot u_k}
											\hspace{-0.1cm}
											Fe_N^{\seq{w}}
								}
			\end{array}
			\\
			\\
			&=&
			\displaystyle	{	\sum	_{\seq{w} \in sh\text{\textbf{\underline{\textit{e}}}} (\seq{u}   ;   \seq{v})}
										Fe_N^{\seq{w}}
							} \ .
		\end{array}
	$
	\\
	\\
	Thus, by induction, for all sequences $\seq{s}^1$ and $\seq{s}^2$ of $\text{seq} (\C)$ satisfying (\ref{caracterisation}),
	we have:
	$$	Fe_N^{\seq{s}^1} Fe_N^{\seq{s}^2}
		=
		\displaystyle	{	\sum	_{\seq{\pmb{\gamma}} \in sh\text{\textbf{\underline{\textit{e}}}} (\seq{s}^1   ;   \seq{s}^2)}
									Fe_N^{\seq{\pmb{\gamma}}}
						} \text{ where } N \in \N \ ,
	$$
	which means that, for all $z \in \mathcal{U}$, the mould $Fe_N^\p (z)$ is a \symmetrel one.
	\qed
\end{Proof}

We obtain, as a corollary, the second version of this lemma, but for sums over all integers:

\begin{Lemma}
	\label{definition_des_moules_symetrEls2}
	\textit{(Definition of symmetrel moulds, version $2$.)}
	\\
	Let $\mathcal{U}$ be an connected open subset of $\C$ on which a complex logarithm is well-defined, $(f_n)_{n \in \N}$ a sequence of non-vanishing
	holomorphic functions on $\mathcal{U}$~.	\\
	We assume that for all compact subsets $K$ of $\mathcal{U}$, 
	$$	|| f_n ||_{\infty, K}
		\underset {n \longrightarrow \pm \infty} {=}
		\mathcal{O}	\left (
							\displaystyle	{	\frac	{1}	{|n|}	}
					\right ) \ .
	$$

	\begin{enumerate}
		\item	Then, for all sequences $\seq{s} \in \text{seq} (\C) - \{\emptyset\}$, of length $r$, satisfying
				\begin{equation}	\label{caracterisation2}
					\forall k \in \crochet{1}{r}  ,  
					\left\{
						\begin{array}{l}
							\re (s_1 + \cdots + s_k) > k ,			
							\vspace{0.1cm}
							\\
							\re (s_r + \cdots + s_{r - k + 1}) > k ,	\\
						\end{array}
					\right.
				\end{equation}
		\item[]	the function
				$	\begin{array}[t]{llll}
						Fe^{\seq{s}} :	&	\mathcal{U}	&	\longrightarrow	&	\C	\\
										&	z			&	\longmapsto		&	\displaystyle	{	\sum	_{- \infty < n_r < \cdots < n_1 < + \infty}
																											\big (f_{n_1} (z) \big )^{s_1}
																											\cdots
																											\big (f_{n_r} (z) \big )^{s_r}
																								}
					\end{array}
				$
				is well defined on $\mathcal{U}$, holomorphic on $\mathcal{U}$ and satisfy:
				$$	\forall z \in \mathcal{U}  , 
					\left (
							Fe^{\seq{s}}
					\right )
					' (z)
					=
					\displaystyle	{	\sum	_{- \infty < n_r < \cdots < n_1 < + \infty}
												\left (
														\prod	_{i=1}
																^r
																(f_{n_i} (z))^{s_i}
												\right )
												'
									}~.
				$$
		\item[2.]		Moreover, if we set $Fe^\emptyset = 1$, then $Fe^\p$ is a symmetrel mould defined on the set of sequences
						$\seq{s} \in \text{seq} (\C)$ satisfying (\ref{caracterisation2}), with values in $\mathcal{H}(\mathcal{U})$~.
	\end{enumerate}
\end{Lemma}

\begin{Proof}
	This lemma uses a mould factorization, and consequently uses mould calculus. The reader can refer to Appendix \ref{Elements de calcul moulien debut},
	\S \ref{AppendixDefinition}, \S \ref{AppendixOperations} and \ref{AppendixSymmetrelity}.
	
	\bigskip
	
	The lemma which gives the definition of symmetrel moulds, version $1$, has several consequences.
	
	\bigskip

	\noindent
	First, the mould $Fe^\p$ can be factorised:
	\begin{equation}	\label{trifactorisation generale}
		Fe^\p (z) = Fe_+^\p (z) \times Ce^\p (z) \times Fe_-^\p (z) \ ,
	\end{equation}
	where, for all $\seq{s} \in \text{seq} (\C)$ satisfying (\ref{caracterisation2}), of length $r$
	the functions $Fe_+^{\seq{s}}$ , $Ce^{\seq{s}}$ and $Fe_-^{\seq{s}}$ are defined on $\mathcal{U}$ by:
	$$	\begin{array}{lll}
			Fe_+^{\seq{s}} (z)
			&=&
			\left \{
					\begin{array}{cl}
						1
						&
						\text{, if } r = 0	\ .
						\\
						\displaystyle	{	\sum	_{0 < n_r < \cdots < n_1 < + \infty}
													\displaystyle	{	\prod	_{i = 1}
																				^r
																				\left (
																						f_{n_i} (z)
																				\right )
																				^{s_i}
																	}
										}
						&
						\text{, otherwise.}
					\end{array}
			\right.
			\vspace{0.2cm}
			\\
			Ce^{\seq{s}} (z)
			&=&
			\left \{
					\begin{array}{cl}
						1								&	\text{, if } r = 0	\ .	\\
						\left (f_0 (z) \right )^{s_1}	&	\text{, if } r = 1	\ .	\\
						0								&	\text{, otherwise.}
					\end{array}
			\right.
			\vspace{0.1cm}
			\\
			Fe_-^{\seq{s}} (z)
			&=&
			\left \{
					\begin{array}{cl}
						1
						&
						\text{, if } r = 0	\ .
						\\
						\displaystyle	{	\sum	_{- \infty < n_r < \cdots < n_1 < 0}
													\displaystyle	{	\prod	_{i = 1}
																				^r
																				\left (
																						f_{n_i} (z)
																				\right )
																				^{s_i}
																	}
										}
						&
						\text{, otherwise.}
					\end{array}
			\right.
		\end{array}
	$$
	
	One can notice that $Fe_+^\p = Fe_0^\p$ and $Fe_-^\p = \widetilde{Fe_0}^{\overset {\leftarrow} {\p}}$ correspond to the sequence of functions $\widetilde{f}_n = f_{-n}$ and a reverse sequence of positive integers denoted by $\overset {\leftarrow} {\p}$~.
	\\
	
	To prove \eqref{trifactorisation generale}, let us set
	$	Fe_{+0}^{\seq{s}} (z)
		=
		\hspace{-0.3cm}
		\displaystyle	{	\sum	_{0 \leq n_r < \cdots < n_1 < + \infty}
									 \displaystyle	{	\prod	_{i = k}
																^{l(\seq{s})}
																\left (
																		f_{n_i} (z)
																\right )
																^{s_i}
													}
						}
	$, where $z \in \mathcal{U}$ and $\seq{s} \in \text{seq} (\C)$ satisfy (\ref{caracterisation2})~.
	In the definition of $Fe_{+0}^\p (z)$, we obtain by isolating the summation index $n_r$ when it is equal to $0$:
	$$	\begin{array}{@{}lll}
			Fe_{+0}^{\seq{s}} (z)
			&=&
			\displaystyle	{	\hspace{-0.4cm}
								\sum	_{0 = n_r < n_{r - 1} < \cdots < n_1 < + \infty}
										 \displaystyle	{	\prod	_{i = k}
																	^{l(\seq{s})}
																	\left (
																			f_{n_i} (z)
																	\right )
																	^{s_i}
														}
								+
								\hspace{-0.4cm}
								\sum	_{0 < n_r < n_{r - 1} < \cdots < n_1 < + \infty}
										 \displaystyle	{	\prod	_{i = k}
																	^{l(\seq{s})}
																	\left (
																			f_{n_i} (z)
																	\right )
																	^{s_i}
														}
							}
			\\
			\\
			&=&
			\displaystyle	{	\left ( f_0 (z) \right )^{s_r}
								Fe_+^{\seq{s}^{\leq r - 1}} (z)
								+
								Fe_+^{\seq{s}} (z)
							}
			=
			\displaystyle	{	\left (
										Fe_+^{\p} (z)
										\times
										Ce^{\p} (z)
								\right )
								^{\seq{s}}
							} .
		\end{array}
	$$

	In the same way, we show that $Fe^{\p} (z) = Fe_{+0}^{\p} (z) \times Fe_-^{\p} (z)$, which implies the trifactorisation (\ref{trifactorisation generale})~.	\\

	As an example, this trifactorization give us:
	
	$$	\begin{array}{lll}
			Fe^{\omega_1}(z)			&=&	Fe_+^{\omega_1} (z) + Ce^{\omega_1} (z) + Fe_-^{\omega_1} (z)
			\ .
			\\
			\\
			Fe^{\omega_1, \omega_2} (z)	&=& Fe_+^{\omega_1, \omega_2} (z) + Fe_+^{\omega_1} (z) Ce^{\omega_2} (z) + Fe_+^{\omega_1} (z) Fe_-^{\omega_2} (z)
			\\
										&&	+ Ce^{\omega_1} (z) Fe_-^{\omega_2} (z) + Fe_-^{\omega_1, \omega_2} (z)
			\ .
		\end{array}
	$$
	\noindent
	Then, since $\seq{s} \in \text{seq}(\C)$ satisfies (\ref{caracterisation2}), $\seq{s}$ and $\overset {\leftarrow} {\seq{s}}$ satisfy (\ref{caracterisation})~. 
	The lemma of definition of symmetrel moulds, version $1$, shows us that the functions $Fe_+^{\seq{s}^{\leq k}}$ and $Fe_-^{\seq{s}^{\geq k}}$ are well defined
	and holomorphic on $\mathcal{U}$ and that their derivatives can be computed by a term by term process.

	Thus, $Fe^{\seq{s}}$ is well defined and holomorphic on $\mathcal{U}$, with a derivative which is the summation of the summand derivatives.
	\\
	\\
	Finally, according to the first version of this lemma, $Fe_+^\p$ and $Fe_-^\p$ are \symmetrel moulds, as well as $Ce^\p$~. Since the mould product of
	\symmetrel moulds defines a \symmetrel mould, we deduce that $Fe^\p$ is a \symmetrel mould for
	all $z \in \mathcal{U}$.
	\qed
\end{Proof}
	 
\subsection	{Application: definition of multitangent functions}

Let us consider $\mathcal{U} = \C - \Z$ and for $n \in \Z$, the functions
$$	\begin{array}[t]{llll}
		f_n :	&	\mathcal{U}	&	\longrightarrow	&	\C	\\
				&	z			&	\longmapsto		&	\displaystyle	{	\frac	{1}{n + z}	} \ .
	\end{array}
$$

\noindent
It is clear that, for all compact subsets $K$ of $\C - \Z$,
$$	||f_n||_{\infty, K}
	\underset {n \longrightarrow \pm \infty} {=}
	\mathcal{O} \left (  \displaystyle	{ \frac {1} {|n|} } \right ) \ .
$$
The lemma of definition of symmetrel moulds, version $2$, allows us to define a \symmetrel mould, denoted $\mathcal{T}e^\p$, defined by:
$$	\begin{array}[t]{lcll} \label{definition te}
		\mathcal{T}e^{\seq{s}} :	&	\C - \Z	&	\longrightarrow	&	\C
		\\
									&	z			&	\longmapsto	&	\displaystyle	{	\sum	_{- \infty < n_r < \cdots < n_1 < + \infty}
																								\frac	{1}
																											{(n_1 + z)^{s_1} \cdots (n_r + z)^{s_r}}
																						}  .
	\end{array}
$$

This mould, which will be called \textit{the mould of multitangent functions}, is defined, a priori, for all sequences
$$	\seq{s} \in \mathcal{S}^\star_{b,e} =	\big \{
													\seq{s} \in \text{seq}(\N^*)
													;
													s_1 \geq 2 \text{ and } s_{l(\seq{s})} \geq 2
											\big \}
$$
and takes its values in the algebra of holomorphic functions defined on $\C - \Z$~.
														
\subsection	{First properties of multitangent functions}

As a consequence of Lemma $\ref{definition_des_moules_symetrEls2}$ with a simple change of variables in the summations, we have:

\begin{Property}	\label{firstProperties}
	\begin{enumerate}
		\item The function $\mathcal{T}e^{\underline{s}}$ is well-defined on $\C - \Z$ for any
			  sequence $\seq{s} \in \mathcal{S}^\star_{b,e}$~.
		\item The function $\mathcal{T}e^{\underline{s}}$ is holomorphic on $\C - \Z$ for any
		      sequence $\seq{s} \in \mathcal{S}^\star_{b,e}$, it is a uniformly convergent
		      series on every compact subset of $\C - \Z$~ and satisfies, for all
		      $\seq{s} \in \mathcal{S}^\star_{b,e}$ and all $z \in \C - \Z$:
			\vspace{-0.2cm}
			 $$	\displaystyle	{	\frac	{\partial \mathcal{T}e^{\seq{s}}}
								{\partial z}
								(z)
							=
							- \sum	_{i = 1}
								^{l(\seq{s})}
								s_i
								\mathcal{T}e	^{	s_1 , \cdots , s_{i-1},
											s_i + 1,
											s_{i+1} , \cdots , s_{l(\seq{s})}
										} (z)
						} \ .
			$$
			\vspace{-0.4cm}
		\item For any sequence $\seq{s} \in \mathcal{S}^\star_{b,e}$ and all $z \in \C - \Z$
		      we have:
			$$	\mathcal{T}e^{\seq{s}} (-z)
				=
				(-1)^{||\seq{s}||} \mathcal{T}e^{ \overset {\leftarrow} {\seq{s}} } (z) \ .
			$$
		\item For all $z \in \C - \Z$ , $\mathcal{T}e^\p (z)$ is \symmetrel, that is,
		      for any sequence 
			$(\seq{\pmb{\alpha}} ; \seq{\pmb{\beta}}) \in (\mathcal{S}^\star_{b,e})^2$:
		\item[]	$$	\displaystyle	{	\mathcal{T}e^{\seq{\pmb{\alpha}}} (z)  \mathcal{T}e^{\seq{\pmb{\beta}}} (z)
							=
							\sum	_{	\seq{\pmb{\gamma}}
									\in
									sh\text{\textbf{\underline{\textit{e}}}}(\seq{\pmb{\alpha}} ;  \seq{\pmb{\beta}})
								}
								\mathcal{T}e^{\seq{\pmb{\gamma}}} (z)
						} \ .
			$$
	\end{enumerate}
\end{Property}

We will speak respectively of the \textit{differentiation property} \label{differentiability property} and the \textit{parity property} \label{parity property} to refer to the formula of the second point and that of the third point. Let us also notice that if
$(\seq{\pmb{\alpha}} ; \seq{\pmb{\beta}}) \in (\mathcal{S}^\star_{b,e})^2$, then
$sh\text{\textbf{\underline{\textit{e}}}}(\seq{\pmb{\alpha}} ;  \seq{\pmb{\beta}}) \subset \mathcal{S}_{b,e}^\star$~.
\section	{Reduction into monotangent functions}
\label{reduction}

The aim of this section is to show a non-trivial link between multitangent functions and multizeta values.
More precisely, we will show that all (convergent) multitangent functions can be expressed in terms of
multizeta values and monotangent functions\footnotemark. In order to do this, we will perform a partial
fraction expansion (in the variable $z$) and sum after a reorganisation of the terms.

\footnotetext	{	Let us recall that a monotangent function is a multitangent function of length $1$~.	}

Let us notice that this idea had already been mentioned by Jean Ecalle (cf \cite{Ecalle1}, p. $429$)~.

\subsection	{A partial fraction expansion}
\label{section-3-1}
Let us fix a positive integer $r$, a family of positive integers $\seq{s} = (s_i)_{1 \leq i \leq r}$
and a family of complex numbers $\seq{a} = (a_i)_{1 \leq i \leq r}$, where the $a_i$'s are pairwise
distinct. Let us also consider the rational fraction defined by
$$	F_{\seq{a}, \seq{s}} (z)
	=
	\displaystyle	{	\frac	{1}
					{	(z + a_1)^{s_1}
						\cdots
						(z + a_r)^{s_r}
					}
			} \ .
$$

We know that the partial fraction expansion of $F_{\seq{a}, \seq{s}} (z)$ can be written in the following way:
$$	F_{\seq{a}, \seq{s}} (z)
	=
	\displaystyle	{	\sum	_{i = 1}
					^r
					\sum	_{k = 0}
							^{s_i - 1}
							\displaystyle	{	\frac	{1}	{k!}
												\frac	{	\left (
																	F_{	{\seq{a}}^{\leq i - 1}
																		\cdot
																		{\seq{a}}^{\geq i + 1}
																		,
																		{\seq{s}}^{\leq i - 1}
																		\cdot
																		{\seq{s}}^{\geq i + 1}
																	}
														\right )
														^{(k)} (-a_i)
													}
													{(z + a_i)^{s_i - k}}
											}
			} \ ,
$$
where the sequences ${\seq{a}}^{\leq i - 1} \cdot {\seq{a}}^{\geq i + 1}$ and
${\seq{s}}^{\leq i - 1} \cdot {\seq{s}}^{\geq i + 1}$ denote respectively				\linebreak
$(a_1 , \cdots , a_{i - 1} , a_{i + 1} , \cdots , a_r)$ and
$(s_1 , \cdots , s_{i - 1} , s_{i + 1} , \cdots , s_r)$~.

An easy computation shows that, for all $k \in \N$, we have:

$$	\displaystyle	{	\frac	{(-1)^k}	{k!}
						F_{{\seq{a}}, {\seq{s}}}^{(k)} (z)
						=
						\sum	_{	k_1, \cdots, k_r \geq 0
									\atop
									k_1 + \cdots + k_r = k
								}
								\frac	{	\binom{s_1 + k_1 - 1}{s_1 - 1}
											\cdots
											\binom{s_r + k_r - 1}{s_r - 1}
										}
										{	(z + a_1)^{s_1 + k_1} \cdots (z + a_r)^{s_r + k_r}
										}
					} \ .
$$

Let us now introduce some notations:

$$	\begin{array}{l}
		\varepsilon_i^{\seq{s} , \seq{k}}
		=
		(-1)^{s_1 + \cdots + s_{i - 1} + s_{i + 1} + \cdots + s_r + k_1 + \cdots + k_{i - 1} + k_{i + 1} + \cdots + k_r}
		\ ,
		\\
		^i D_{{\seq{k}}}^{{\seq{s}}} ({\seq{a}})
		=
		\displaystyle	{	\left (
									\prod	_{l = 1}
											^{i - 1}
											(a_i - a_l)^{s_l + k_l}
							\right )
							\left (
									\prod	_{l = i + 1}
											^{r}
											(a_l - a_i)^{s_l + k_l}
							\right )
						} \ ,
		\\
		^i B_{{\seq{k}}}^{{\seq{s}}}
		=
		\displaystyle	{	\left (
									\prod	_{l = 1}
											^{i - 1}
											(-1)^{k_l}
							\right )
							\left (
									\prod	_{l = i + 1}
											^r
											(-1)^{s_l}
							\right )
							\left (
									\prod	_{	l = 1
												\atop
												l \neq i
											}
											^r
											\left (
													\begin{array}{@{}c@{}}
															s_l + k_l - 1	\\
															s_l - 1
													\end{array}
											\right )
							\right )
						} \ ,
	\end{array}
$$
\noindent
where the sequence $\seq{k}$ has the same length as $\seq{a}$ and an $\text{i}^{\text{th}}$ index which does
not intervene.

So, we finally have the following partial fraction expansion:
\begin{equation}	\label{partial fraction expansion}
	\displaystyle	{	F_{{\seq{a}},{\seq{s}}} (z)
						=
						\sum	_{i = 1}
								^r
								\sum	_{k = 0}
										^{s_i - 1}
										\displaystyle	{	\sum	_{	k_1 , \cdots , \widehat{k_i} , \cdots , k_r	\geq 0
																		\atop
																		k_1 + \cdots + \widehat{k_i} + \cdots + k_r	= k
																	}
																	\frac	{\varepsilon_i^{\seq{s} , \seq{k}}}
																			{(z + a_i)^{s_i - k}}
																	\frac	{{}^i B_{{\seq{k}}}^{{\seq{s}}}}
																			{^i D_{{\seq{k}}}^{{\seq{s}}} ({\seq{a}})}
														}
					} \ .
\end{equation}

The notation $\widehat{k_i}$ in the third summation means that the non negative integer $k_i$ does not
appear.

\subsection	{Expression of a multitangent function in terms of multizeta values and monotangent functions}

Plugging $(\ref{partial fraction expansion})$ in the definition of a multitangent function, we can
exchange the multiple summation (from the definition of a multitangent) with the finite summation
(from the partial fraction expansion), because of the absolute convergence, and then sum by decomposing
the multiple summation into three terms. The following are successively equal to
$\mathcal{T}e^{\seq{s}} (z)$, if $\seq{s} \in \mathcal{S}^\star_{b,e}$:\\

\noindent
$	\begin{array}{@{}l}
		\hspace{-0.2cm}
		\Bigg (
			\displaystyle	{	\sum	_{i = 1}
										^r
										\sum	_{k = 0}
												^{s_i - 1}
												\displaystyle	{	\sum	_{	k_1 , \cdots , \widehat{k_i} , \cdots , k_r	\geq 0
																				\atop
																				k_1 + \cdots + \widehat{k_i} + \cdots + k_r	= k
																			}

																			\sum	_{- \infty < n_r < \cdots < n_1 < + \infty}
																}
							}
		\Bigg )
		\Big (
			\displaystyle	{	\frac	{\varepsilon_i^{\seq{s} , \seq{k}}}
										{(z + n_i)^{s_i - k}}
								\frac	{{}^i B_{{\seq{k}}}^{{\seq{s}}}}
										{^i D_{{\seq{k}}}^{{\seq{s}}} ({\seq{n}})}
							}
		\Big )
		\\
		\hspace{-0.15cm}
		=
		\hspace{-0.15cm}
		\Bigg (
			\displaystyle	{	\sum	_{i = 1}
										^r
										\sum	_{k = 0}
												^{s_i - 1}
												\hspace{-0.05cm}
												\displaystyle	{	\sum	_{	k_1 , \cdots , \widehat{k_i} , \cdots , k_r	\geq 0
																				\atop
																				k_1 + \cdots + \widehat{k_i} + \cdots + k_r	= k
																			}
																			\hspace{-0.05cm}
																			\sum	_{n_i \in \Z}
																					\displaystyle	{	\sum	_{	(n_1 ; \cdots ; n_{i - 1}) \in \Z^{i - 1}
																													\atop
																													n_i < n_{i - 1} < \cdots < n_1
																												}
																												\hspace{-0.05cm}
																												\sum	_{	(n_{i + 1} ; \cdots ; n_r) \in \Z^{r - i}
																															\atop
																															n_r < \cdots < n_{i + 1} < n_i
																														}
																									}
																}
						}
		\Bigg )
		\Big (
			\displaystyle	{	\frac	{\varepsilon_i^{\seq{s} , \seq{k}}}
										{(z + n_i)^{s_i - k}}
								\frac	{{}^i B_{{\seq{k}}}^{{\seq{s}}}}
										{^i D_{{\seq{k}}}^{{\seq{s}}} ({\seq{n}})}
							}
		\Big )
		\\
		\hspace{-0.15cm}
		=
		\hspace{-0.15cm}
		\begin{array}[t]{l@{}l}
			\Bigg (
				\displaystyle	{	\sum	_{i = 1}
											^r
											\sum	_{k = 0}
													^{s_i - 1}
													\displaystyle	{	\sum	_{	k_1 , \cdots , \widehat{k_i} , \cdots , k_r	\geq 0
																					\atop
																					k_1 + \cdots + \widehat{k_i} + \cdots + k_r	= k
																				}
																	}
								}
			\Bigg )
			&
			\hspace{-0.1cm}
			\left (
				\displaystyle	{	{}^i B_{{\seq{k}}}^{{\seq{s}}}
									\*
									\sum	_{n_i \in \Z}
											\Bigg (
													\displaystyle	{	\frac	{\varepsilon_i^{\seq{s} , \seq{k}}}
																				{(z + n_i)^{s_i - k}}
																		\bigg [
																				\sum	_{	(n_{i + 1} ; \cdots ; n_r) \in \Z^{r - i}
																							\atop
																							- \infty < n_r < \cdots < n_{i + 1} < n_i
																						}
																						\hspace{-0.2cm}
																						\frac	{1}
																							{^i D_{{\seq{k}}^{\geq i}}^{{\seq{s}}^{\geq i}} ({\seq{n}^{\geq i}})}
																		\bigg ]
																	}
								}
			\right.
			\\
			&
			\hspace{3.2cm}
			\left.
				\displaystyle	{				\*
												\bigg [
													\sum	_{	(n_1 ; \cdots ; n_{i - 1}) \in \Z^{i - 1}
																\atop
																n_i < n_{i - 1} < \cdots < n_1 < + \infty
														}
														\frac	{1}
															{^i D_{{\seq{k}}^{\leq i}}^{{\seq{s}}^{\leq i}} ({\seq{n}^{\leq i}})}
												\bigg ]
											\Bigg )
								}
			\right ) \ .
		\end{array}
	\end{array}
$

Since we have:
$$	\begin{array}{lll}
		\displaystyle	{	\sum	_{	(n_{i + 1} ; \cdots ; n_r) \in \Z^{r - i}
							\atop
							- \infty < n_r < \cdots < n_{i + 1} < n_i
						}
						\frac	{1}
							{^i D_{{\seq{k}}^{\geq i}}^{{\seq{s}}^{\geq i}} ({\seq{n}^{\geq i}})}
				}
		&=&
		\displaystyle	{	\sum	_{	(n_{i + 1} ; \cdots ; n_r) \in \Z^{r - i}
							\atop
							- \infty < n_r < \cdots < n_{i + 1} < n_i
						}
						\prod	_{l = i + 1}
							^r
							\displaystyle	{	\frac	{1}	{(n_l - n_i)^{s_l + k_l}}	}
				}
	\end{array}
$$

\vspace{-1cm}
$$	\begin{array}{lll}
		&=&
		\displaystyle	{	\sum	_{- \infty < n_r < \cdots < n_{i + 1} < 0}
						\ \ 
						\prod	_{l = i + 1}
							^r
							\displaystyle	{	\frac	{1}	{n_l^{s_l + k_l}}	}
				}
		\\
		\\
		&=&
		(-1)^{||\seq{s}^{> i}|| + ||\seq{k}^{> i}||} \mathcal{Z}e^{s_r + k_r, \cdots, s_{i + 1} + k_{i + 1}}
		\ ,
		\\
		\\
		\displaystyle	{	\sum	_{	(n_1 ; \cdots ; n_{i - 1}) \in \Z^{i - 1}
							\atop
							n_i < n_{i - 1} < \cdots < n_1 < + \infty
						}
						\frac	{1}
							{^i D_{{\seq{k}}^{\leq i}}^{{\seq{s}}^{\leq i}} ({\seq{n}^{\leq i}})}
				}
		&=&
		(-1)^{||\seq{s}^{< i}|| + ||\seq{k}^{< i}||} \mathcal{Z}e^{s_1 + k_1, \cdots, s_{i - 1} + k_{i - 1}}
		\ ,
	\end{array}
$$
where $\seq{s}^{< i}$ and $\seq{s}^{> i}$ denote the sequences $(s_1, \cdots, s_{i - 1})$ and
$(s_{i + 1} , \cdots , s_r)$ and $||\seq{s}|| = s_1 + \cdots + s_r$~, we can conclude that:
$$	 \begin{array}{@{}l@{}l@{}l}
		\mathcal{T}e^{\seq{s}} (z)
		&\,=\,&
		\displaystyle	{	\sum	_{i = 1}
						^r
						\sum	_{k = 0}
							^{s_i - 1}
							\displaystyle	{	\hspace{-0.2cm}
										\sum	_{	k_1 , \cdots , \widehat{k_i} , \cdots , k_r	\geq 0
												\atop
												k_1 + \cdots + \widehat{k_i} + \cdots + k_r	= k
											}
											\hspace{-0.2cm}
											\sum	_{n_i \in \Z}
												{}^i B_{{\seq{k}}}^{{\seq{s}}}
												\displaystyle	{	\frac	{	\mathcal{Z}e^{s_r + k_r, \cdots, s_{i + 1} + k_{i + 1}}
																	\mathcal{Z}e^{s_1 + k_1, \cdots, s_{i - 1} + k_{i - 1}}
																}
																{(z + n_i)^{s_i - k}}
														}
									}
				}
		\\
		\\
		&\,=\,&
		\displaystyle	{	\sum	_{i = 1}
						^r
						\sum	_{k = 0}
							^{s_i - 1}
							\displaystyle	{	\hspace{-0.2cm}
										\sum	_{	k_1 , \cdots , \widehat{k_i} , \cdots , k_r	\geq 0
												\atop
												k_1 + \cdots + \widehat{k_i} + \cdots + k_r	= k
											}
											\hspace{-0.6cm}
											{}^i B_{{\seq{k}}}^{{\seq{s}}}
											\displaystyle	{	\mathcal{Z}e^{s_r + k_r, \cdots, s_{i + 1} + k_{i + 1}}
														\mathcal{Z}e^{s_1 + k_1, \cdots, s_{i - 1} + k_{i - 1}}
														\mathcal{T}e^{s_i - k}
														\ .
													}
									}
				}
	 \end{array}
$$

\noindent
Consequently, we have the following relation:
$$	\mathcal{T}e^{\seq{s}} (z)
	=
	\displaystyle	{	\sum	_{i = 1}
								^r
								\sum	_{k = 0}
										^{s_i - 1}
										\mathcal{Z}_{i,k}^{\seq{s}}
										\mathcal{T}e^{s_i - k} (z)
					} \ ,
$$
where
\begin{equation}	\label{Ziks}
	\displaystyle	{	\mathcal{Z}_{i,k}^{\seq{s}}	=	\sum	_{	k_1 , \cdots , \widehat{k_i} , \cdots , k_r	\geq 0
																	\atop
																	k_1 + \cdots + \widehat{k_i} + \cdots + k_r		= k
																}
																{}^i B_{{\seq{k}}}^{{\seq{s}}}
																\mathcal{Z}e^{s_r + k_r, \cdots , s_{i + 1} + k_{i + 1}}
																\mathcal{Z}e^{s_1 + k_1, \cdots , s_{i - 1} + k_{i - 1}}
					} \ .
\end{equation}

From $s_1 , s_r \geq 2$, we can observe that only convergent multizeta values appear in this computation. Moreover, this gives also an expression of $\mathcal{T}e^{\seq{s}}$ in terms of integer coefficients (since ${}^i B_{{\seq{k}}}^{{\seq{s}}} \in \Z$) and multizeta values, since the product of two multizeta values can be expressed as a linear combination of multizeta values from the \symmetrelity of $\mathcal{Z}e^\p$.
\\

The divergent monotangent
$\mathcal{T}e^1 : z \longmapsto \displaystyle	{	\frac	{\pi}	{\tan(\pi z)}	}$
seems to appear in this relation. Nevertheless, the $\mathcal{T}e^1$ coefficient is necessarily
null. Indeed, it is not difficult to see (see \S \ref{caractere exponentiellement plat}) that all
(convergent) multitangent function tends exponentially to $0$ when $\im z \longrightarrow + \infty$~.
\label{annonce caractere exp plat} Let us notice that this computation looks like if some divergent sums
appear, but it is not the case if we perform this computation using a multidimensionnal Eisenstein
proccess generalizing the one described in \cite{Weil} for a simple sum. Consequently, we obtain:

\begin{Theorem}
	\label{reduction en monotangente}
	\textit{(Reduction into monotangent functions, version $1$)}	\\
	For all sequence $\seq{s} \in \mathcal{S}^\star_{b,e}$, we have for all $z \in \C - \Z$:
	$$	\mathcal{T}e^{\seq{s}} (z)
		=
		\displaystyle	{	\sum	_{i = 1}
									^r
									\sum	_{k = 2}
											^{s_i}
											\mathcal{Z}_{i,s_i - k}^{\seq{s}}
											\mathcal{T}e^k (z)
						}
		=
		\displaystyle	{	\sum	_{k = 2}
									^{\text{max}(s_1, \cdots, s_r)}
									\Big (
											\sum	_{	i \in \crochet{1}{r}
														\atop
														s_i \geq k
													}
													\mathcal{Z}_{i,s_i - k}^{\seq{s}}
									\Big )
									\mathcal{T}e^k (z)
						} \ .
	$$
\end{Theorem}

\subsection	{Tables of convergent multitangent functions}

With a suitable computer algebra software, we can easily generate a table of multitangent functions up to a fixed weight.
Different tables can be computed:
\begin{enumerate}
	\item those given by the previous theorem ;
	\item those obtained from the first ones, as soon as we have downloaded a table of values
		of the multizeta values expressed in terms of special multizeta values
		(see \cite{table de MZV} for this purpose)~;
	\item those obtained from the first ones, as soon as we have downloaded a table of numerical values of the multizeta values (see \cite {BBBL} or \cite{Crandall} for this purpose)~;
	\item those obtained from the first ones, after a linearization of products of multizeta values (the choice of linearization by \symmetrelity is more natural in this context than using the symmetr\textbf{\textit {\underline {a}}}lity)~.
\end{enumerate}

Table \ref{table1} of appendix \ref{Tables}, contains some examples of such tables. Some boxes in it are empty, which means the expression is the same as in the previous column.
Let us immediately remark that there are a lot of $\Q$-linear relations between multitangents and none of them is trivial. Here are two of them which are easy to state, but the second one is still quite mysterious:
\begin{equation}
	\mathcal{T}e^{2, 1, 2}								=	0	\ .
\end{equation}

\begin{equation}
	3 \mathcal{T}e^{2, 2, 2} + 2 \mathcal{T}e^{3, 3}	=	0	\ .
\end{equation}

We will study this in detail in Section \ref{linear relations between mtgfs}.

\input {table1.sty}

\subsection	{Linear independence of monotangent functions}

Now, we will give a fundamental lemma which will be used here and there repeatedly in this article.
It will be a little incursion in the world of the arithmetic of multitangent functions, a quite obscure
world.

\begin{Lemma}	\label {linear independence of monotangent functions}
	The monotangent functions are $\C$-linearly independent.
\end{Lemma}	

We give a proof based on the differentiation property of multitangent functions. 
Many different proofs are possible, for instance using the Fourier coefficients of
monotangent functions or by looking at the poles of monotangent functions, etc.

\begin{Proof}
	Let us suppose the familly $(\mathcal{T}e^{n})_{n \in \N^*}$ is not $\C$-linearly independant.

	So, there would exist an integer $r \geq 1$, a $r$-tuple of integers $(n_1 ; \cdots ; n_r)$ satisfying $0 < n_1 < \cdots < n_r$ and
	$(\lambda_{n_1} ; \cdots ; \lambda_{n_r}) \in \C^r$ such that:
	$$	\displaystyle	{	\sum	_{k = 1}
									^r
									\lambda_{n_k} \mathcal{T}e^{n_k}
							=
							0
						} \ .
	$$
	
	Using the differentiation property of multitangent functions, we would obtain:
	$$	\displaystyle	{	\sum	_{k = 1}
									^r
									\frac	{(-1)^{n_k - 1} \lambda_{n_k}}
											{(n_k - 2)!}
									\frac	{\partial^{n_k - 1} \mathcal{T}e^1}
											{\partial z^{n_k - 1}}
							=
							0
						} \ .
	$$
	
	So, $\mathcal{T}e^1$ would satisfy a linear differential equation with constant coefficients, and therefore could be written as a
	$\C$-linear combination of exponential polynomials. This would allow us to obtain an analytic continuation over all $\C$ of the cotangent function.

	Because such an analytic continuation is impossible, we have proved that monotangent functions are $\C$-free.
	\qed
\end{Proof}
		
Since we have just seen that are many linear relations between multitangent functions, we know that this lemma can not be extended to multitangent
functions. In fact, since we know the Eisenstein relation $ \mathcal{T}e^{2} (z) \mathcal{T}e^1 (z) = \mathcal{T}e^3(z) $, we can affirm that monotangent functions are not algebraically independent, even if we restrict to convergent monotangent functions:
$$	2
	\Big (
			\mathcal{T}e^3
	\Big )	^2
	=
	3 \mathcal{T}e^2 \mathcal{T}e^4
	-
	\Big (
			\mathcal{T}e^2
	\Big )	^3
	\ .
$$

\subsection	{A first approach to algebraic structure of $\mathcal{M}TGF$}

Recall that we have denoted by $\mathcal{M}TGF_{CV,p}$ the $\Q$-algebra spanned by all the functions
$\mathcal{T}e^{\seq{s}}$ , with sequences $\seq{s} \in \mathcal{S}^\star_{b,e}$ of weight $p \in \N$~,
and $\mathcal{M}ZV_{CV,p}$ the $\Q$-algebra spanned by all the numbers $\mathcal{Z}e^{\seq{s}}$ , with
sequences $\seq{s} \in \mathcal{S}^\star_b$ of weight $p$~.

So, the reduction into monotangent functions, together with the previous lemma, yields the following corollary\footnotemark:

\footnotetext	{	The notation $E \cdot \alpha$ denotes the set $\{e \cdot \alpha ; e \in E\}$~.
				}

\begin{Corollary} \label{structure of the algebra MTGF, version 1}
	For all $p \geq 2$ ,	$	\displaystyle	{	\mathcal{M}TGF_{CV , p} \subseteq \bigoplus	_{k = 2}
																								^{p}
																								\mathcal{M}ZV_{CV , p - k} \cdot \mathcal{T}e^k
												} \ .
							$
\end{Corollary}

\begin{Proof}
	The reduction process gives:
	$$	\displaystyle	{	\mathcal{M}TGF_{CV , p} \subseteq \sum	_{k = 2}
																	^{p}
																	\mathcal{M}ZV_{CV , p - k} \cdot \mathcal{T}e^k
						} \ .
	$$
	But, from the previous lemma, it is quite clear that:
	$$	\displaystyle	{	\sum	_{k = 2}
									^{p}
									\mathcal{M}ZV_{CV , p - k} \cdot \mathcal{T}e^k
							=
							\bigoplus	_{k = 2}
										^{p}
										\mathcal{M}ZV_{CV , p - k} \cdot \mathcal{T}e^k
						}
						\ .
	$$
	\qed
\end{Proof}

\subsection	{Consequences}

An easy consequence of the reduction in monotangent functions is that each property on multitangent functions will have
implications on multizeta values. The process will often be like this: we express the property we
are studying in terms of multitangent functions, then we reduce all the multitangents into
monotangent functions and finally use the $\C$-linear independence of the monotangent functions
to conclude something on multizeta values.
\\

The following diagram will evolve throughout the article to explain how the multitangent functions are linked to the multizeta values. 

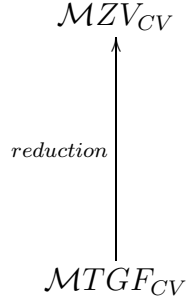
\begin{figure}[h]
	\centering
	$	\xymatrix	{	\mathcal{M}ZV_{CV}
						\\
						\\
						\\
						\mathcal{M}TGF_{CV}
						\ar[uuu]	^{reduction}
					}
	$
	\caption	{Links between multizeta values and multitangent functions}
	\label{figure une de liens entre MZV et MTGF}
\end{figure}

To illustrate this idea, let us show how a computation on multitangent functions, which will be done
in \S \ref{calculs explicite}, gives us a computation on multizeta values. For this purpose, let us consider the following formal power series:
$$	\begin{array}{lllllll}
		\mathcal{Z}_2		&=&	\displaystyle	{	\sum	_{p \geq 0}
															\mathcal{Z}e^{2^{[p]}} X^p
												}
		&,&
		\mathcal{T}_2 (z)	&=&	\displaystyle	{	\sum	_{p \geq 0}
															\mathcal{T}\!e^{2^{[p]}} (z) X^p
												} \ .
	\end{array}
$$
Let us remind that in this definition, the sequence $2^{[p]}$ is defined to be $(\underbrace{2 ; \cdots ; 2}_{p \text{ times}})$ (See Appendix \ref{Elements de calcul moulien debut} , \S \ref {AppendixNotations}).

\noindent
Property \ref{Property2Intro} will show that, in $\mathcal{M}TGF_{CV}\!\!\left[\!\!\left[\sqrt{X}\right]\!\!\right]$, we have:

\begin{equation}	\label{T21}
		\mathcal{T}_2 (z)
		=
		\displaystyle	{	1 + \sum	_{k \geq 1}
										\frac	{2^{2k - 1} \pi^{2k - 2}}	{(2k)!}
										X^k \mathcal{T}\!e^2 (z)
						}
		=
		\displaystyle	{	1 + \frac	{\ch (2 \pi \sqrt{X}) - 1}	{2\pi^2} \mathcal{T}\!e^2 (z) }\ .
\end{equation}

\noindent
On the other hand, the reduction into monotangent functions implies
\begin{equation}	\label{T22}
	\mathcal{T}_2 (z) = 1 + X {\mathcal{Z}_2}^2 \mathcal{T}\!e^2 (z) \ ,
\end{equation}
because we have, for all $p \in \N$:
$	\displaystyle	{	\mathcal{T}\!e^{2^{[p]}}	=	\sum	_{i = 1}
																^p
																\mathcal{Z}e^{2^{[p - i]}} \mathcal{Z}e^{2^{[i - 1]}}
																\mathcal{T}\!e^2 (z)  \ .
					}
$

\noindent
From the last two equations, we therefore obtain:
$$	\displaystyle	{	\mathcal{Z}_2 	= 	\frac	{\sh^2 (\pi \sqrt{X})}	{\pi \sqrt{X}}	} \ , 	$$
that is:
$$	\forall n \in \N \ , \ \mathcal{Z}e^{2^{[n]}} = \displaystyle	{	\frac	{\pi^{2n}}	{(2n + 1) !}	}\ .	$$
\section	{Projection onto multitangent functions}
\label{projection}
\subsection	{A second approach to the algebraic structure of $\mathcal{M}TGF$}

We have just seen that for all $p \geq 2$:
$$	\displaystyle	{	\mathcal{M}TGF_{CV , p} \subseteq \bigoplus	_{k = 2}
										^{p}
										\mathcal{M}ZV_{CV , p - k} \cdot \mathcal{T}e^k
			}\ .
$$
The table of convergent multitangents that we have established up to weight $18$ shows that
the inclusion is in fact an equality for $p \leq 18$. This is why we conjecture that the equality
holds for all $p \in \N$~.

\begin{Conjecture}	\label{structure of the algebra MTGF, version 2}
	For all $p \geq 2$ ,
	$	\displaystyle	{	\mathcal{M}TGF_{CV , p} = \bigoplus	_{k = 2}
										^{p}
										\mathcal{M}ZV_{CV , p - k} \cdot \mathcal{T}e^k
				} \ .
	$
\end{Conjecture}

\subsection	{A structure of $\mathcal{M}ZV$-module}

If Conjecture \ref{structure of the algebra MTGF, version 2} is true, then
$\mathcal{M}ZV_{CV,p} \cdot \mathcal{T}e^q$ can be seen as a subset of $\mathcal{M}TGF_{CV, p + q}$
for all integers $(p ; q) \in \N^2$~. This leads us to the following conjecture:

\begin{Conjecture}	\label{MZV-module structure}
	$1$. For all sequences $\seq{\pmb{\sigma}} \in \mathcal{S}_b^\star$ and  $\seq{s} \in \mathcal{S}^\star_{b,e}$ , we have:
	$$\mathcal{Z}e^{\seq{\pmb{\sigma}}} \mathcal{T}e^{\seq{s}} \in \mathcal{M}TGF_{CV , ||\seq{\pmb{\sigma}}|| + ||\seq{s}||} \ .$$
	$2$. $\mathcal{M}TGF_{CV}$ is a $\mathcal{M}ZV_{CV}$-module.
\end{Conjecture}

Of course, the second point is a direct consequence of the first one, but it is the one that
interests us in theory ; in practice, we will favour the first one because of the weight
homogeneity. What is important is the links between the conjectures
\ref{structure of the algebra MTGF, version 2} and \ref{MZV-module structure}:

\begin{Property} \label{prop-proj-1}
	Conjecture \ref{structure of the algebra MTGF, version 2} is equivalent to Conjecture \ref{MZV-module structure}.
\end{Property}

\begin{Proof}
	$1$. Let us suppose that Conjecture \ref{structure of the algebra MTGF, version 2} holds.
	
	Thus, from the reduction into monotangent functions and the \symmetrelity of the mould $\mathcal{Z}e^\p$, it follows, for $(\seq{\pmb{\sigma}} ; \seq{s}) \in \mathcal{S}_b^\star \times \mathcal{S}^\star_{b,e}$, that
	$$	\begin{array}{lll}
			\mathcal{Z}e^{\seq{\pmb{\sigma}}} \mathcal{T}e^{\seq{s}}
			&=&
			\hspace{-0.6cm}
			\displaystyle	{	\sum	_{k = 2}
							^{\max(s_1 ; \cdots ; s_r)}
							\left (
								\displaystyle	{	\sum	_{i}
												c_i
												\mathcal{Z}e^{\seq{s}_{i,k}}
												\mathcal{Z}e^{\seq{\pmb{\sigma}}}
										}
							\right )
							\mathcal{T}e^{k}
					} , \text{where }	\left \{
									\begin{array}{@{}l}
										\seq{s}_{i,k} \in \mathcal{S}_b^\star \ ,
										\\
										||\seq{s}_{i,k}|| \!=\! ||\seq{s}|| - k \ ,
									\end{array}
								\right.
		\end{array}
	$$
	$$	\begin{array}{lll}
			\phantom{\mathcal{Z}e^{\seq{\pmb{\sigma}}} \mathcal{T}e^{\seq{s}}}
			&=&
			\hspace{-0.6cm}
			\displaystyle	{	\sum	_{k = 2}
							^{\max(s_1 ; \cdots ; s_r)}
							\left (
								\displaystyle	{	\sum	_{i}
												c_i' \mathcal{Z}e^{\seq{\pmb{\sigma}}_{i,k}}
										}
							\right )
							\mathcal{T}e^{k}
					} , \text{where }	\left \{
									\begin{array}{@{}l}
										\seq{\pmb{\sigma}}_{i,k} \in \mathcal{S}_b^\star \ .
										\\
										||\seq{\sigma}_{i,k}|| \!=\! ||\seq{s}|| + ||\sigma|| - k \ .
									\end{array}
								\right.
		\end{array}
	$$
	
	According to Conjecture~\ref{structure of the algebra MTGF, version 2}, we are now able to write
	each term of the form $\mathcal{Z}e^{\seq{\pmb{\sigma}}_{i,k}} \mathcal{T}e^k$ in
	$\mathcal{M}TGF_{CV , ||\seq{\pmb{\sigma}}|| + ||\seq{s}||}$ that is,
	$\mathcal{Z}e^{\seq{\pmb{\sigma}}} \mathcal{T}e^{\seq{s}} \in \mathcal{M}TGF_{CV , ||\seq{\pmb{\sigma}}|| + ||\seq{s}||}$~. This concludes that $\mathcal{M}TGF_{CV}$ is a $\mathcal{M}ZV_{CV}$-module.
	\\
	\\
	$2$. According to Conjecture~\ref{MZV-module structure}, for all sequences
	$\seq{s} \in \mathcal{S}^\star_b$ and all integer $k \geq 2$, we are able to express
	$\mathcal{Z}e^{\seq{s}} \mathcal{T}e^k$ in $\mathcal{M}TGF_{CV , ||\seq{s}|| + k}$~. 
	In other words, for all $p \geq 2$:
	$$	\displaystyle	{	\bigoplus	_{k = 2}
							^{p}
							\mathcal{M}ZV_{CV , p - k} \cdot \mathcal{T}e^k
					\subseteq
					\mathcal{M}TGF_{CV, p}
				} \ .
	$$
	We now conclude to the equality	$$	\displaystyle	{	\mathcal{M}TGF_{CV, p}
									=
									\bigoplus	_{k = 2}
											^{p}
											\mathcal{M}ZV_{CV , p - k} \cdot \mathcal{T}e^k
								}
					$$ for all $p \geq 2$ from Corollary \ref{structure of the algebra MTGF, version 1}
	\qed
\end{Proof}

The first point of Conjecture \ref{MZV-module structure} can be a bit reduced. We will restrict to the following statement:

\begin{Conjecture}	\label{new conjecture on projection}
	For all sequences $\seq{\pmb{\sigma}} \in \mathcal{S}_b^\star$ , 
	$\mathcal{Z}e^{\seq{\pmb{\sigma}}} \mathcal{T}e^2 \in \mathcal{M}TGF_{CV,||\seq{\pmb{\sigma}}|| + 2}$ .
\end{Conjecture}

From the differentiation property, it is easy to see that we have not lost information.

\begin{Property}\label{prop-proj-2}
	Conjecture \ref{new conjecture on projection} is equivalent to Conjectures \ref{structure of the algebra MTGF, version 2} and \ref{MZV-module structure}.
\end{Property}

\begin{Proof}
	It is sufficient to show that Conjecture \ref{new conjecture on projection} implies the first point of
	Conjecture \ref{MZV-module structure}.	\\
	Let us assume that Conjecture \ref{new conjecture on projection} holds. We are able
	to express each expression of the type $\mathcal{Z}e^{\seq{\pmb{\sigma}}} \mathcal{T}e^2$ in
	$\mathcal{M}TGF_{CV,||\seq{\pmb{\sigma}}|| + 2}$ . By differentiation, we explicitly find an expression of
	$\mathcal{Z}e^{\seq{\pmb{\sigma}}} \mathcal{T}e^k$ in $\mathcal{M}TGF_{CV,||\seq{\pmb{\sigma}}|| + k}$
	for all $k \geq 2$~. Using the reduction into monotangent functions, we have proved the first point of
	Conjecture \ref{MZV-module structure}, which is the desired conclusion. 
	\qed
\end{Proof}

\subsection	{About the projection of a multizeta value onto multitangent functions}

For $k \geq 2$, an explicit expression of
$\mathcal{Z}e^{\seq{\pmb{\sigma}}} \mathcal{T}e^k \in \mathcal{M}TGF_{CV,||\seq{\pmb{\sigma}}|| + k}$
will be called a projection of the multizeta $\mathcal{Z}e^{\seq{\pmb{\sigma}}}$ onto the space of
multitangent functions, or a projection onto multitangents to shorten the terminology.

\subsubsection	{A converse of the reduction into monotangent functions}

The relations of projections onto multitangent functions can be considered as a converse of the reduction into monotangent functions in the following sense:
according to Conjecture \ref{MZV-module structure}, each property that will be proven on multizeta values will have implications on multitangent functions.

The process will often be like this: we express the fact under consideration in terms of multizeta functions, then we multiply the different relations by a monotangent function and finally project all of these onto multitangent functions to conclude something on multitangent functions.
\\

The diagram on the figure \ref{figure deux de liens entre MZV et MTGF} completes the figure \ref{figure une de liens entre MZV et MTGF}. An arrow indicates a link between two algebras while an arrow in dotted lines indicates a hypothetical link.

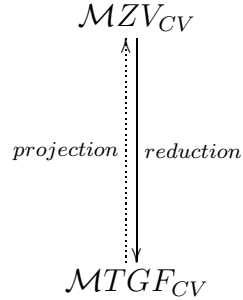
\begin{figure}[h]
	\centering
	$	\xymatrix	{	\mathcal{M}ZV_{CV}
						\ar[ddd]				^{reduction}
						\\
						\\
						\\
						\mathcal{M}TGF_{CV}
						\ar @{.>} @<4pt> [uuu]	^{projection}
					}
	$
	\caption	{Links between multizeta values and multitangent functions}
	\label{figure deux de liens entre MZV et MTGF}
\end{figure}

\subsubsection{How to find a projection onto multitangents in practice?}

The idea is to proceed by induction on the weight of $\seq{\pmb{\sigma}}$. Let us suppose that we know how to express $\mathcal{Z}e^{\seq{s}} \mathcal{T}e^{2}$ in $\mathcal{M}TGF_{CV,p + 2}$ for all sequences $\seq{s}$ of weight $p < ||\seq{\pmb{\sigma}}||$~. According to the differentiation property, we know how to express $\mathcal{Z}e^{\seq{s}} \mathcal{T}e^{||\seq{\pmb{\sigma}}|| - ||\seq{s}|| + 2}$ in $\mathcal{M}TGF_{CV,||\seq{\pmb{\sigma}}|| + 2}$~.

We write the reduction into monotangent functions of $\mathcal{T}e^{\seq{s}}$ , for all sequences $\seq{s} \in \mathcal{S}^\star_{b,e}$ of weight $||\seq{\pmb{\sigma}}|| + 2$, as:
$$	\Big (
			\sum	_{i \in E(\seq{s})}
					\mathcal{Z}e^{\seq{s}^i}
	\Big )
	\mathcal{T}e^{2}
	+
	\cdots
	\ .
$$

Here, the set $E(\seq{s})$ is finite and has only sequences of $\mathcal{S}_b^\star$ of weight $||\seq{\pmb{\sigma}}||$ ; the dotted stand for some elements of $\mathcal{M}TGF_{CV,||\seq{\pmb{\sigma}}|| + 2}$~. In order to express $\mathcal{Z}e^{\seq{\pmb{\sigma}}} \mathcal{T}e^2$ in $\mathcal{M}TGF_{CV,||\seq{\pmb{\sigma}}|| + 2}$~, the idea is to find out a linear relation with rational coefficients between
$	\displaystyle	{	\sum	_{i \in E(\seq{s})}
								\mathcal{Z}e^{\seq{s}^i}
					}
$ which is equal to $\mathcal{Z}e^{\seq{\pmb{\sigma}}}$ .

\subsubsection{Some examples}

Just before giving some examples, let us introduce some notations.
We denote by $\mathcal{M}TGF_2 = Vect_\Q (\mathcal{T}e^{\seq{s}})	_{	\seq{s} \in \text{seq} (\N_2)	}$, the vector space spanned by multitangents with valuation at least $2$ (in general, the valuation of a sequence is the smallest integer composing this sequence ; here this means that all the $s_i$'s are greater or equal to $2$) and by
$\mathcal{M}TGF_{2,p} = Vect_\Q (\mathcal{T}e^{\seq{s}})	_{	\seq{s} \in \text{seq} (\N_2)
																\atop
																||\seq{s}|| = p
															}
$ the subspace of convergent multitangent functions with valuation at least $2$ and weight $p$, where $\N_2 = \{n \in \N \, , \, n \geq 2\}$~.

As an example of what we have explained in the previous section, let us express $\mathcal{Z}e^{\seq{s}} \mathcal{T}e^{2}$ in
$\mathcal{M}TGF_{2,||\seq{s}|| + 2}$, for all sequences satisfying $||\seq{s}|| \leq 5$:
\\

The table \ref{table1} gives us 
$	\mathcal{Z}e^2 \mathcal{T}e^{2}
		=
		\displaystyle	{	\frac	{1}	{2}
							\mathcal{T}e^{2,2}
						}
$ .
\\

Moreover, we know that all multizeta values of weight $3$ and $4$ can be respectively expressed in terms of $\mathcal{Z}e^3$ and $(\mathcal{Z}e^2)^2$. Hence, we only have to consider the quantities $\mathcal{Z}e^{3} \mathcal{T}e^{2}$ and $(\mathcal{Z}e^{2})^2 \mathcal{T}e^{2}$~.

The table \ref{table1} gives us the reduction into monotangent functions of weight $5$ and $6$. We can, for example, choose the expressions:
$$	\begin{array}{rcl}
		\mathcal{Z}e^3 \mathcal{T}e^{2}
		&=&
		\displaystyle	{	\frac	{1}	{6}
							\mathcal{T}e^{3,2}
							-
							\frac	{1}	{6}
							\mathcal{T}e^{2,3}
						}  .
		\\
		(\mathcal{Z}e^2)^2 \mathcal{T}e^{2}
		&=&
		\displaystyle	{	-\frac	{5}	{12}
							\mathcal{T}e^{3,3}
						}  .
	\end{array}
$$

The case $||\seq{s}|| = 5$ is a bit more complicated. Indeed, a classical conjecture is that the
$\Q$-vector space spanned by multizeta values of weight $5$ is a $2$-dimensional vector space and
one of its bases is $(\mathcal{Z}e^5 ; \mathcal{Z}e^{2}\mathcal{Z}e^{3})$. What is certain, is that
its dimension is bounded by $2$~. Therefore, it is sufficient to express $\mathcal{Z}e^5 \mathcal{T}e^2$
and $\mathcal{Z}e^{2}\mathcal{Z}e^{3}\mathcal{T}e^2$ in $\mathcal{M}TGF_{2,7}$~. An easy computation
based on reduction into monotangent functions gives us the following projections:																 
$$	\begin{array}{rcl}
		\mathcal{Z}e^5 \mathcal{T}e^{2}
		&=&
		\displaystyle	{	\frac	{1}	{30}
							\left (
									\mathcal{T}e^{2,5}
									+
									2\mathcal{T}e^{3,4}
									-
									2
									\mathcal{T}e^{4,3}
									-
									\mathcal{T}e^{5,2}
							\right )
						} \ .
	\end{array}
$$

\begin{equation*}
	\begin{array}{rcl}
		\mathcal{Z}e^2 \mathcal{Z}e^3 \mathcal{T}e^{2}
		&=&
		\displaystyle	{	\frac	{1}	{12}
							\mathcal{T}e^{3,2,2}
							-
							\frac	{1}	{12}
							\mathcal{T}e^{2,2,3}
						}
		\\
		&&
		\displaystyle	{	+
							\frac	{1}	{24}
							\mathcal{T}e^{2,5}
							+
							\frac	{1}	{12}
							\mathcal{T}e^{3,4}
							-
							\frac	{1}	{12}
							\mathcal{T}e^{4,3}
							-
							\frac	{1}	{24}
							\mathcal{T}e^{5,2}
						} \ .
	\end{array}
\end{equation*}

To complete our computation, we only have to use the exact expressions of multizeta values in terms of those just considered (see for instance \cite{table de MZV})~. Table \ref{quelques expressions de projections de multizetas} gives us the complete table of projections onto multitangents of all the multizeta values of weight at most $5$~.

\input{table2.sty}

\subsubsection{An abstract formalization of the method}
\label{Formalisation abstraite de la methode.}

Recall that all multizeta values of weight $n$ can be expressed as $\Q$-linear combinations of a finite number of them, which are called the irreducible multizeta values. Let us denote by $c_n$ the number of irreducibles of weight~$n$.

A theorem proved by Goncharov and Terasoma shows that if $(d_n)_{n \in \N}$ is the sequence defined by
$$	\left \{
			\begin{array}{l}
					d_1 = 0 \ ,\\
					d_2 = d_3 = 1	\ ,\\
					\forall n \in \N^*, d_{n + 3} = d_{n + 1} + d_n\ ,
			\end{array}
	\right.
$$ then we have: $c_n \leq d_n$ for all $n \in \N$~.
A conjecture due to Zagier\footnotemark\, asserts that $(c_n)_{n \in \N}$ satisfies the same recurrence relation as $(d_n)_{n \in \N}$, that is, $c_n = d_n$ for all $n \in \N$~.

\footnotetext	{	Concerning this well-known conjecture, we refer the reader to the works of P. Deligne and A. B. Goncharov
					(see \cite{Deligne-Goncharov}), of T. Terasoma (see \cite{Terasoma}) as well as the recent works of F. Brown (see \cite{Brown1} and
					\cite{Brown2})~.
					\cite{Brown2})~.
				}

Let $\seq{\pmb{\sigma}} \in \mathcal{S}_b^\star$ be of weight $p$. Our objective is to express $\mathcal{Z}e^{\seq{\pmb{\sigma}}} \mathcal{T}e^2$ in $\mathcal{M}TGF_{CV, p + 2}$~.
\\

For this, first begin to write all the reductions into monotangents of weight $p + 2$.
Then, isolate from the same side of the sign $=$ the components $\mathcal{T}e^2$~. Now, in all these expressions, we
can express each multizeta value in terms of the $c_p$ corresponding irreducibles. Since each term
$\mathcal{Z}e^{\seq{\pmb{\sigma}}} \mathcal{T}e^k$, $k \geq 3$, which appears in a reduction into monotangents can be
expressed by induction in $\mathcal{M}TGF_{CV, p + 2}$ according to the differentiation property, we will have written
some $\Q$-linear equations with a left-hand side composed of terms of the form
$\mathcal{Z}e^{\seq{\pmb{\sigma}}} \mathcal{T}e^2$ and a right-hand side expressed in $\mathcal{M}TGF_{CV, p + 2}$~.
We will see in the next section that only a subsystem of this one may be useful: it will be the system with equations
coming from the reduction into monotangents of multitangent functions with valuation at least $2$.
\\

As an example, let us do it for weight~$4$. Each multizeta value of weight $4$ can be written in term of $\mathcal{Z}e^4$, so we obtain the system:
$$	\begin{array}{lll}
		\phantom{-}4 \mathcal{Z}e^4 \mathcal{T}e^2
		&=&
		\mathcal{T}e^{2,4} + 2 \mathcal{Z}e^3 \mathcal{T}e^3 - \mathcal{Z}e^2 \mathcal{T}e^4  .
		\\
		
		-6 \mathcal{Z}e^4 \mathcal{T}e^2
		&=&
		\mathcal{T}e^{3,3} .
		\\

		\phantom{-}4 \mathcal{Z}e^4 \mathcal{T}e^2
		&=&
		\mathcal{T}e^{4,2} - 2 \mathcal{Z}e^3 \mathcal{T}e^3 - \mathcal{Z}e^2 \mathcal{T}e^4  .
		\\
		
		\phantom{-}4 \mathcal{Z}e^4 \mathcal{T}e^2
		&=&
		\mathcal{T}e^{2,4} + 2 \mathcal{Z}e^3 \mathcal{T}e^3 - \mathcal{Z}e^2 \mathcal{T}e^4 .
		\\

		- \mathcal{Z}e^4 \mathcal{T}e^2
		&=&
		\mathcal{T}e^{2,1,3} - \mathcal{Z}e^3 \mathcal{T}e^3 .
		\\

		\phantom{-}4 \mathcal{Z}e^4 \mathcal{T}e^2
		&=&
		\mathcal{T}e^{2,2,2} .
		\\

		- \mathcal{Z}e^4 \mathcal{T}e^2
		&=&
		\mathcal{T}e^{3,1,2} + \mathcal{Z}e^3 \mathcal{T}e^3 .
		\\	

		\phantom{-}2 \mathcal{Z}e^4 \mathcal{T}e^2
		&=&
		\mathcal{T}e^{2,1,1,2} .
	\end{array}
$$

If we see the quantity $\mathcal{Z}e^4 \mathcal{T}e^2$ as a formal variable, this system is connected
to the column matrix ${}^t (4 ; -6 ; 4 ; 4 ; -1 ; 4 ; -1 ; 2)$, which has rank $1$~. Of course, the
second equation in the above system can be chosen to produce the simplest relations when we will
have to derive it.
\\

The first difficulty is to find out the dimension of the matrix we have to deal with. This is exactly
the question of the number of irreducible multizeta values of weight $n$, which is hypothetically solved
by the conjecture of Zagier. Of course, we can bypass this difficulty by treating all the possible values
of $c_n$, i.e. $1, \cdots,  d_n$ , but this is not really satisfying. Nevertheless, we know that,
if $A \in \mathcal{M}_{p,q} (\Q)$ has rank $q$ and if its column are denoted by $a_1, \cdots, a_q$,
then the matrix with columns $a_1 - \alpha a_2 ,  a_3, \cdots, a_r$ , $\alpha \in \Q$, has rank
$q - 1$~. Applying this principle to our matrix (eventually more than once), it is sufficient to consider
the case where $c_n = d_n$, i.e. to suppose that Zagier's conjecture holds.

If we suppose that Zagier's Conjecture holds, the second difficulty is now to evaluate the rank of
the matrix... We expect to obtain a maximal rank, which will therefore be $c_n$.
\\

Table \ref{matrices} shows us the submatrices obtained, using the values of multizeta values
given by \cite{table de MZV}, for the weight $4$, $5$, $6$ and $7$. Their rank are respectively
$1 = c_4$ , $2 = c_5$ , $2 = c_6$ , $3 = c_7$~. Our table
of multitangents up to weight $18$ shows easily that the rank of this matrix is $c_n$ up to $n = 16$~.
Consequently, Conjecture \ref{new conjecture on projection}, and then Conjectures
\ref{structure of the algebra MTGF, version 2} and \ref{MZV-module structure} , hold up to  weight $18$~.

\subsection	{About unit cleansing of multitangent functions}

Remind we have called valuation of a sequence of positive integers, the smallest integer composing this sequence.
Following \cite{Ecalle6}, we know that every multizeta value, even if it is a divergent one, can be
expressed as a $\Q$-linear combination of multizeta values with a valuation at least $2$ , the same
weight and a length which might be lesser. Moreover, this expression is unique up to the relations
of \symmetrelity. The most famous example of such a relation is due to Euler:
$\mathcal{Z}e^{2,1} = \mathcal{Z}e^3$.

Such an expression can be called a ``unit cleansing of multizeta values''.

\subsubsection	{A conjecture about cleansing of multitangent functions}
\label{ConjectureAboutCleansing}
Let us remind we have denoted $\mathcal{M}TGF_2 = Vect_\Q (\mathcal{T}e^{\seq{s}})	_{	\seq{s} \in \text{seq} (\N_2)	}$,
the vector space spanned by multitangents with valuation at least $2$, and 															\linebreak
$\mathcal{M}TGF_{2,p} = Vect_\Q (\mathcal{T}e^{\seq{s}})	_{	\seq{s} \in \text{seq} (\N_2)
																\atop
																||\seq{s}|| = p
															}$, the subspace of convergent multitangent functions with valuation at least $2$ and weight $p$.

We conjecture a similar result relatively to multitangent functions that one can call ``unit cleansing of multitangent functions''.
This can be written by the following

\begin{Conjecture}	\label{Unit cleansing conjecture}
	For all sequences $\seq{s} \in \mathcal{S}^\star_{b,e}$, $\mathcal{T}e^{\seq{s}} \in \mathcal{M}TGF_{2, ||\seq{s}||}$~.
\end{Conjecture}

For example, the simplest convergent multitangent, which is not cleaned and not the null function, is
$\mathcal{T}e^{2,1,3}$, since $\mathcal{T}e^{3,1,2}(z) = \mathcal{T}e^{2,1,3}(-z)$ and
$\mathcal{T}e^{2,1,2} = 0$~.
Its unit cleansing is given by the following relations, which can be proved by using relations
between multizeta values:

$$	\begin{array}{lll}
		\mathcal{T}e^{2,1,3}
		&=&
		\displaystyle	{	\frac	{1}	{4}
							\mathcal{T}e^{4,2}
							-
							\frac	{1}	{4}
							\mathcal{T}e^{2,4}
							+
							\frac	{1}	{6}
							\mathcal{T}e^{3,3}
						}
		\vspace{0.2cm}	\\
		&=&
		\displaystyle	{	\mathcal{T}e^{4,2}
							-
							\frac	{1}	{4}
							\mathcal{T}e^{2,4}
							-
							\frac	{1}	{4}
							\mathcal{T}e^{2,2,2}
						}
		\vspace{0.2cm}	\\
		&=&
		\displaystyle	{	\mathcal{T}e^{4,2}
							-
							\frac	{1}	{4}
							\mathcal{T}e^{2,4}
							+
							\frac	{1}	{15}
							\mathcal{T}e^{3,3}
							-
							\frac	{3}	{20}
							\mathcal{T}e^{2,2,2} \ .
						}
	\end{array}
$$

As this example shows us, there is no uniqueness of such a cleansing. This is, of course, due to the many relations between multitangent functions.
Here, the responsible relation is $3\mathcal{T}e^{2,2,2} + 2\mathcal{T}e^{3,3} = 0$, which is the prototype of a more general relation between multitangent functions: $\forall k \in \N^* \ , \ 3 \mathcal{T}e^{2^{[3k]}} + (-1)^k 2 \mathcal{T}e^{3^{[2k]}} = 0$. This relation is immediately obtained from the reduction into monotangent functions, or by the evaluation of $\mathcal{T}e^{n^{[p]}}$ that will be given in section \ref {calculs explicite}.

Table \ref{quelques exemples de nettoyage} gives us more examples of unit cleansing for multitangent functions.

\input{table4.sty}

\subsubsection	{On projection onto unit-free multitangent functions}

By analogy with Conjecture \ref{new conjecture on projection}, it is quite natural to consider the following conjecture:

\begin{Conjecture}	\label{new conjecture on projection/cleansing}
	For all sequences $\seq{\pmb{\sigma}} \in \mathcal{S}_b^\star$ , 
	$\mathcal{Z}e^{\seq{\pmb{\sigma}}} \mathcal{T}e^2 \in \mathcal{M}TGF_{2,||\seq{\pmb{\sigma}}|| + 2}$ .
\end{Conjecture}

Conjectures \ref{new conjecture on projection} and \ref{new conjecture on projection/cleansing} are
probably equivalent but it is sufficient for us to know that Conjecture
\ref{new conjecture on projection/cleansing} implies Conjecture \ref{new conjecture on projection}.

Of course, what we have said on the abstract formalization of the method is also valid in this case.
The only modification is to consider multitangent functions of
$\mathcal{M}TGF_{2,||\seq{\pmb{\sigma}}|| + 2}$ instead of
$\mathcal{M}TGF_{CV,||\seq{\pmb{\sigma}}|| + 2}$~. So, we will obtain a linear system with
$f_{n + 1} - 1$ equations\footnotemark \ and $c_n$ unknown ; the matrix we will obtain has size
$(f_{n + 1} - 1) \times c_n$~. Let us mention that the number of equations is $f_{n + 1} - 1$ since
it is, by definition, equals to the numbers of sequences in $\text{seq} (\N_2)$ with weight
$n$ and length $r \geq 2$.
\\

\footnotetext	{	Here, $(f_n)_{n \in \N}$ denote the classical Fibonacci sequence defined by
					$f_0 = 0$, $f_1 = 1$ and the recurence $f_{n + 2} = f_{n + 1} + f_n$ for all
					non negative integer $n$.
				}

Table \ref{matrices} shows us the obtained matrices\footnotemark, using the values of multizeta
values given by \cite{table de MZV}, for the weight $4$, $5$, $6$ and $7$, whose ranks are respectively
$1 = c_4$ , $2 = c_5$ , $2 = c_6$ and $3 = c_7$~. Again, the table of multitangents up to weight $18$
shows easily that the obtained system has a maximal rank, that is $c_n$, up to $n = 16$~. Consequently,
Conjecture \ref{new conjecture on projection/cleansing} holds up to the weight $18$~.

\footnotetext	{	That's why, in section \ref{Formalisation abstraite de la methode.}, we refer oursef to the same table as submatrix of these we must have.
					It was the submatrix obtained by considering only multitangent functions of $\mathcal{M}TGF_{2,||\seq{\pmb{\sigma}}|| + 2}$~.
				}

\input{table3.sty}

It leads to think that:
\begin{Conjecture}
	Let $(c_n)_{n \in \N}$ be the sequence defined by:
	$$	\left \{
			\begin{array}{l}
				c_1 = 0,  c_2 = c_3 = 1  .	\\
				\forall n \in \N, c_{n + 3} = c_{n + 1} + c_n .
			\end{array}
		\right.
	$$
	If $p \geq 2$, the $(f_{p + 1} - 1) \times c_p$ matrix obtained by the previous process
	from all the sequences of $\text{seq} (\N_2)$ of weight $p$ has rank $c_p$~.
\end{Conjecture}

This new conjecture is equivalent to Conjecture \ref{new conjecture on projection/cleansing} and hence implies Conjecture \ref{Unit cleansing conjecture} as well as Conjectures \ref{structure of the algebra MTGF, version 2} and \ref{MZV-module structure}~.
\\

Here is a quantitative argument to support this last conjecture, and consequently all of the other conjectures of this section.

The first matrices we have obtained contain lots of $0$'s. This allows us to say they have ``highly''
rank $c_n$ ; we are hinting that if this matrix has rank $c_n$, this is not by chance. It results
from the large number of equations and from the small number of unknowns (in comparison to the other
one) , but also from the repartition of the large number of zeros which forced the column vectors to
be linearly independent.

Moreover, we know that:
$$	\begin{array}{rcl l rcl}
			f_{p + 1}	&\approx&	0.45 \times 1,62^{p}
			&\hspace{1cm},\hspace{1cm} &
			c_p		&\approx&	0,41 \times 1,32^p \ .
	\end{array}
$$
\indent
Thus, the more $p$ will be tall, the more there will be ``chances'' to find some linearly independent rows. Therefore, the conjecture will be probably more true.

\subsubsection	{Unit cleansing of divergent multitangent functions}
\label{Unit cleansing of divergent multitangent functions}

Let us finish this section by a little anticipation on a later section. We will see in Section \ref{prolongement des multitangentes au cas divergent} that there exists a regularization process allowing us to define multitangent functions for sequences $\seq{s} \in \text{seq}(\N^*)$ which begin or end by a $1$. These functions will be expressed by the reduction into monotangent functions, with a small non-zero correction which will be a power of $\pi$ in a few cases. So, according to Conjecture \ref{new conjecture on projection}, each divergent multitangent function can be expressed as a $\Q$-linear combination of unit-free convergent multitangent functions.

Therefore, Conjecture \ref{Unit cleansing conjecture} can be generalised to:

\begin{Conjecture}
	For all sequences $\seq{s} \in \text{seq} (\N^*)$ , $\mathcal{T}e^{\seq{s}} \in \mathcal{M}TGF_{2, ||\seq{s}||}$ .
\end{Conjecture}
\section	{Algebraic properties}
\label{algebraic properties}
\subsection	{Is $\mathcal{M}TGF_{CV}$ a graded algebra ?}

Many conjectures have been stated about multizeta values. These are deep ones, but seem to be completely out of reach nowadays. We will see a first application of the dual process of reduction and projection. Thanks to it, we will state a new conjecture, which is related to a hypothetical structure of graded algebra. Then, we will see two simple examples where it is impossible to have non-trivial $\Q$-linear relations between multitangent functions of different weights.

\subsubsection	{Hypothetical absence of $\Q$-linear relations between different weight}

Let us remind the following well-known conjecture on multizeta values:

\begin{Conjecture} \label{Conjecture MZV}
	There is no non-trivial $\Q$-linear relation between multizeta values of different weights:
	$$	\displaystyle	{	\mathcal{M}ZV_{CV}
							=
							\bigoplus	_{p \in \N}
										\mathcal{M}ZV_{CV , p} \ .
						}
	$$
	In other words, $\mathcal{M}ZV_{CV}$ is a graded $\Q$-algebra.
\end{Conjecture}

Let us remark that this conjecture implies in particular the transcendence of all the numbers $\mathcal{Z}e^{s}$, where $s \geq 2$. According to the reduction/projection process, we can state the analogue conjecture for the multitangent functions:

\begin{Conjecture} \label{Conjecture MTGF}
	There is no non-trivial $\Q$-linear relation between multitangent functions of different weights:
	$$	\displaystyle	{	\mathcal{M}TGF_{CV}
							=
							\bigoplus	_{p \in \N}
										\mathcal{M}TGF_{CV , p + 2} \ .
						}
	$$
	In other words, $\mathcal{M}TGF_{CV}$ is a graded $\Q$-algebra.
\end{Conjecture}

In Section \ref{projection}, we have conjectured that for all sequences $\seq{\pmb{\sigma}} \in \mathcal{S}_b^\star$~, we have
$\mathcal{Z}e^{\seq{\pmb{\sigma}}} \mathcal{T}e^2 \in \mathcal{M}TGF_{CV , ||\seq{\pmb{\sigma}}|| + 2}$~.
The following property explains how these two conjectures are related.

\begin{Property}	\label{ConjectureEquivalenceTwoProperties}
	$1$. Conjecture \ref{Conjecture MZV} implies Conjecture \ref{Conjecture MTGF}.	\\
	$2$. Conjectures \ref{MZV-module structure} and \ref{Conjecture MTGF} imply Conjecture \ref{Conjecture MZV}.
\end{Property}

\begin{Proof}
		$1$. Suppose that there exists a $\Q$-linear relation between multitangent functions of different weights. So, 
			there exists a family of non zero	\linebreak
			$\Q$-linear combination of multitangent functions
			$(t_i)_{i \in I} \in	\displaystyle	{	\prod	_{i \in I}
																\mathcal{M}TGF_{CV , i}
													}
			$ such that
			$$	\displaystyle	{	\sum	_{i \in I}
											t_i
									=
									0
								} \ ,
			$$
			where $I$ is a finite subset of $\N$.
			
			Each $t_i$ is a $\Q$-linear combination of convergent multitangents of weight $i$ that we can suppose to be nonzero.
			By reduction into monotangent functions, for all terms $t_i$, there exists a familly $(z_{i,j})_{j \in \crochet{1}{i}}$ of multizeta values, $z_{i,j}$
			being of weight $j$, such that:
			$$	t_i = \displaystyle	{	\sum	_{k = 2}
												^i
												z_{i, i - k} \mathcal{T}e^k
									} .
			$$
			Let us remark that, for a fixed $i$ in $I$, there exists $z_{i,j} \neq 0$ (otherwise, we would have $t_i = 0$)~. Thus, denoting by $M$ the greatest element of $I$, we can write:
			\begin{center}
				$	\displaystyle	{	0	=	\sum	_{i \in I}
														t_i
											=	\sum	_{i \in I}
														\sum	_{k = 2}
																^i
																z_{i, i - k} \mathcal{T}e^k
											=	\sum	_{k = 2}
														^M
														\Bigg(
																\sum	_{	i \in I
																			\atop
																			i \geq k
																		}
																		z_{i , i - k}
														\Bigg )
														\mathcal{T}e^k  .
									}
				$
			\end{center}
			Consequently, the linear independence of the monotangent functions implies that for all $k \in \crochet{2}{M}$:
			$$	\sum	_{i \in I  \atop  i \geq k}
					z_{i , i - k}
				=
				0 \ .
			$$
			We have obtained a non trivial $\Q$-linear relation between multizeta values of different weights. Thus:
			$	\displaystyle	{	\mathcal{M}ZV_{CV}
									\neq
									\bigoplus	_{k \in \N^*}
												\mathcal{M}ZV_{CV,k} \ .
								}
			$
			\\
			We have therefore shown that:
			$$	\mathcal{M}TGF_{CV}
				\neq
				\bigoplus	_{k \in \N^*}
							\mathcal{M}TGF_{CV,k}
				\Longrightarrow
				\mathcal{M}ZV_{CV}
				\neq
				\bigoplus	_{k \in \N^*}
							\mathcal{M}ZV_{CV,k} \ .
			$$
			\\
			$2$. Suppose now that there exists some non trivial $\Q$-linear relation between
			multizeta values of different weights. So, there exist two famillies, one of
			sequences $\seq{s}^1$, $\cdots$, $\seq{s}^n$ in $\mathcal{S}_b^\star$ and the
			second of non-zero rational numbers					 \linebreak
			$c_1$, $\cdots$, $c_n$ such that
			$$	\displaystyle	{	\sum	_{i = 1}
								^n
								c_i \mathcal{Z}e^{\seq{s}^i}
						}
				=
				0
				\ ,
			$$
			where the map $i \longmapsto ||\seq{s}^i||$ is supposed non constant.
			\\
			Thus:
			$$	\sum	_{i = 1}
					^n
					c_i \mathcal{Z}e^{\seq{s}^i} \mathcal{T}e^2 (z)
				=
				0 \ , \text{ for all } z \in \C - \Z \ .
			$$
			
			According to Conjecture \ref{MZV-module structure}, the term
			$c_i \mathcal{Z}e^{\seq{s}^i} \mathcal{T}e^2 (z)$ can be expressed as
			a $\Q$-linear combination of multitangent functions
			$$	c_i \mathcal{Z}e^{\seq{s}^i} \mathcal{T}e^2
				=
				\displaystyle	{	\sum	_{j = 1}
								^{n_i}
								c_{i, j} \mathcal{T}e^{\seq{s}^{i,j}}
						} \ ,
			$$
			where $||\seq{s}^{i,j}|| = ||\seq{s}^i|| + 2$ and for each index $i$, there
			exists an index $j$ such that $c_{i,j} \neq 0$~(otherwise, $c_i$ would be $0$)~.

			Therefore, we obtain a non-trivial $\Q$-linear relation between multitangent
			functions not all of the same weight:
			$$	\displaystyle	{	\sum	_{i = 1}
								^n
								\sum	_{j = 1}
									^{n_i}
 									c_{i, j} \mathcal{T}e^{\seq{s}^{i,j}}
    							=
 							0
						} \ .
			$$
			This is a contradiction with the absence of $\Q$-linear relation between multitangent functions of different weights. Therefore, we have shown:
			\begin{center}
			$	\left \{
					\begin{array}{@{}l}
						\displaystyle	{	\mathcal{M}TGF_{CV}
											=
											\bigoplus	_{k \in \N^*}
														\mathcal{M}TGF_{CV,k}
										}
						\\
						\forall \seq{\pmb{\sigma}} \in \mathcal{S}_b^\star , 
						\mathcal{Z}e^{\seq{\pmb{\sigma}}} \mathcal{T}e^2 \in \mathcal{T}_{||\seq{\pmb{\sigma}}|| + 2}
					\end{array}
				\right.
				\hspace{-0.5cm}
				 \Longrightarrow 
				\displaystyle	{	\mathcal{M}ZV_{CV}
									=
									\bigoplus	_{k \in \N^*}
												\mathcal{M}ZV_{CV,k} \ .
								}
			$
			\end{center}
			\qed
\end{Proof}

\subsubsection	{Transcendence of multitangent functions which are not identically zero}

As an example of the absence of the existence of $\Q$-linear combinations between multitangent
of different weight, we can of course think about the linear independence of monotangent functions
given in Lemma \ref{linear independence of monotangent functions}. Another example concerns a
transcendence property.

In order to use a transcendence method, it may be useful to know if a function is transcendent or not.
Here, a transcendent function means a function which is transcendant over $\C[X]$. If we found a nonzero multitangent function which is not transcendent,
then we would have a $\Q$-linear relation between multitangents of different weights, which would be a
contradiction to the property \ref{ConjectureEquivalenceTwoProperties}. Fortunately for
Conjecture \ref{Conjecture MTGF}, we can state that:

\begin{Lemma}
	Any nonzero multitangent function is transcendent.
\end{Lemma}

Let us remark that if we want to be able to use this lemma in a transcendence argument, it will be necessary to characterize the null multitangent functions. This will be hypothetically done in a forthcoming section (see Section \ref{odd, even or null mtgf}).

\begin{Proof}
	Let us consider $\seq{s} \in \mathcal{S}^\star_{b,e}$ such that
	$\mathcal{T}e^{\seq{s}} \not \equiv 0$~.

	If we suppose that $\mathcal{T}e^{\seq{s}}$ is an algebraic function, there exists
	a polynomial $P \in \C [X ; Y]$ such that
	$P \big ( z ; \mathcal{T}e^{\seq{s}}(z) \big ) \equiv 0$. Let us consider the smallest
	possible degree in $Y$ of such a polynomial, which will be denoted by $d$~. Writing $P$
	in an expanded form, there exists a non-trivial familly of polynomials
	$(P_i)_{i \in \crochet{0}{d}}$ such that we would have:
	$$	\displaystyle	{	\sum	_{i = 0}
						^d
						P_i(z)	\left (
								\mathcal{T}e^{\seq{s}} (z)
							\right )
						^i
					=
					0 \ , \text{ for all } z \in \C - \Z\ .
				}
	$$
	Thanks to the exponentially flat character of convergent multitangent functions (see \S
	\ref{caractere exponentiellement plat})	when $\im z$ goes to the infinity, we have
	for all polynomial $P$ and sequence $\seq{s} \in \mathcal{S}_{b,e}^\star$:
	$$  P(z) \mathcal {T}e^{\seq{s}} (z)	\longrightarrow 0 \ .	$$

	Thus:	$$	\displaystyle	{	\sum	_{i = 1}
							^d
							P_i(z)	\left (
									\mathcal{T}e^{\seq{s}} (z)
								\right )
								^i
						\longrightarrow
						0
					} \ .
		$$
	Because the function	$	\displaystyle	{	\sum	_{i = 0}
									^d
									P_i	\left (
											\mathcal{T}e^{\seq{s}}
										\right )
										^i
							}
				$ is supposed to be null, this shows that
	$P_0(z) \longrightarrow 0$, which means that $P_0$ is null. From the hypothesis,
	$\mathcal{T}e^{\seq{s}}$ is not the null function. So:
	$$	\displaystyle	{	\sum	_{i = 1}
						^d
						P_i(z)	\left (
								\mathcal{T}e^{\seq{s}} (z)
							\right )
							^{i - 1}
					=
					\sum	_{i = 0}
						^{d - 1}
						P_{i + 1} (z)	\left (
									\mathcal{T}e^{\seq{s}} (z)
								\right )
								^i
					\equiv
					0 \ .
				}
	$$
	This contradicts the fact that $d$ is the smallest possible degree in $Y$ for such a
	polynomial $P$. Consequently, we have shown that every nonzero multitangent function
	is transcendent.
	\qed
\end{Proof}

\subsection	{On a hypothetical basis of $\mathcal{M}TGF_{CV , p}$}

In this paragraph, we will study the analogue of Zagier's conjecture on the dimension of
$\mathcal{M}ZV_{CV , p}$~. For this, we will use the reduction/projection process to
translate it in $\mathcal{M}TGF_{CV}$~. Recall that the Zagier conjecture states that
$(\text{dim } \mathcal{M}ZV_{CV , n})_{n \in \N}$ satisfies the recurrence relation:
$$	\left \{
			\begin{array}{l}
					c_0 = 1  \ , \ c_1 = 0  \ , \ c_2 = 1  \ .	\\
					c_{n + 3} = c_{n + 1} + c_n \text{ , for all } n \in \N\ .
			\end{array}
	\right.
$$

The analogue for $\mathcal{M}TGF_{CV}$ is:

\begin{Conjecture}	\label{dimension}
	$(\text{dim } \mathcal{M}TGF_{CV , n + 2})_{n \in \N}$ satisfies the recurrence relation:
	$$	\left \{
				\begin{array}{l}
					d_0 = d_1 = 1 	\ , \ d_2 = 2	\ , \ d_3 = 3			\ .	\\
					d_{n + 4} = d_{n + 3} + d_{n + 2} - d_n \text{ , where } n \in \N\ .
				\end{array}
		\right.
	$$
\end{Conjecture}

Of course, this conjecture is related to that of Zagier:

\begin{Property}
	Let us suppose that Conjecture \ref{MZV-module structure} holds, i.e. that $\mathcal{M}TGF_{CV}$ is a $\mathcal{M}ZV_{CV}$-module.	\\
	Then, Conjecture \ref{dimension} is equivalent to the Zagier's conjecture.
\end{Property}

The proof is based on Property \ref{property 8} stated below. We have the following formal power
series which are respectively the hypothetically Hilbert-Poincar\'e series of the hypothetically
graded $\Q$-algebras $\mathcal{M}ZV_{CV}$ and $\mathcal{M}TGF_{CV}$:
$$	\begin{array}{lclcl}
		\mathcal{H}_{\mathcal{M}ZV_{CV}}	&=&					\displaystyle	{	\sum	_{p \in \N}
																							\text{dim } \mathcal{M}ZV_{CV , p}
																							X^p
																				}
											&\stackrel{?}{=}&	\displaystyle	{	\frac	{1}	{1 - X^2 - X^3}	} \ .
		\\
		\mathcal{H}_{\mathcal{M}TGF_{CV}}	&=&					\displaystyle	{	\sum	_{p \in \N}
																							\text{dim } \mathcal{M}TGF_{CV , p + 2}
																							X^p
																				}
											&\stackrel{?}{=}&	\displaystyle	{	\frac	{1}	{(1 - X^2 - X^3)(1 - X)}	} \ .
	\end{array}
$$
So, it shall be sufficient to prove that $(1 - X) \mathcal{H}_{\mathcal{M}TGF_{CV}} = \mathcal{H}_{\mathcal{M}ZV_{CV}}$, which is done in the following:

\begin{Property}	\label{property 8}
	Let us suppose that Conjecture \ref{MZV-module structure} holds, i.e. that $\mathcal{M}TGF_{CV}$ is a $\mathcal{M}ZV_{CV}$-module.	\\
	$1$.	If	$	\left (
							\mathcal{Z}e^{\seq{s}_p^k}
					\right )
					_{k = 1 , \cdots , \text{dim } \mathcal{M}ZV_{CV , p}}
				$ denotes a basis of $\mathcal{M}VZ_{CV , p}$ for all $p \in \N$ ,
			then	$	\left (
								\mathcal{Z}e^{\seq{s}_u^k}	\mathcal{T}e^v
						\right )
						_{	{	k = 1 , \cdots , \text{dim } \mathcal{M}ZV_u
								\atop
								u + v = p + 2
							}
							\atop
							v \geq 2
						}
					$ is a basis of $\mathcal{M}TGF_{CV , p + 2}$ for all $p \in \N$~.
	\\
	$2$.	We have:	$	\displaystyle	{	\forall p \in \N \ , \ 
												\text{dim } \mathcal{M}TGF_{CV , p + 2}	= \sum	_{k = 0}
																								^{p}
																								\text{dim } \mathcal{M}ZV_{CV , k}
											} \ .
						$
\end{Property}

\begin{Proof}
	Since the second point follows directly from the first one, we will only prove that a basis of $\mathcal{M}TGF_{CV , p + 2}$ is given by the family:
	$$	\left (
				\mathcal{Z}e^{\seq{s}_u^k}	\mathcal{T}e^v
		\right )
		_{	{	k = 1 , \cdots , \text{dim } \mathcal{M}ZV_u
				\atop
				u + v = p + 2
			}
			\atop
			v \geq 2
		}\ .
	$$
	\noindent
	\underline{Step 0:}\\
	\\
	Because we have supposed that $\mathcal{M}TGF_{CV}$ is a $\mathcal{M}ZV_{CV}$-module, each term of the hypothetical basis is indeed an element of
	$\mathcal{M}TGF_{CV , p + 2}$~.\\
	\\
	\underline{Step 1:} the linear independence property.	\\
	\\
	If we suppose the existence of scalars $(\lambda_{v , k})$ such that:
	$$	\displaystyle	{	\sum	_{v = 2}
									^{p + 2}
									\hspace{0.2cm}
									\sum	_{k = 1}
											^{\text{dim }\mathcal{M}ZV_{CV , p + 2 - v}}
											\lambda_{v , k}
											\mathcal{Z}e^{\seq{s}_u^k}
											\mathcal{T}e^v
							=
							0
						}\ ,
	$$
	by the linear independence of monotangent functions, we obtain:
	$$	\forall v \in \crochet{2}{p + 2} \ , \	\sum	_{k = 1}
														^{\text{dim }\mathcal{M}ZV_{CV , p + 2 - v}}
														\lambda_{v , k}
														\mathcal{Z}e^{\seq{s}_u^k}
												=
												0 \ .
	$$
	Consequently, from the linear independence of the	$	\left (
																	\mathcal{Z}e^{\seq{s}_u^k}
															\right )
															_{	k = 1 , \cdots , \text{dim } \mathcal{M}ZV_u	}
														$, we conclude that all the scalars $\lambda_{v , k}$ are null, which concludes this step.\\
	\\
	\underline{Step 2:} the spanning property.	\\
	\\
	By the reduction into monotangent functions, one writes each multitangent function in terms of $\mathcal{Z}e^{\seq{s}^1} \mathcal{Z}e^{\seq{s}^2} \mathcal{T}e^n$. Consequently, using the \symmetrelity of the mould $\mathcal{Z}e^\p$, we now only have to express each multizeta value of weight $k$ which appears in such a relation in terms of the basis of $\mathcal{M}ZV_{CV , k}$~.
	
	By this process, we would have expressed each convergent multitangent function in terms of $\mathcal{Z}e^{\seq{s}_u^k}	\mathcal{T}e^v$ , where $k \in \crochet{1}{\text{dim } \mathcal{M}ZV_u}$ , $u + v = p + 2$ and $v \geq 2$ because the reduction into monotangent functions, as well as the \symmetrelity, preserves the weight.
	\qed
\end{Proof}

To conclude this paragraph, we give in the following figure the first hypothetical dimensions of the space of multitangent functions of weight $p + 2$~. We can recognize the sequences $A000931$ and $A023434$ from the On-Line Encyclopedia of Integer Sequences (see \cite{Sloane})

\begin{figure}[h]

	$$	\begin{array}{|c|c|c|c|c|c|c|c|c|c|c|c|c|c|}
			\hline
			p										&	0	&	1	&	2	&	3	&	4	&	5	&	6	&	7	&	8	&	9	&	10	&	11	&	12	\\
			\hline
			
			\text{dim } \mathcal{M}ZV_{CV , p}		&	1	&	0	&	1	&	1	&	1	&	2	&	2	&	3	&	4	&	5	&	7	&	9	&	12	\\
			\hline
			
			\text{dim } \mathcal{M}TGF_{CV , p + 2}	&	1	&	1	&	2	&	3	&	4	&	6	&	8	&	11	&	15	&	20	&	27	&	36	&	48	\\
			\hline
		\end{array}\\
	$$
	
	\caption{\textbf{The first hypothetical dimensions of the space of multitangent functions of weight $p + 2$.}}

\end{figure}

\subsection	{The $\Q$-linear relations between multitangent functions}
\label{linear relations between mtgfs}
We know that multizeta values have two encodings. The first one is the one we have used since
the beginning of this article, resulting from the specialization $x_n = n^{-1}$ (up to a
convention of choosing the summation sequence to be an increasing or a decreasing sequence)
of the monomial basis of quasi-symmetric functions: it is exactly the \symmetrel mould
$\mathcal{Z}e^\p$~. The reader is invited to consult \cite{Gessel} for the first appearance
of quasi-symmetric functions in literature, as well as \cite{Bergeron}, \cite{Hazewinkel},
\cite{Hivert}, or \cite{Stanley} for other presentations. The second encoding of
multizeta values comes from an iterated integral representation (which has been mentioned
in the introduction, see \S \ref{iterated integral representation}): it is the \symmetral
mould $\mathcal{W}a^\p$~.

Because of the dual process reduction/projection, we can imagine that the \symmetral encoding
has a translation in the algebra $\mathcal{M}TGF$~. One can think that a quadratic relation in
$\mathcal{M}ZV$ will be translated in $\mathcal{M}TGF$ into another quadratic relation, but this
does not provide to be true (essentially because Hurwitz multizeta functions have one and only one
encoding, see \cite{Bouillot-Cras})~. In fact, there are lots of null $\Q$-linear relations between multitangent functions;
that correspond to the \symmetrality relations. Of course, the existence of such relations is
a natural fact, but, as we will see, they all have an odd look. The situation is not yet
completely understood.

We will see in the following that sometimes, it is possible to prove easily a relation using
multitangent properties (because it is the derivative of a well-know relation between multitangent
function or using a parity argument). Most of the time, these relations remain completely mysterious.

Now, let us explain what happens for small weights. 

\paragraph	{Up to weight $5$}

The only $\Q$-linear relation up to weight $5$ is the surprising existence of a null multitangent
function, $\mathcal{T}e^{2,1,2}$~. Although it is an interesting fact, we postpone the study of
null multitangent functions to \S \ref{odd, even or null mtgf}~. Let us just mention that
many multitangent functions are the null function and that we have a conjectural
characterization of those.

\paragraph	{Weight $6$}

Since
$\text{dim } \mathcal{M}ZV_{CV,2} =  \text{dim } \mathcal{M}ZV_{CV,3} = \text{dim } \mathcal{M}ZV_{CV,4} = 1$,
we deduce\footnotemark\, from Conjecture \ref{structure of the algebra MTGF, version 2} that
$\text{dim } \mathcal{M}TGF_{CV,6} = 4$~. Consequently, there exists exactly four independent
$\Q$-linear relations between the eight convergent multitangent functions of weight $6$. Actually, it is not
difficult to find them, using the known values of multizeta values of weight $4$~.

\footnotetext	{	Because Conjecture \ref{new conjecture on projection} is true for the global
			weight $6$, as we have seen with the table~\ref{quelques
			expressions de projections de multizetas}~.
		}

These four independent relations are given in Table \ref{linear combination of weight 6}.

\begin{table}[h]
	
	$$	\hspace{-2.25cm}
		\begin{array}{|r@{\ }c@{\ }l||r@{\ }c@{\ }l|c|}
			\hline
			\multicolumn{3}{|c||}{\text{Relations in } \mathcal{M}TGF_6}
			&
			\multicolumn{4}{c|}{\text{Equivalent relations in }\mathcal{M}ZV}
			\\
			&&&&&\text{Relations}&\text{Origin}
			\\
			\hline
			\mathcal{T}e^{3,1,2} + \mathcal{T}e^{2,1,3} + \mathcal{T}e^{2,1,1,2}	&=&	0 \ .
			&
			\mathcal{Z}e^{2, 1, 1} &=& \mathcal{Z}e^{3, 1} + \mathcal{Z}e^{2, 2} \ .
			&
			\text{Double-shuffle}
			\\
			\hline
			2 \mathcal{T}e^{3,1,2} + \mathcal{T}e^{2,2,2} + 2 \mathcal{T}e^{2,1,3}	&=&	0 \ .
			&
			\big (\mathcal{Z}e^{2} \big)^2 &=& 4 \mathcal{Z}e^{3, 1} + 2 \mathcal{Z}e^{2, 2} \ .
			&
			\text{Shuffle}
			\\
			\hline
			\mathcal{T}e^{2,4} - \mathcal{T}e^{4,2} + 2 \mathcal{T}e^{2,1,3} - 2 \mathcal{T}e^{3,1,2}	&=&	0 \ .
			&
			\left \{
					\begin{array}{@{}l@{}}
						\mathcal{Z}e^{3} \mathcal{Z}e^{2}	\\
						\mathcal{Z}e^{3} 	\\
					\end{array}
			\right.
			&
			\begin{array}{c}
				=	\\
				=
			\end{array}
			&
			\begin{array}{@{}l@{}}
				6 \mathcal{Z}e^{4, 1} + 3 \mathcal{Z}e^{3, 2} + \mathcal{Z}e^{2, 3} \ .	\\
				\mathcal{Z}e^{2, 1} \ .
			\end{array}
			&
			\begin{array}{@{}c@{}}
				\text{Shuffle}	\\
				\text{Double-shuffle}
			\end{array}
			\\
			\hline
			3 \mathcal{T}e^{3,1,2} + 3 \mathcal{T}e^{2,1,3}	- \mathcal{T}e^{3,3} &=&	0 \ .
			&
			\mathcal{Z}e^{4} &=& \mathcal{Z}e^{3, 1} + \mathcal{Z}e^{2, 2}
			&
			\text{Double-shuffle}
			\\
			\hline
		\end{array}
	$$
	
	\caption{\textbf{The four independent $\Q$-linear relations between multitangent functions of weight $6$.}}
	\label{linear combination of weight 6}
\end{table}

\paragraph	{Weight $7$}

Concerning the weight $7$, we obtain Table $\ref{linear combination of weight 7}$. As it is
conjectured that $\text{dim } \mathcal{M}TGF_7 = 6$ and as $\mathcal{M}TGF_7$ is spanned
by sixteen functions, one could find exactly ten independent relations. That is what we obtain,
but the remaining question is ``are there any other relations between multitangent functions
of weight $7$?'' Consequently, this is the first weight where we can only speak hypothetically.

\begin{table}[h]

	$$	\begin{array}{|r@{\ }c@{\ }l|}
			\hline
			\multicolumn{3}{|c|}{\text{Relations in } \mathcal{M}TGF_7}
			\\
			\hline
			\mathcal{T}e^{2,1,1,1,2}	&=&	0 \ .
			\\
			- 4 \mathcal{T}e^{3,1,3} + \mathcal{T}e^{3,1,1,2} + \mathcal{T}e^{2,1,1,3}	&=&	0 \ .
			\\
			4 \mathcal{T}e^{3,1,3} - 2 \mathcal{T}e^{3,1,1,2} + \mathcal{T}e^{2,1,2,2}	&=&	0 \ .
			\\
			-4 \mathcal{T}e^{3,1,3} + 2 \mathcal{T}e^{3,1,1,2}	+ \mathcal{T}e^{2,2,1,2} &=&	0 \ .
			\\
			\mathcal{T}e^{4,1,2} + 5 \mathcal{T}e^{3,1,3}	+ \mathcal{T}e^{2,1,4} &=&	0 \ .
			\\
			- \mathcal{T}e^{4,3} + \mathcal{T}e^{4,1,2}	+ 5\mathcal{T}e^{3,1,3} + \mathcal{T}e^{2,2,3} - 4 \mathcal{T}e^{3,1,1,2} &=&	0 \ .
			\\
			-5 \mathcal{T}e^{3,1,3} + \mathcal{T}e^{2,3,2} &=&	0 \ .
			\\
			\mathcal{T}e^{4,3} -  \mathcal{T}e^{4,1,2} + \mathcal{T}e^{3,2,2} - 8 \mathcal{T}e^{3,1,3} + 4 \mathcal{T}e^{3,1,1,2} &=&	0 \ .
			\\
			- \mathcal{T}e^{5,2} + \mathcal{T}e^{2,5} - 4 \mathcal{T}e^{4,1,2} - 18 \mathcal{T}e^{3,1,3} + 4 \mathcal{T}e^{3,1,1,2} &=&	0 \ .
			\\
			\mathcal{T}e^{4,3} + \mathcal{T}e^{3,4}	+ 8 \mathcal{T}e^{3,1,3} &=&	0 \ .
			\\
			\hline
		\end{array}
	$$
	
	\caption{\textbf{The ten independent $\Q$-linear relations between multitangent functions of weight $7$.}}
	\label{linear combination of weight 7}
\end{table}

\subsection	{On the possibility of finding relations between multizeta values from the multitangent functions}

\subsubsection	{Two different process for multiplying two multitangent functions.}

It is clear that we have two possibilities to compute a product of two multitangent functions,
according either to the \symmetrelity of $\mathcal{T}e^\p$ or to the relations of reduction into
monotangent functions. This is summed up in the following diagram, where ``reduction''
(resp. ``\symmetrel multiplication'') indicates the linear extension to $\mathcal{M}TGF_{CV}$
(resp. $\mathcal{M}TGF_{CV} \otimes \mathcal{M}TGF_{CV}$) of the reduction process (resp. the
multiplication by the \symmetrelity of $\mathcal{T}e^\p$):

\begin{center}	\label{diagramme}
	\hspace{0.1cm}
	$\xymatrix	{	\ar@{}[ddrrr] 
					\mathcal{M}TGF \otimes \mathcal{M}TGF
					\ar[rrr]^-{\symmetrel}_{\phantom{mumu}\text{multiplication}}
					\ar[d]_{\text{reduction} \otimes \text{reduction}}							&&&	\mathcal{M}TGF \ar[dd]^-{\text{reduction}}	\\
					\mathcal{M}TGF \otimes \mathcal{M}TGF
					\ar[d]_{\symmetrel}^{\text{multiplication}}								&&&												\\
					\mathcal{M}TGF \ar[rrr] ^{\text{reduction}}									&&&	\mathcal{M}TGF
				}
	$
\end{center}

Obviously, these two processes give the same result in $\mathcal{M}TGF_{CV}$, but the expressions
are different. This gives us the opportunity to find out some relations between multizeta values.

First, we see that the previous commutative diagram gives us a way to find out all the relations
of \symmetrelity. Let us emphasize that the following property is not the best way to prove the
\symmetrelity of $\mathcal{Z}e^\p$, but its aim is just to begin to describe the relations between
multizeta values obtained from relations between multitangent functions.

\begin{Property}\label{relations de symetrelite}
	\begin{enumerate}
		\item The relations of \symmetrelity of $\mathcal{T}e^\p$ and the previous commutative diagram imply all the relations of \symmetrelity of
				$\mathcal{Z}e^\p$.
		\item The previous commutative diagram gives us more relations than those of \symmetrelity.
	\end{enumerate}
\end{Property}

\begin{Proof}
	$1$. First of all, let us explain how we can extract the leading term in the reduction in monotangent functions of the multitangent $\mathcal{T}e^{\seq{s}}$,
	where $\seq{s} \in \mathcal{S}^*_{b,e}$, if $\text{max }(s_1, \cdots, s_r)$ is reached one and only one time at the index $i$. According to the
	expression \eqref{Ziks} and Theorem \ref{reduction en monotangente}, this coefficient is:
	$$	\mathcal{Z}_{i, 0}^{\seq{s}}
		= {}^{i}B^{\seq{s}}_{0, \cdots, \widehat{0}, \cdots, 0} \mathcal{Z}e^{s_r, \cdots, s_{i + 1}} \mathcal{Z}e^{s_1, \cdots, s_{i - 1}}
		= (-1)^{s_{i + 1} + \cdots + s_r}	\mathcal{Z}e^{s_r, \cdots, s_{i + 1}} \mathcal{Z}e^{s_1, \cdots, s_{i - 1}}
		.
	$$
	\noindent
	Therefore, if $i = r$, we have: $	\mathcal{Z}_{r, 0}^{\seq{s}} = \mathcal{Z}e^{s_1, \cdots, s_{r - 1}}	$.
	
	\bigskip
	
	Now, let us consider two sequences $\seq{s}^1$ and $\seq{s}^2$ in $\mathcal{S}^\star_{b,e}$ and denote by $M_0$, the largest integer which appears in these
	two sequences, and finally set $M = M_0 + 1$.
	\\
	\\
	Using the \symmetrelity of $\mathcal{T}e^\p$ and then from the recursive definition of the set sh\textit{\textbf{\underline{e}}}, we obtain:
	$$	\mathcal{T}e^{\seq{s}^1 \cdot M}\mathcal{T}e^{\seq{s}^2 \cdot M}
		=
		\displaystyle	{	\sum	_{	\seq{s}
										\in
										sh\textit{\textbf{\underline{e}}}(\seq{s}^1 ;  \seq{s}^2 \cdot M)
										\cup
										sh\textit{\textbf{\underline{e}}}(\seq{s}^1 \cdot M   ;   \seq{s}^2)
									}
									\mathcal{T}e^{\seq{s} \cdot M}
							+
							\sum	_{	\seq{s}
										\in
										sh\textit{\textbf{\underline{e}}}(\seq{s}^1  ;  \seq{s}^2)
									}
									\mathcal{T}e^{\seq{s} \cdot 2M}
						} \ .
	$$

	Thus, the coefficient of $\mathcal{T}e^{2M}$ in the product $\mathcal{T}e^{\seq{s}^1 \cdot M}\mathcal{T}e^{\seq{s}^2 \cdot M}$, obtained first by the \symmetrel multiplication and then by reduction into monotangent functions, is:
	$$	\displaystyle	{	\sum	_{	\seq{s}
										\in
										sh\textit{\textbf{\underline{e}}}(\seq{s}^1   ;  \seq{s}^2)
									}
									\mathcal{Z}e^{\seq{s}}
						} \ .
	$$
	
	On another hand, in the reduction into monotangent functions of $\mathcal{T}e^{\seq{s}^1 \cdot M}$ and $\mathcal{T}e^{\seq{s}^2 \cdot M}$, the only monotangents which may contribute to $\mathcal{T}e^{2M}$ are the terms $\mathcal{T}e^{M}$ coming from the reduction into monotangent functions of $\mathcal{T}e^{\seq{s}^1 \cdot M}$ and $\mathcal{T}e^{\seq{s}^2 \cdot M}$. These are respectively equal to $\mathcal{Z}e^{\seq{s}^1} \mathcal{T}e^{M}$ and $\mathcal{Z}e^{\seq{s}^2} \mathcal{T}e^{M}$. Consequently, the coefficient of $\mathcal{T}e^{2M}$ in the product $\mathcal{T}e^{\seq{s}^1 \cdot M}\mathcal{T}e^{\seq{s}^1 \cdot M}$, obtained first by reduction into monotangent functions, then by the \symmetrel multiplication and finally by reduction one more time is: $\mathcal{Z}e^{\seq{s}^1} \mathcal{Z}e^{\seq{s}^2}$.

	As a consequence, we obtain the relation we are looking for:
	$$	\mathcal{Z}e^{\seq{s}^1} \mathcal{Z}e^{\seq{s}^2}
		=
		\displaystyle	{	\sum	_{	\seq{s}
										\in
										sh\textit{\textbf{\underline{e}}}(\seq{s}^1  ;  \seq{s}^2)
									}
									\mathcal{Z}e^{\seq{s}}
						} \ .
	$$
	
	$2$. We make the computation in the simplest case where the diagram gives a result:
	\begin{enumerate}
		\item $		\begin{array}[t]{lll}
						\mathcal{T}e^2 \cdot \mathcal{T}e^{2,2}
						&=&
						2 \mathcal{Z}e^2 \big ( \mathcal{T}e^2 \big )^2
						=
						4 \mathcal{Z}e^2 \mathcal{T}e^{2,2} + 2 \mathcal{Z}e^2 \mathcal{T}e^4
						\vspace{0.15cm}
						\\
						&=&
						8 \big ( \mathcal{Z}e^2 \big )^2 \mathcal{T}e^2 + 2 \mathcal{Z}e^2 \mathcal{T}e^4 \ .
					\end{array}
				$
		\vspace{0.15cm}
		\item $		\begin{array}[t]{lll}
						\mathcal{T}e^2 \cdot \mathcal{T}e^{2,2}
						&=&
						3 \mathcal{T}e^{2,2,2} + \mathcal{T}e^{4,2} + \mathcal{T}e^{2,4}
						\vspace{0.15cm}
						\\
						&=&
						\left (
								3 \big ( \mathcal{Z}e^2 \big ) ^2 + 6 \mathcal{Z}e^{2,2} + 8 \mathcal{Z}e^4
						\right ) \mathcal{T}e^2 + 2 \mathcal{Z}e^2 \mathcal{T}e^4 \ .
					\end{array}
				$
	\end{enumerate}

	\noindent
	Thus, by the linear independence of monotangent functions, we obtain:
	\begin{equation} \label{(*)}
		6 \mathcal{Z}e^{2,2} + 8 \mathcal{Z}e^{4} = 5 \big (\mathcal{Z}e^2 \big )^2 \ .
	\end{equation}
	\noindent
	Using the \symmetrelity of multizeta values, $(\ref{(*)})$ can be written:
	\begin{equation} \label{(**)}
		3 \mathcal{Z}e^{4} = 4 \mathcal{Z}e^{2, 2} \ .
	\end{equation}
	But, the only relation of \symmetrelity of weight $4$ is:
	\begin{equation} \label{(***)}
		\big (\mathcal{Z}e^2 \big )^2 = 2\mathcal{Z}e^{2,2} + \mathcal{Z}e^{4} \ .
	\end{equation}
		
	If $(\ref{(**)})$ could be proven from $(\ref{(***)})$, using only the \symmetrelity of $\mathcal{Z}e^\p$, we would know the following values
	$	\left \{
				\begin{array}{rll}
						\mathcal{Z}e^4	&=&	\displaystyle	{	\frac	{2}	{5}
																\big (
																		\mathcal{Z}e^2
																\big )
																^2  .
															}
						\vspace{0.2cm}
						\\
						\mathcal{Z}e^{2,2}	&=&	\displaystyle	{	\frac	{3}	{10}
																	\big (
																			\mathcal{Z}e^2
																	\big )
																	^2  .
																}
				\end{array}
		\right.
	$

	To prove these relations, we actually need the following three equations coming
	from the three types of relations which are known to conjecturally describe all
	the relation between multizeta values, that is the relations of \symmetrality\!\!/\symmetrelity
	and the double shuffle relations:
	$$	\begin{array}{rll}
				\big (\mathcal{Z}e^2 \big )^2 	&=&	2\mathcal{Z}e^{2,2} + \mathcal{Z}e^{4}		\ .		\\
				\big (\mathcal{Z}e^2 \big )^2 	&=&	2\mathcal{Z}e^{2,2} + 4\mathcal{Z}e^{3,1}	\ .		\\
				\mathcal{Z}e^4 					&=& \mathcal{Z}e^{2,2} + \mathcal{Z}e^{3,1}		\ .
		\end{array}
	$$

	This proves that the diagram gives us more relations than those of \symmetrelity.
	\qed
\end{Proof}

Let us also notice that we can write another commutative diagram for the derivative of a
multitangent function, but this does not provide a way to find new relations between multizeta values.

\subsection	{Back to the absence of the monotangent $\mathcal{T}e^1$ in the relations of reduction}
\label{absence de composante TE1}

We will see that the convergent multitangent functions are exponentially flat, near infinity
(see \textsection \ref{caractere exponentiellement plat}). This implies the absence of the
monotangent $\mathcal{T}e^1$ in the relations of reduction. From this, we can deduce some
relations between multizeta values. For all sequences $\seq{s} \in \mathcal{S}_{b,e}^\star$
of length $r$, we have:
\begin{center}
	$	\displaystyle	{	\sum	_{i = 1}
						^r
						\sum	_{	k_1 ,\cdots, \widehat{k_i},\cdots, k_r \geq 0
								\atop
								k_1 + \cdots + \widehat{k_i} + \cdots + k_r = s_i - 1
							}
							{}^i B_{\seq{k}}^{\seq{s}}
							\mathcal{Z}e^{s_1 + k_1, \cdots , s_{i - 1} + k_{i - 1}}
							\mathcal{Z}e^{s_r + k_r, \cdots , s_{i + 1} + k_{i + 1}}
					=
					0
				} \ .
	$
\end{center}

For instance, with $\seq{s} = (2,1,2)$, we obtain:
$	\left (
		\mathcal{Z}e^2
	\right )
	^2
	=
	2 \mathcal{Z}e^{2, 2} + 4 \mathcal{Z}e^{3,1}
	\ ,
$
which is a \symmetrality relation.

On the other hand, we can prove these relations between multizeta values in an independent way. Consequently, this will immediately show the exponentially flat character as a corollary and also will answer the question ``\textit{Are these relations consequences of the quadratic relations?}''
\\

Indeed, these relations are consequences of \symmetrality relations of multizetas values.
For all sequences $\seq{s} \in \mathcal{S}_{b,e}^\star$ of length $r$, let us denote:\\
$$	S(r) =		\displaystyle	{	\sum	_{i = 1}
							^r
							\sum	_{	k_1 ,\cdots, \widehat{k_i},\cdots, k_r \geq 0
									\atop
									k_1 + \cdots + \widehat{k_i} + \cdots + k_r = s_i - 1
								}
								{}^i B_{\seq{k}}^{\seq{s}}
								\mathcal{Z}e^{s_1 + k_1, \cdots , s_{i - 1} + k_{i - 1}}
								\mathcal{Z}e^{s_r + k_r, \cdots , s_{i + 1} + k_{i + 1}} \ .
					}
$$

We will show that $S(r) = 0$ for all sequences $\seq{s} \in \mathcal{S}_{b,e}^\star$, by linearization
of the product of multizeta values coming from the relations of 
symmetr\textbf{\textit {\underline {a}}}lity. Let us explain this in detail.

Each multizeta value can be written as an iterated integral. Writting the equality $(81)$ of \cite{Cartier}
like an iterated integral, we obtain for all sequences $\seq{s} \in \mathcal{S}_b^\star$:
$$	\begin{array}{@{}lll}
		\mathcal{Z}e^{\seq{s}}
		&=&
		\displaystyle	{	\int	_{0 < u_1 < \cdots < u_r < + \infty}
						\hspace{-1cm}
						\frac	{{u_1}^{s_1 - 1} (u_2 - u_1)^{s_2 - 1} \cdots {(u_r - u_{r - 1})}^{s_r - 1}}
							{(e^{u_1} - 1) \cdots (e^{u_r} - 1)}
						\frac	{du_1 \cdots du_r}	{	\displaystyle	{	\prod	_{i = 1}
															^r
															(s_i - 1) !
													}
										}
						}
		\ .
	\end{array}
$$

Thus, for all sequences $\seq{s} \in \mathcal{S}_{b,e}^\star$ of length $r$ and all
$i \in \crochet{1}{r}$, we have if we set $u_0 = 0$ and $u_{r + 1} = 0$:
\\

\noindent
$	\displaystyle	{	\sum	_{	k_1 ,\cdots, \widehat{k_i},\cdots, k_r \geq 0
									\atop
									k_1 + \cdots + \widehat{k_i} + \cdots + k_r = s_i - 1
								}
								{}^i B_{\seq{k}}^{\seq{s}}
								\mathcal{Z}e^{s_1 + k_1, \cdots , s_{i - 1} + k_{i - 1}}
								\mathcal{Z}e^{s_r + k_r, \cdots , s_{i + 1} + k_{i + 1}}
					}
$
\vspace{-0.1cm}	\\
\noindent
$	=
	\displaystyle	{	\hspace{-0.75cm}
				\sum	_{	k_1 ,\cdots, \widehat{k_i},\cdots, k_r \geq 0
						\atop
						k_1 + \cdots + \widehat{k_i} + \cdots + k_r = s_i - 1
					}
			}
						\begin{array}[t]{@{}l}
							\displaystyle	{	\frac	{(-1)^{s_{i + 1} + \cdots + s_r}}
											{	\displaystyle	{	\prod	_{p = 1}
																					^r
																								(s_p - 1)!
																					}
																}
														\cdot
														\frac	{(s_i - 1)!}
																{k_1! \cdots \widehat{k_i!} \cdots k_r!}
									}
							\vspace{0.15cm}
							\\
							\left (
								\displaystyle	{	\int	_{0 < u_1 < \cdots < u_{i - 1} < + \infty}
												\prod	_{p = 0}
													^{i - 2}
													\displaystyle	{	\frac	{	(u_{p + 1} - u_p)^{s_{p + 1} + k_{p + 1} - 1}
																											}
																											{	e^{u_{p+1}} - 1	}
																								}
																		du_1 \cdots du_{i - 1}
										}
							\right )
							\vspace{0.15cm}
							\\
							\left (
								\displaystyle	{	\int	_{0 < u_r < \cdots < u_{i + 1} < + \infty}
												\prod	_{p = i + 1}
													^{r}
																				\displaystyle	{	\frac	{	(u_p - u_{p + 1})^{s_p + k_p - 1}
																											}
																											{	e^{u_p} - 1	}
																								}
													du_{i + 1} \cdots du_r
										}
							\right ) ,
						\end{array}
$
\\
\noindent
$	=
	\displaystyle	{	\frac	{(-1)^{s_{i + 1} + \cdots + s_r}}
								{	\displaystyle	{	\prod	_{p = 1}
																^r
																(s_p - 1)!
													}
								}
						\int	_{\{0 < u_1 < \cdots < u_{i - 1} < + \infty\} \times \{0 < u_r < \cdots <u_{i + 1} < + \infty \}}
								\hspace{-2cm}
								g_i (u_1 ; \cdots ; u_r)
								 du_1 \cdots \widehat{du_i} \dots du_r \  ,
					}
$
\\
\noindent
where
$	g_i (u_1 ; \cdots ; u_r)
	=
	\displaystyle	{	\left (
								\prod	_{p \in \crochet{1}{r} - \{i\}}
										\frac	{(u_p - u_{p - 1})^{s_p - 1}}
												{e^{u_{p}} - 1}
						\right )
						( u_{i + 1} - u_{i - 1} )^{s_i - 1}
					} , 
$
always with $u_0 = 0$ and $u_{r + 1} = 0$~.
\\

To conclude that $S(r) = 0$, we have to compute the last integral.
An integral over a cartesian product like this integral is given by the sum of all the integrals over the domain
$\{0 < u_{\sigma(1)} < \cdots < u_{\sigma(r)} < + \infty\}$, where $\sigma$ is a permutation of $\crochet{1}{r}$ satisfying
$\sigma^{-1}(1) < \cdots < \sigma^{-1}(i - 1)$ and $\sigma^{-1}(r) < \cdots < \sigma^{-1}(i + 1)$. This set is exactly 
encoding by the shuffle product
$$u_1 \cdots u_{i - 1} \shuffle u_r \cdots u_{i + 1} \ .$$
The shuffle product is defined in Appendix \ref{Elements de calcul moulien debut}, \S \ref{AppendixSymmetrality}, 
from the syntatic point of view.

In order to compute $S(r)$, let us introduce a few notations. We now consider the alphabet $\Omega_r = \{u_1 ; \cdots ; u_r\}$, the non-commutative polynomial
$	e_r	=	\displaystyle	{	\sum	_{i = 1}
										^r
										(u_1 \cdots u_{i - 1}) \shuffle (u_r \cdots u_{i + 1})
							}
$ and the set $E_r$ of words of $\text{seq} (\Omega_r)$ which appears in $e_r$, that is to say in the linearization of the multizeta values we have to take in account:
$$	E_r
	=
	\{	\omega \in \text{seq} (\Omega_r)
		 \ ; \ 
		\left < e_r | \omega \right > \neq 0
	\}
$$ 
Finally, to each word $\omega = u_{i_1} \cdots u_{i_r}$ of $\text{seq} (\Omega_r)$ which contains exactly one time all the letters of $\text{seq} (\Omega_r)$ except one which will be denoted $u_i$, we associate an integral $I(\omega)$ defined by:
$$	I(\omega)	=	(-1)^{s_{i + 1} + \cdots + s_r}
					\displaystyle	{	\int	_{0 < u_{i_1} < \cdots < u_{i_r} < + \infty}
												g_i (u_1 ; \cdots ; u_r)
												du_1 \cdots \widehat{du_i} \cdots du_r
									} \ .
$$
\noindent
Thus:
$$	S(r) = \displaystyle	{	\frac	{1}
										{	\displaystyle	{	\prod	_{p = 1}
																		^r
																		(s_p - 1)!
															}
										}
								\left (
										\sum	_{\omega \in E_r}
												I(\omega)
								\right )
							}	\ .
$$

We will evaluate the sum, in the right-hand side, by grouping pairwise the elements of $E_r$~. $E_r$ will then be decomposed into a family of pairs of words we will call associated words.

\begin{Definition}
	Let us consider, for all $(k ; l) \in \crochet{1}{r}^2$, the morphism $\varphi_{k, l}$ from $\text{seq} (\Omega_r)$ (for the word concatenation) defined by:
	$$	\begin{array}{cccl}
			\varphi_{k, l} :&	\Omega_r	&	\longrightarrow	&	\Omega_r
			\\
								&	u_i		&	\longmapsto		&	\left \{
																	\begin{array}{l@{}l}
																		u_i	&	\text{ , if } i \neq l \ .	\\
																		u_k	&	\text{ , if } i = l \ .
																	\end{array}
																\right.
		\end{array}
	$$

	We will say that two words $\omega^1$ and $\omega^2$ of $\text{seq} (\Omega_r)$ are associated when:
	$$	\exists i \in \crochet{1}{r} ,  
		\omega^2 = \varphi_{i, i + 1} (\omega^1)
		\text{  or  }
		\omega^1 = \varphi_{i, i + 1} (\omega^2) \ .
	$$
	
	We then write: $\omega^1 \between \omega^2$\ .
\end{Definition}

Let $\widetilde{\text{seq} (\Omega_r)}$ be the set of words of $\text{seq} (\Omega_r)$ which
contain exactly one time all the letters of $\Omega_r$, except one. One can notice that two
words of $\widetilde{\text{seq} (\Omega_r)}$ can of course have a different missing letter.
Finally, we associate a permutation $\sigma_{\omega}$ of $\crochet{1}{r} - \{i\}$ with each
word $\omega = u_{s_1} \cdots u_{s_{r - 1}}$ of $\widetilde{\text{seq} (\Omega_r)}$, where
$i$ is the index of the absent letter of $\omega$, defined by:
$$	\sigma_{\omega} =	\left (
								\begin{array}{cccccc}
									1	&	\cdots	&	i - 1		&	i + 1	&	\cdots	&	r			\\
									s_1	&	\cdots	&	s_{i - 1}	&	s_{i}	&	\cdots	&	s_{r - 1}	\\
								\end{array}
						\right )
	\ .
$$

We have then:

\begin{Lemma}
	\begin{enumerate}
		\item For all $\omega \in E_r$ and all integer $i \in \crochet{1}{r - 1}$ , we have:
			$$	\sigma_{\varphi_{i, i + 1}(\omega)}
				=
				\rho_{i, i + 1} \circ \sigma_{\omega} \circ {\rho_{i, i + 1}}^{-1} \ ,
			$$ with	$	\begin{array}[t]{cccl}
								\rho_{i, i + 1} :	&	\crochet{1}{r} - \{i\}	&	\longrightarrow	&	\crochet{1}{r} - \{i + 1\}	\\
													&	k						&	\longmapsto		&	\left \{
																												\begin{array}{l@{}l}
																													k	&	\text{ , if } k \neq i + 1 \ .
																													\\
																													i	&	\text{ , if } k = i + 1 \ .
																												\end{array}
																										\right.
							\end{array}
						$
		\item For all $\omega \in E_r$ , there exists a unique $\omega' \in E_r - \{\omega\}$ such that: $\omega \between \omega'$~.
		\item For all $(\omega^1 ; \omega^2) \in {E_r}^2$ , we have :
			$\omega^1 \between \omega^2  \Longrightarrow  I(\omega^1) = - I(\omega^2)$~.
	\end{enumerate}
\end{Lemma}

\begin{Proof}
	1. Let $\omega \in E_r$ and $i$ the index of the missing letter in $\omega$.
	\\
	\\
	We will distinguish two different cases, depending on whether $\sigma_{\omega} (i + 1)$ equals $i + 1$ or not.
	\\
	\\
	\underline{First case:} $\sigma_{\omega} (i + 1) = i + 1$.
	\\
	\\
	The word $\varphi_{i, i + 1}(\omega)$ can be written:
	$\varphi_{i, i + 1}(\omega) = \omega^{< i - 1} u_i \omega^{> i + 1}$. Remind that the notations $\omega^{< i - 1}$ and $\omega^{> i + 1}$ denote
	respectively the first and the last terms in the sequence $\omega$ (see appendix \ref{Elements de calcul moulien debut}, \S \ref{AppendixNotations}).
	Thus: 
	$$	\sigma_{\varphi_{i, i + 1}(\omega)}
		=
		\left (
			\begin{array}{cccccccc}
				1					&	\cdots	&	i - 1					&	i	&	i + 2					&	\cdots	&	r					\\
				\sigma_{\omega}(1)	&	\cdots	&	\sigma_{\omega}(i - 1)	&	i	&	\sigma_{\omega}(i + 2)	&	\cdots	&	\sigma_{\omega}(r)	\\
			\end{array}
		\right ) \ .
	$$
	We only have to compute $\rho_{i, i + 1} \circ \sigma_\omega \circ {\rho_{i, i + 1}}^{-1}$ to conclude that:
	$$	\left \{
			\begin{array}{l@{}l}
				\rho_{i, i + 1} \circ \sigma_{\omega} \circ {\rho_{i, i + 1}}^{-1} (k)
				=
				\rho_{i, i + 1} \circ \sigma_{\omega} (k)
				=
				\sigma_{\omega} (k)
				&\text{ , if } k \in \crochet{1}{r} - \{i, i + 1\} \ .
				\\
				\rho_{i, i + 1} \circ \sigma_{\omega} \circ {\rho_{i, i + 1}}^{-1} (i) = i\ .
			\end{array}
		\right.
	$$
	\\
	Indeed, with $k \in \crochet{1}{r} - \{i \,;i + 1\} \ , \ \sigma_{\omega} (k) \neq i + 1$~.	\\
	\\
	\underline{Second case:} $\sigma_{\omega} (i + 1) \neq i + 1$.
	\\
	\\
	This case is something like the first one if we denote $j = \sigma_{\omega}^{-1} (i + 1)$~.
	Let us write $\varphi_{i, i + 1}(\omega) = \omega^{< i} \cdot \omega^{i \leq  \cdot  < j} \cdot u_i \cdot \omega^{> j}$
	or $\varphi_{i, i + 1}(\omega) = \omega^{< j} \cdot u_i \cdot \omega^{j <  \cdot  < i} \cdot \omega^{\geq i}$, depending if
	$i + 1 < j$ or $i + 1 > j$.
	\\
	\\
	Thus, we have:	$	\sigma_{\varphi_{i, i + 1}(\omega)}
						=
						\rho_{i, i + 1} \circ \sigma_\omega \circ {\rho_{i, i + 1}}^{-1}
					$.
	\\
	\\
	2.	Let $\omega \in E_r$ and $i$ the index of the missing letter in the word $\omega$.
	\\
		Since no letters appears twice in $\omega$, the associated words to $\omega$ in $E_r$ can only be
		$\omega' = \varphi_{i, i + 1} (\omega)$ or $\omega'' = \varphi_{i - 1, i} (\omega)$.
	\\
	\\
		If $i = 1$, $\omega'$ is the only well-defined candidate. More precisely, we have $\omega = u_r \cdots u_2$
		and $\omega' = u_r \cdots u_3 u_1$.
		Thus, $\omega'$ is a shuffle of $u_1$ and $u_r \cdots u_3$ ; consequently, it is an element of $E_r$ .
	\\
	\\
		If $i \in \crochet{2}{r - 1}$, since $\omega$ is a shuffle of $u_1 \cdots u_{i - 1}$ and $u_r \cdots u_{i + 1}$,
		according to the definition of a shuffle of two sequences, we have:
		$$	\sigma_{\omega}^{-1} (1) < \cdots < \sigma_{\omega}^{-1} (i - 1)
			\ \text{ and } \ 
			\sigma_{\omega}^{-1} (r) < \cdots < \sigma_{\omega}^{-1} (i + 1) \ .	\\
		$$
		\noindent
		Moreover	$	\sigma_{\varphi_{i, i + 1}(\omega)}
						=
						\rho_{i, i + 1} \circ \sigma_\omega \circ {\rho_{i, i + 1}}^{-1}
					$ and	$	\sigma_{\varphi_{i, i - 1}(\omega)}
								=
								\rho_{i, i - 1} \circ \sigma_\omega \circ {\rho_{i, i - 1}}^{-1}
							$. So, we obtain that:
		$$	\left \{
				\begin{array}{@{}l}
					\sigma_{\omega'}^{-1} (1) < \cdots < \sigma_{\omega'}^{-1} (i - 1)	\\
					\sigma_{\omega'}^{-1} (r) < \cdots < \sigma_{\omega'}^{-1} (i + 2) < \sigma_{\omega'}^{-1} (i)	\\
				\end{array}
			\right.
		$$
		and
		$$	\left \{
				\begin{array}{@{}l}
					\sigma_{\omega''}^{-1} (1) < \cdots < \sigma_{\omega''}^{-1} (i - 2) < \sigma_{\omega''}^{-1} (i)\ .	\\
					\sigma_{\omega''}^{-1} (r) < \cdots < \sigma_{\omega''}^{-1} (i + 1) \ .	\\
				\end{array}
			\right.
		$$
		Consequently:
		$$	\left \{
				\begin{array}{lllll}
					\omega' \in E_r	
					&\iff&
					\sigma_{\omega'}^{-1} (i - 1) < \sigma_{\omega'}^{-1} (i)	&\iff& \sigma_{\omega}^{-1} (i - 1) < \sigma_{\omega}^{-1} (i + 1) \ .
					\\
					\omega'' \in E_r
					&\iff&
					\sigma_{\omega''}^{-1} (i + 1) < \sigma_{\omega''}^{-1} (i)	&\iff& \sigma_{\omega}^{-1} (i + 1) < \sigma_{\omega}^{-1} (i - 1) \ .
				\end{array}
			\right.
		$$
		Therefore, $\omega$ has one, and only one, associated word in $E_r$.
		\\
		\\
		Finally, if $i = r$, like in the first case, $\omega''$ is the only well-defined candidate:
		$\omega'' = u_1 \cdots u_{r - 2} u_r \in E_r$ . Since $\omega = u_1 \cdots u_{r - 1}$, we have $\omega \between \omega''$ .
		\\
		\\
		3.	Let $\omega^1 = u_{k_1} \cdots u_{k_{r - 1}}$ and $\omega^2 = u'_{k_1} \cdots u'_{k_{r - 1}}$ two associated words in $E_r$ .
		Let us denote by $i$ the index of the missing letter of $\omega^1$ and $j = \sigma_{\omega^1}^{-1} (i + 1)$. Thus,
		$\omega^2 = {\omega^1}^{<j} u_i {\omega^1}^{>j}$. Let us also denote
		$$	\left \{
				\begin{array}{l}
					D(\omega^1) = \{0 < u_{k_1} < \cdots < u_{k_{j - 1}} < u_{i + 1} < u_{k_{j + 1}} < u_{k_{r - 1}} < + \infty \} \ .
					\\
					D(\omega^2) = \{0 < u_{k_1} < \cdots < u_{k_{j - 1}} < u_i < u_{k_{j + 1}} < u_{k_{r - 1}} < + \infty \} \ .
				\end{array}
			\right.
		$$
		Performing the change of variable's names $u_{i + 1} \leftrightarrow u_i$ in the integral  $I(\omega^1)$, we obtain:
		\\	
		$	\begin{array}[t]{@{}lll}
				I(\omega^1)
				&=&
				\displaystyle	{	\frac	{(-1)^{s_{i + 1} + \cdots + s_r}}
											{	\displaystyle	{	\prod	_{p = 1}
																		^r
																			(s_i - 1)!
																}
											}
									\int	_{D(\omega^1)}
											g_i (u_1 ; \cdots ; u_r)
											\ du_1 \cdots \widehat{du_i} \dots du_r
								}
				\vspace{0.25cm}
				\\
				&=&
				\displaystyle	{	\frac	{(-1)^{s_{i + 1} + \cdots + s_r}}
											{	\displaystyle	{	\prod	_{p = 1}
																			^r
																			(s_i - 1)!
																}
										}
									\int	_{D(\omega^2)}
											(-1)^{s_{i + 1} - 1} g_{i + 1} (u_1 ; \cdots ; u_r)
												\ du_1 \cdots \widehat{du_{i + 1}} \dots du_r
								}
				\vspace{0.25cm}
				\\
				&=&
				- I(\omega^2) \ .
			\end{array}
		$
		\qed
\end{Proof}

Since $E_r$ has $2^{r - 1}$ elements counted with their multiplicity in $E_r$, we can conclude, from the second statement of the previous lemma that $E_r$ can be cut into $2^{r - 2}$ pairs of associated words in $E_r$.

Thus, according to the third statement of the lemma, we finally deduce that $S(r) = 0$. This can be written as:

\begin{Property}
	\begin{enumerate}
		\item The  exponentially flat character of multitangent functions implies some relations between multizeta values which are coming from the \symmetrality
				relations of the multizeta values.
		\item The \symmetrality relations of the multizeta values impose the absence of the monotangent $\mathcal{T}e^1$ in the relation of reduction into monotangent functions, and thus force the multitangent functions to be exponentially flat.
	\end{enumerate}
\end{Property}

Consequently, one can ask the following question:
\\
``\textit	{	Using to the exponentially flat character of convergent multitangent functions,
			are we able to find all the \symmetrality relations between multizeta values? If the
			answer is negative, which relations do we obtain?
		}''
\\

The answer is simple and comes from a rapid exploration of the table of multitangent functions.
We immediately see that, in weight $5$, we find all the \symmetrality relations of multizeta values
of weight $4$, but this situation is really exceptional. For instance, in weight $6$, one \symmetrality
relation is not obtained:
\begin{equation}	\label{relation zetaique}
	3\mathcal{Z}e^{2, 2, 1}+6\mathcal{Z}e^{3, 1, 1}+\mathcal{Z}e^{2, 1, 2}
	=
	\mathcal{Z}e^{2, 1}\mathcal{Z}e^{2}  .
\end{equation}

Nevertheless, considering all $\Q$-linear relations between multitangent functions, we are able
to find all the \symmetrality relations. To illustrate this, one can find $(\ref{relation zetaique})$
from
\begin{equation}
	4 \mathcal{T}e^{3, 1, 3} - 2 \mathcal{T}e^{3, 1, 1, 2} + \mathcal{T}e^{2, 1, 2, 2} = 0 \ ,
\end{equation}
using the reduction into monotangent as well as the values of the multizeta values given for instance
from \cite{table de MZV}.
\section	{Analytic properties}
\label{analytiques}

As announced in Section \ref{annonce caractere exp plat}, we now will see that each multitangent
function tends to $0$ when $\im z$ goes to infinity at least as an exponential function. This is
the so-called exponentially flat character of convergent multitangent functions. In order to study
the convergence of series involving multitangent functions, we look for an upper bound depending on
the weight of the sequence $\seq{s}$ and which also shows us the exponentially flat character.

To obtain such an upper bound, we will have to avoid the use of the triangular inequality, 
which is not so precise. Consequently, we want to use directly an upper bound on the sum. First, we will
focus on Fourier coefficients of multitangents ; then, we will deal with geometric upper
bounds of multitangent functions in order to obtain upper bound of the Fourier coefficients.
Finally, using this, we obtain an upper bound as required.

\subsection{Fourier expansion of convergent multitangent functions}
\label{Fourier et multitangentes}

Since multitangent functions are $1$-periodic on $\C - \Z$, we are naturally interested in their
Fourier expansions. The result proved here is central for the explicit computation of analytical
invariants of tangent-to-identity diffeomorphisms (see \cite{Bouillot})

Let us first remind the Fourier expansion of $\mathcal{T}e^1$~(see~\cite{Serre})~:
\begin{equation}	\label{dvp de fourier de te1}
	\displaystyle	{	\mathcal{T}e^1 (z)	=	\frac	{\pi}
									{\tan (\pi z)}
							=
			}
	\left \{
		\begin{array}{l@{}l}
			\phantom{-}i \pi + 2i \pi \displaystyle	{	\sum	_{n < 0}
										e^{2i n \pi z}
								} 
			&
			\text{ , if } \im z < 0 \ .
			\\
			- i \pi - 2i \pi \displaystyle	{	\sum	_{n > 0}
									e^{2i n \pi z}
							} 
			&
			\text{ , if } \im z > 0 \ .
		\end{array}
	\right.
\end{equation}

Since the convergence of the right hand side is normal on the set 			\linebreak
$\{\zeta \in \C  ; \im \zeta < -c\}$ and $\{\zeta \in \C  ; \im \zeta > c\}$, for all $c > 0$
the expression $(\ref{dvp de fourier de te1})$ is the Fourier expansion of $\mathcal{T}e^1$~.
The differentiation property gives us for all $\sigma \in \N-\{0 ;1\}$ and all $z \in \C - \Z$:

$$	\mathcal{T}e^\sigma (z)
	=
	\left \{
			\begin{array}{l@{}l}
				\displaystyle	{	\phantom{-}2i\pi
									\sum	_{n < 0}
											\frac	{(-2i n \pi)^{\sigma - 1}}
													{(\sigma - 1)!}
											e^{2i n \pi z}
								} 
				&
				\text{ , if } \im z < 0 \ .
				\vspace{0.2cm}
				\\
				\displaystyle	{	- 2i\pi
									\sum	_{n > 0}
											\frac	{(-2i n \pi)^{\sigma - 1}}
													{(\sigma - 1)!}
											e^{2i n \pi z}
								} 
				&
				\text{ , if } \im z > 0 \ .
			\end{array}
	\right.
$$

Plugging this Fourier expansion in the expression for the reduction into monotangents (see \textsection \ref{reduction en monotangente})
when $\seq{s} \in \mathcal{S}^\star$ , we obtain:	\\
$	\begin{array}{@{}lll}
		\mathcal{T}e^{\underline{s}} (z)
		&=&
		\displaystyle	{	\sum	_{j = 1}
									^r
									\sum	_{k = 2}
											^{s_j}
											\mathcal{Z}_{j,s_j - k}^{\seq{s}}
											\mathcal{T}e^k (z)
						}
		\vspace{0.2cm}
		\\
		&=&
		\left \{
			\begin{array}{l@{}l}
				\displaystyle	{	\phantom{-} 2i\pi
									\sum	_{n < 0}
											\sum	_{j = 1}
													^r
													\left (
															\displaystyle	{	\sum	_{k = 2}
																						^{s_j}
																						\frac	{(-2i n \pi)^{k - 1}}
																								{(k - 1)!}
																						\mathcal{Z}_{j,s_j - k}^{\seq{s}}
																			}
													\right )
													e^{2i n \pi z}
								}
				&
				\text{ , if } \im z < 0 \ .
				\vspace{0.2cm}
				\\
				\displaystyle	{	- 2i\pi
									\sum	_{n > 0}
											\sum	_{j = 1}
													^r
													\left (
															\displaystyle	{	\sum	_{k = 2}
																						^{s_j}
																						\frac	{(-2i n \pi)^{k - 1}}
																								{(k - 1)!}
																						\mathcal{Z}_{j,s_j - k}^{\seq{s}}
																			}
													\right )
													e^{2i n \pi z}
								}
				&
				\text{ , if } \im z > 0 \ .
			\end{array}
		\right.
	\end{array}
$
\\

Since the convergence of this series is normal on $\{\zeta \in \C  ; \im \zeta < -c\}$ and $\{\zeta \in \C  ; \im \zeta > c\}$, for all $c > 0$
we obtain the Fourier expansion of the multitangent functions:

\begin{Lemma}
	Let us set\footnotemark, for all $n \in \Z$ and all $\seq{s} \in \mathcal{S}^\star_{b,e}$~:
	$$	\displaystyle	{	\widehat{\mathcal{T}}_n^{\seq{s}}
							=
							2i \pi
							\sum	_{j = 1}
									^{l(\seq{s})}
									\left (
											\displaystyle	{	\sum	_{k = 2}
																		^{s_j}
																		\frac	{(-2i n \pi)^{k - 1}}
																				{(k - 1)!}
																		\mathcal{Z}_{j , s_j - k}^{\seq{s}}
															}
									\right )
						} \ .
	$$
	Then, for all sequences $\seq{s} \in \mathcal{S}^\star_{b,e}$ and all $z \in \C - \Z$, we have:
	$$	\mathcal{T}e^{\seq{s}} (z)
		=
		\left \{
			\begin{array}{l@{}l}
				\displaystyle	{	\phantom{-}
									\sum	_{n < 0}
											\widehat{\mathcal{T}}_n^{\seq{s}}  q^n
								}
			&	
				\text { , if } \im z < 0\ ,
				\\
				\displaystyle	{	-
									\sum	_{n > 0}
											\widehat{\mathcal{T}}_n^{\seq{s}}  q^n
								}
			&
				\text { , if } \im z > 0\ ,
			\end{array}
		\right.
	$$
	where $q = e^{2 \pi i z}$~.
\end{Lemma}

\footnotetext	{	Let us remark that the mould $\widehat{\mathcal{T}}^{\p}_n$ can not be a \symmetrel one.
					For example, we have: 
					\begin{center}
						$	\left \{
									\begin{array}{l}
										\displaystyle	{	2 \widehat{\mathcal{T}}^{2,2}_1 + \widehat{\mathcal{T}}^{4}_1
															=
															2 \times \frac{4}{3} \pi^4 - \frac{8}{3} \pi^4 = 0 \ .
														}
										\\
										\left ( \widehat{\mathcal{T}}^{2}_1 \right )^2 = 16 \pi^4  .
									\end{array}
							\right.
						$
					\end{center}
					This explains the absence of the letter $e$ in its name.
				}

\subsection	{An upper bound for multitangent functions}

In this paragraph, we will prove two geometric upper bounds (or nearly geometric ones), where the exponent will be the weight of the multitangent and then give two hypothetical upper bounds. For this, we will use elementary methods.

\label{estimation_des_multitangentes}

We begin, for a convergent multitangent function, by proving the following upper bound:

\begin{Lemma} \label {lemma_estimation_des_multitangentes}
	For all sequences $\seq{s} \in \mathcal{S}^\star_{b,e}$ and all $z \in \C - \R$, we have :
	$$	\left |
				\mathcal{T}e^{\seq{s}} (z)
		\right |
		\leq
		\displaystyle	{	\frac	{4 l(\seq{s})}	{|\im z|^{||\seq{s}|| - l(\seq{s}) - 1}}
						} \ .
	$$
\end{Lemma}

\begin{Proof}
			Let $\seq{s} \in \mathcal{S}^\star_{b,e}$ and $z \in \C - \R$~.
			\\
			We will denote by $fe^{\seq{s}}$ the function defined on $\R \times \R^*_+$ by:
			$$	fe^{\seq{s}} (x ; y)
				=
				\displaystyle	{	\sum	_{- \infty < n_r < \cdots < n_1 < + \infty}
											\frac	{1}
													{	\big (
																(n_1 + x)^2 + y^2
														\big )
														^{\frac{s_1}{2}}
														\cdots
														\big (
																(n_r + x)^2 + y^2
														\big )
														^{\frac{s_r}{2}}
													} \ .
								}
			$$
			We hence have:
			$	\left |
						\mathcal{T}e^{\seq{s}}(z)
				\right |
				\leq
				fe^{\seq{s}} (\re z ; |\im z|)
			$. Moreover, using an argument we will develop in a forthcoming section (see Section \ref{factorisation moulienne}), we obtain the following
			trifactorisation:
			$$ fe^\p (x ; y) = fe_+^\p(x ; y) \times Ie^\p(x ; y) \times fe_-^\p(x ; y) \ ,$$
			where the functions $fe_+^\p$ , $fe_-^\p$ and $Ie^\p$ are defined on $\R \times \R_+^*$ by:

			$$	\begin{array}{lll}
				fe_+^{\seq{s}}(x ; y)
				&=&
				\displaystyle	{	\sum	_{- E(x) < n_r < \cdots < n_1 < + \infty}
											\frac	{1}
													{	\left (
																(n_1 + x)^2 + y^2
														\right )
														^{\frac{s_1}{2}}
														\cdots
														\left (
																(n_r + x)^2 + y^2
														\right )
														^{\frac{s_r}{2}}
													} \ .
								}
				\\
				\\
				fe_-^{\seq{s}}(x ; y)
				&=&
				\displaystyle	{	\sum	_{- \infty < n_r < \cdots < n_1 < - E(x)}
											\frac	{1}
													{	\left (
																(n_1 + x)^2 + y^2
														\right )
														^{\frac{s_1}{2}}
														\cdots
														\left (
																(n_r + x)^2 + y^2
														\right )
														^{\frac{s_r}{2}}
													} \ .
								}
				\\
				\\
				Ie^{\seq{s}} (x ; y)
				&=&
				\left \{
					\begin{array}{c@{}l}
						0				&	\text{ , if } l(\seq{s}) \neq 1 \ .
						\\
						\left (
								(x - E(x))^2 + y^2
						\right )
						^{-\frac{s}{2}}	&	\text{ , if } l(\seq{s}) = 1 \ .
					\end{array}
				\right.
				\end{array}
			$$
			
			Thus, we successively have:	\\
			$	\begin{array}{@{}l@{}l@{}l}
					fe_+^{\seq{s}}(x ; y)
					& =&
					\displaystyle	{	\sum	_{0 < n_r < \cdots < n_1 < + \infty}
												\frac	{1}
														{	\left (
																	(n_1 + x - E(x))^2 + y^2
															\right )
															^{\frac{s_1}{2}}
															\cdots
															\left (
																	(n_r + x - E(x))^2 + y^2
															\right )
															^{\frac{s_r}{2}}
														}
									}
					\vspace{0.2cm}
					\\
					& \leq &
					\displaystyle	{	 \frac	{1}
												{(y^2)^{\frac{s_1 - 2}{2}} (y^2)^{\frac{s_2 - 1}{2}} \cdots (y^2)^{\frac{s_r - 1}{2}}}
										 
										fe_+^{2,1, \cdots , 1} (x - E(x) ; 0)
									}
					\vspace{0.2cm}
					\\
					& \leq &
					\displaystyle	{	 \frac	{1}
												{y^{s_1 - 2} y^{s_2 - 1} \cdots y^{s_r - 1}}
										\sum	_{0 < n_r < \cdots < n_1 < + \infty}
												\frac	{1}
														{	{n_1}^2
															n_2 \cdots n_r
														}
									}
					\vspace{0.2cm}
					\\
					& = &
					\displaystyle	{	 \frac	{\mathcal{Z}e^{2,1, \cdots , 1}}
												{y^{||\seq{s}|| - l(\seq{s}) - 1}}
									}
					\ \leq \ 
					\displaystyle	{	 \frac	{2}
												{y^{||\seq{s}|| - l(\seq{s}) - 1}}
									} \ .		\label{majoration a affiner}
				\end{array}
			$
			\\
			On the same way, we have:
			$	fe_-^{\seq{s}}(x ; y)
				\leq
				\displaystyle	{	\frac	{2}
											{y^{||\seq{s}|| - l(\seq{s}) - 1}}
								}
			$\ .
			Hence:\\
			$	\begin{array}{@{}lll}
					\left |
							\mathcal{T}e^{\seq{s}}(z)
					\right |
					&\leq&
					\displaystyle	{	\sum	_{k = 1}
												^{l(\seq{s})}
												fe_+^{\seq{s}^{< k}} (\re z ; |\im z|)
												Ie^{s_k} (\re z ; |\im z|)
												fe_-^{\seq{s}^{> k}} (\re z ; |\im z|)
									}
					\vspace{0.2cm}
					\\
					&\leq&
					\displaystyle	{	\sum	_{k = 1}
												^{l(\seq{s})}
												\frac	{2}
														{|\im z|^{||\seq{s}^{< k}|| - (k - 1) - 1}}
												\times
												\frac	{1}
														{|\im z|^{s_k}}
												\times
												\frac	{2}
														{|\im z|^{||\seq{s}^{> k}|| - (l(\seq{s}) - k) - 1}}
									}
					\vspace{0.2cm}
					\\
					&=&
					\displaystyle	{	\sum	_{k = 1}
												^{l(\seq{s})}
												\frac	{4}
														{|\im z|^{||\seq{s}|| - l(\seq{s}) - 1}}
									}
					=
					\displaystyle	{	\frac	{4 l(\seq{s})}
												{|\im z|^{||\seq{s}|| - l(\seq{s}) - 1}}
									} \ .
				\end{array}
			$\\
			\qed
\end{Proof}

Next, we present the second upper bound. It is an improvement of the first one when we restrict to nonempty sequences of $\text{seq} (\N_2)$~.
The proof uses the same notations and also the same ideas as for the first upper bound.

\begin{Lemma}	\label{lemma_estimation_des_multitangentes_valuation_2}
	For all $\seq{s} \in \text{seq} (\N_2) - \{\emptyset\}$ and all $z \in \C - \R$ satisfying $|z| \geq 1$, we have:
	\label{majoration des multitangentes de valuations superieure a 2}
	$	\left |
				\mathcal{T}e^{\seq{s}} (z)
		\right |
		\leq
		\displaystyle	{	\frac	{1}	{l(\seq{s}) !}
							\left (
									\frac	{2}	{\sqrt{|\im z|}}
							\right )^{||\seq{s}||}
						}.
	$
\end{Lemma}

\begin{Proof}
			Let $\seq{s} \in \text{seq} (\N_2)$ be a sequence of length $r$ and $z \in \C - \R$ satisfying $|z| \geq 1$~.
			Let us also consider, as in the previous proof, the notations $fe^{\seq{s}}$ , $fe_+^{\seq{s}}$ , $fe_-^{\seq{s}}$ and $Ie^{\seq{s}}$~.
			\\
			\\
			When $\seq{s} \in \text{seq} (\N_2) \subset \mathcal{S}^\star_{b,e}$ , we will improve the upper bounds which have just been found in the
			proof of the first upper bound by using an integral test for convergence:
		
			$	\begin{array}{@{}lll}
					fe_+^{\seq{s}}(x ; y)
					&=&
					\displaystyle	{	\sum	_{0 < n_r < \cdots < n_1 < + \infty}
												 
												\prod	_{i = 1}
														^r
														\displaystyle	{	\left (
																					\frac	{1}
																					{	\left (
																								(n_i + x - E(x))^2 + y^2
																						\right )
																						^{\frac{s_i}{2}}
																					}
																			\right )
																		}
									}
					\vspace{0.2cm}	\\
					&=&
					\displaystyle	{	\frac	{1}	{r!}
										\sum	_{	(n_1 ; \cdots ; n_r) \in (\N^*)^r
													\atop
													i \neq j \Longrightarrow n_i \neq n_j
												}
												\prod	_{i = 1}
														^r
														\displaystyle	{	\left (
																					\frac	{1}
																							{	\left (
																										(n_i + x - E(x))^2 + y^2
																								\right )
																								^{\frac{s_i}{2}}
																								}
																			\right )
																		}
									}
					\vspace{0.2cm}	\\
					&\leq&
					\displaystyle	{	\frac	{1}	{r!}
										\prod	_{i = 1}
												^r
												\left (
														\displaystyle	{	\sum	_{n \in \N^*}
																					\frac	{1}
																							{	\left (
																										(n + x - E(x))^2 + y^2
																								\right )
																								^{\frac{s_i}{2}}
																							}
																		}
												\right )
									}
					\vspace{0.2cm}	\\
					&\leq&
					\displaystyle	{	\frac	{1}	{r!}
										\prod	_{i = 1}
												^r
												\left (
														\displaystyle	{	\int	_0
																					^{+ \infty}
																					\frac	{dt}
																							{	\left (
																										(t + x - E(x))^2 + y^2
																								\right )
																								^{\frac{s_i}{2}}
																							}
																		}
												\right )
									}
					\vspace{0.2cm}	\\
					&\leq&
					\displaystyle	{	\frac	{1}	{r!}
										\frac	{1}	{y^{||\seq{s}|| - 2r}}
										\left (
												\int	_{0}
														^{+ \infty}
														\displaystyle	{	\frac	{du}
																					{u^2 + y^2}
																		}
										\right )
										^r
									}
					\ \leq \ 
					\displaystyle	{	\frac	{1}	{r!}
										\frac	{	\left (
															\frac	{\pi}	{2}
													\right )
													^r
												}
												{y^{||\seq{s}|| - r}}
									} \ .
				\end{array}
			$
			\\
			\\
			On the same way, we have:
			$	fe_-^{\seq{s}}(x;y)
				\leq
				\displaystyle	{	\frac	{1}	{r!}
									\frac	{	\left (
														\frac	{\pi}	{2}
												\right )
												^r
											}
											{y^{||\seq{s}|| - r}}
								}
			$ .
			Hence:
			\vspace{-0.65cm}
			$$	\begin{array}[t]{lll}
					\left |
							\mathcal{T}e^{\seq{s}}(z)
					\right |
					&\leq&
					\displaystyle	{	\sum	_{k = 1}
												^r
												fe_+^{\seq{s}^{< k}} (\re z ; |\im z|)
												Ie^{s_k} (\re z ; |\im z|)
												fe_-^{\seq{s}^{> k}} (\re z ; |\im z|)
									}
					\\
					\vspace{-0.35cm}
					&\leq&
					\displaystyle	{	\sum	_{k = 1}
												^r
												 \begin{array}[t]{l}
														\displaystyle	{	\frac	{1}	{(k - 1)!}
																			\frac	{	\left (
																								\displaystyle	{\frac	{\pi}	{2}}
																						\right )
																						^{k - 1}
																					}
																					{|\im z|^{||\seq{s}^{< k}|| - (k - 1)}}
																			\times
																			\frac	{1}
																					{|\im z|^{s_k - 1}}
																			\times
																		}
														\\
														\displaystyle	{	\frac	{1}	{(r - k)!}
																			\frac	{	\left (
																								\displaystyle	{\frac	{\pi}	{2}}
																						\right )
																						^{r - k}
																					}
																					{|\im z|^{||\seq{s}^{> k}|| - (r - k)}}
																		}
													\end{array}
									}
				\end{array}
			$$
			$$	\begin{array}[t]{lll}
					\phantom	{	\left |
											\mathcal{T}e^{\seq{s}}(z)
									\right |
								}
					&\leq&
					\displaystyle	{	\frac	{	\left (
															\frac	{\pi}	{2}
													\right )
													^{r - 1}
												}
												{|\im z|^{||\seq{s}|| - r}}
										\sum	_{k = 1}
												^r
												\frac	{1}	{(k - 1)!}
												\frac	{1}	{(r - k)!}
									}
					\leq
					\displaystyle	{	\frac	{2^{r - 1}}	{(r - 1)!}
										\frac	{	\left (
															\frac	{\pi}	{2}
													\right )
													^{r - 1}
												}
												{|\im z|^{||\seq{s}|| - r}}
									}
					\vspace{0.2cm}	\\
					&\leq&
					\displaystyle	{	\frac	{\pi^{r - 1}}	{(r - 1)!}
										\frac	{1}		{|\im z|^{||\seq{s}|| - r}}
									}
					\leq
					\displaystyle	{	\frac	{4^{r}}	{r!}
										\frac	{1}		{|\im z|^{||\seq{s}|| - r}}
									}
					\ .
				\end{array}
			$$
			
			To conclude, we only have to notice that, for $\seq{s} \in \text{seq} (\N_2)$ , we have $||\seq{s}|| \geq 2 l(\seq{s})$~. Consequently, we deduce
			the sought upper bound:
			$$	\left |
						\mathcal{T}e^{\seq{s}}(z)
				\right |
				\leq
					\displaystyle	{	\frac	{1}	{r!}
										\left (
												\frac	{2}		{\sqrt{|\im z|}}
										\right )
										^{||\seq{s}||}
									} .
			$$
			\qed
\end{Proof}

\subsection	{About the exponentially flat character}
\label {caractere exponentiellement plat}

The exponentially flat character of convergent multitangent functions is a consequence of Schwarz's lemma. To enlighten this, let us give the following statement:

\begin{Lemma}	\label{Consequence_of_chwarz_Lemma}
	Suppose that a function $f$ is $1$-periodic, holomorphic on the half-plane $\{ \zeta \in \C \,; \im \zeta > 0\}$ and satisfies
	$\displaystyle	{	\lim	_{t \longrightarrow + \infty}
								f(it)
						=
						0
					}
	$~.
	\\
	Then, for every $c > 0$, we have:
	$$	\im z > c
		\Longrightarrow
		|f(z)| \leq M(c) e^{-2\pi \im z} \ ,
	$$
	where	$M(c) = \displaystyle	{	e^{2 \pi c} \sup	_{z \in \R + i c}
															|f(z)|
									} \ .
			$
\end{Lemma}

\begin{proof}
	Let us fix $c > 0$.	\\
	Writing $f(z) = F(e^{2i\pi z})$, where $F$ is a holomorphic function defined over $\{ \zeta \in \C \,;\, |\zeta| < 1\}$ and $F(0) = 0$, we can
	consider
	$$	\varphi (z) = \displaystyle	{	\frac	{F(e^{-2\pi c} z)}	{M}
									}
		\ ,
	$$
	where	$\displaystyle	{	M	= \sup	_{z \in \R + i c}
											|f(z)|
									= \sup	_{z \in \{\zeta \in \C \,; |\zeta| < e^{-2 \pi c} \}}
											|F(z)|
							} \ .
			$
	\\
	
	Since $\varphi$ satisfies the Schwarz's lemma hypothesis, we can conclude that:
	$$	|\varphi(z)| \leq |z| \ ,
		\text{ for all } z \in \C \text{ such that } |z| < e^{-2 \pi c}\ .
	$$
	Consequently, we have proved:
	$$	|F(z)| \leq M e^{2\pi c} |z| \text{ , for all } z \in \C \text{ such that } |z| < e^{-2 \pi c} \ ,
	$$
	which conclude the proof.
\end{proof}

From Lemma \ref{lemma_estimation_des_multitangentes} and \ref{lemma_estimation_des_multitangentes_valuation_2}, we have:
$$	\displaystyle	{	\sup	_{z \in \R + i c}
								\left |
										\mathcal{T}e^{\seq{s}} (z)
								\right |
					}
	\leq
	\left \{
			\begin{array}{lll}
				\displaystyle	{	\frac	{4l(\seq{s})}	{c^{||\seq{s}|| - l(\seq{s}) - 1}}	}
				&\text{, if } \seq{s} \in \mathcal{S}_{b,e}^\star \ .
				\\
				\displaystyle	{	\frac	{1}	{l(\seq{s})!}
									\left (
											\frac	{2}	{\sqrt{c}}
									\right )
									^{||\seq{s}||}
								}
				&\text{, if } \seq{s} \in \text{seq}(\N_2) \ .
			\end{array}
	\right.
$$

Moreover, we have	$	\displaystyle	{	e^{-2x}
											\leq
											\frac	{1}	{e^{2x} - 1}
											\leq
											\frac	{4}	{\sh^2 (x)}
										}
					$, provided $x > 0$~.

For a given $z \in \C - \R$, setting $c = \displaystyle	{	\frac	{|\im z|}	{2}	}$, we deduce from Lemma \ref{Consequence_of_chwarz_Lemma}:

\begin{Property}
	\begin{enumerate}
		\item For all sequences $\seq{s} \in \mathcal{S}_{b,e}^\star$ and $z \in \C - \R$, we have:
				$$	\left |
							\mathcal{T}e^{\seq{s}} (z)
					\right |
					\leq
					\displaystyle	{	\left (
												\frac	{2 }
														{|\im z|}
										\right )
										^{||\seq{s}|| - l(\seq{s}) - 1}
										\frac	{16 \pi l(\seq{s})}	{\sh^2 \left (	\displaystyle	{	\frac	{\pi |\im z|}	{2}	}	\right )	}
									}\ .
				$$
		\item For all sequences $\seq{s} \in \text{seq}(\N_2)$ and $z \in \C - \R$, we have:
				$$	\left |
							\mathcal{T}e^{\seq{s}} (z)
					\right |
					\leq
					\displaystyle	{	\frac	{1}	{l(\seq{s})!}
								\left (
									\frac	{2 \sqrt{2}}
										{\sqrt{|\im z|}}
								\right )
								^{||\seq{s}||}
								\frac	{4 \pi}	{\sh^2 \left (	\displaystyle	{	\frac	{\pi |\im z|}	{2}	}	\right )	}
							}\ .
				$$
	\end{enumerate}
\end{Property}
\section	{Study of a \symmetrel extension of multitangent functions to $\text{seq}(\N^*)$}
\label{prolongement des multitangentes au cas divergent}

In this section, our aim is to provide a regularization of $\mathcal{T}e^{\seq{s}}$ when the series \eqref{defMTGF} is a divergent one, which is when
$\seq{s} \in \text{seq}(\N^*) - \mathcal{S}^\star_{b,e}$, i.e. when $s_1 = 1$ or $s_r = 1$, as well as when $s_1 = s_r = 1$...

Moreover, we want the extension to satisfy the same properties as the convergent multitangent functions (see Property \ref{firstProperties} and Theorem \ref{reduction en monotangente}). So, we must preserve:

\begin{enumerate}
	\item The \symmetrel character.
	\item The differentiation property.
	\item The parity property.
	\item The property of reduction into monotangent functions.
\end{enumerate}

Note that we have anticipated this section in \S \ref{Unit cleansing of divergent multitangent functions} when we have studied the unit cleansing of  multitangent functions.

\subsection	{A generic method to extend the definition of a \symmetrel mould}

In this section, we consider a \symmetrel mould $\mathcal{S}e^\p$ over the alphabet $\Omega = \N^*$, with values in an algebra $\mathbb{A}$, which is well-defined for sequences in $S_b^\star = \{ \seq{s} \in \text{seq} (\N^*) \ ; \ s_1 \geq 2 \}$~. We want to define an extension of $\mathcal{S}e^\p$ for all sequences of $\text{seq} (\N^*)$ such that the `new' mould $\mathcal{S}e^\p$ is again a \symmetrel one.

The following lemma is due to Jean Ecalle. The first part is now well-known while the second was not published. To be exhaustive, we shall prove both points.

\begin{Lemma}	\label{extension lemma}
	\begin{enumerate}
		\item For all $\theta \in \mathbb{A}$ , there exists a unique \symmetrel extension of $\mathcal{S}e^\p$ to $\text{seq} (\N^*)$, denoted by
				$\mathcal{S}e^{\p}_{\theta}$ , such that $\mathcal{S}e_{\theta}^1 = \theta$~.
		\item For all $\gamma \in \mathbb{A}$ , let $\mathcal{N}e_{\gamma}^\p$ be the \symmetrel mould defined on sequences of $\text{seq}(\N^*)$ by:
			$	\mathcal{N}e^{\seq{s}}_{\gamma} =	\left \{
															\begin{array}{cl}
																\displaystyle	{	\frac	{\gamma^r}	{r!}	}	&	\text{, if } \seq{s} = 1^{[r]}\ .
																\\
																0													&	\text{, otherwise.}
															\end{array}
													\right.
			$
			\\
			Then, for all $(\theta_1 \,;\, \theta_2) \in \mathbb{A}^2$ , we have:
			$$	\mathcal{S}e_{\theta_1}^{\p}
				=
				\mathcal{N}e^{\p}_{\theta_1 - \theta_2}
				\times
				\mathcal{S}e_{\theta_2}^{\p}
				\ .
			$$
	\end{enumerate}
\end{Lemma}

\begin{Proof}
	$1$.	If such an extension $\mathcal{S}e^\p$ exists, it must satisfy the shift to the right of the ones begining an evaluation sequence. 
	In other words, the following identities must be valid for all $k \in \N$ and all sequences $\seq{s} \in \mathcal{S}_b^\star$:
	\begin{equation}	\label{algorithmic removal...}
		(k + 1) \mathcal{S}e^{1^{[k + 1]} \cdot \seq{s}}
		\ =\ 
		\displaystyle	{	\mathcal{S}e^{1} \mathcal{S}e^{1^{[k]} \cdot \seq{s}}
							\ -
							\sum	_{	\seq{u}
										\in
										sh\text{\textit{\textbf{\underline{e}}}} \, (1 \,;\, 1^{[k]} \cdot \seq{s}) - \{1^{[k + 1]} \cdot \seq{s}\}
									}
									\mathcal{S}e^{\seq{u}}
						}
		\ \ .
	\end{equation}
	By induction on the number of ones that begin such sequences, these identities implies the uniqueness of the extension $\mathcal{S}e^\p$ to 
	$\text{seq}(\N^*)$~.
	\\

	To prove the existence of such an extension of $\mathcal{S}e^\p$, we define $\mathcal{S}e_\theta^{1^{[k]} \cdot \seq{s}}$ recursively by:
	$$	(k + 1) \mathcal{S}e_{\theta}^{1^{[k + 1]} \cdot \seq{s}}
		\ =\ 
		\displaystyle	{	\theta \mathcal{S}e_{\theta}^{1^{[k]} \cdot \seq{s}}
							\ -
							\sum	_{	\seq{u}
										\in
										sh\text{\textit{\textbf{\underline{e}}}} \, (1 \,;\, 1^{[k]} \cdot \seq{s}) - \{1^{[k + 1]} \cdot \seq{s}\}
									}
									\mathcal{S}e_{\theta}^{\seq{u}}
						}
		\text{ , where } k \in \N \ .
	$$

	The only point to check is the \symmetrelity of $\mathcal{S}e^\p_\theta$~. This may be done by induction on $k + l$, where $k$
	and $l$ denote the number of ones beginning the first and second sequences in the product:\\
	$	\begin{array}{@{}ll}
			(k + 1) \mathcal{S}e_{\theta}^{1^{[k + 1]} \cdot \seq{s}^1} \mathcal{S}e_{\theta}^{1^{[l]} \cdot \seq{s}^2}
			\\
			&	\hspace{-3.5cm}
				\begin{array}{@{}ll}
					=&
					\displaystyle	{	\theta \mathcal{S}e_{\theta}^{1^{[k]} \cdot \seq{s}^1} \mathcal{S}e_{\theta}^{1^{[l]} \cdot \seq{s}^2}
										\ -
										\sum	_{	\seq{u}
													\in
													sh\text{\textit{\textbf{\underline{e}}}} \, (1 \,;\, 1^{[k]} \cdot \seq{s}^1) - \{1^{[k + 1]} \cdot \seq{s}^1\}
												}
												\mathcal{S}e_{\theta}^{\seq{u}} \mathcal{S}e_{\theta}^{1^{[l]} \cdot \seq{s}^2}
									}
					\vspace{0.1cm}	\\
					=&
					\displaystyle	{	\theta
										\sum	_{	\seq{u}
													\in
													sh\text{\textit{\textbf{\underline{e}}}} \, (1^{[k]} \cdot \seq{s}^1 , 1^{[l]} \cdot \seq{s}^2)
												}
												\mathcal{S}e_{\theta}^{\seq{u}}
										\ -
										\sum	_{	\seq{u}
													\in
													sh\text{\textit{\textbf{\underline{e}}}} \, (1 \,;\, 1^{[k]} \cdot \seq{s}^1) - \{1^{[k + 1]} \cdot \seq{s}^1\}
												}
												\ 
												\sum	_{	\seq{u'}
															\in
															sh\text{\textit{\textbf{\underline{e}}}} \, (\seq{u} ; 1^{[l]} \cdot \seq{s}^2)
														}
														\mathcal{S}e_{\theta}^{\seq{u'}}
									}
					\vspace{0.1cm}	\\
					=&
					\displaystyle	{	\sum	_{	\seq{u}
													\in
													sh\text{\textit{\textbf{\underline{e}}}} \, (1^{[k]} \cdot \seq{s}^1 ; 1^{[l]} \cdot \seq{s}^2)
												}
												\ 
												\sum	_{	\seq{u'}
															\in
															sh\text{\textit{\textbf{\underline{e}}}} \, (\seq{u} ; 1)
														}
														\mathcal{S}e_{\theta}^{\seq{u}'}
										\ -
										\sum	_{	\seq{u}
													\in
													sh\text{\textit{\textbf{\underline{e}}}} \, (1 \,;\, 1^{[k]} \cdot \seq{s}^1)
												}
												\ 
												\sum	_{	\seq{u'}
															\in
															sh\text{\textit{\textbf{\underline{e}}}} \, (\seq{u} ; 1^{[l]} \cdot \seq{s}^2)
														}
														\mathcal{S}e_{\theta}^{\seq{u'}}
									}
					\vspace{0.2cm}	\\
					&
					\displaystyle	{	\ +
										(k + 1)
										\sum	_{	\seq{u}
													\in
													sh\text{\textit{\textbf{\underline{e}}}} \, (1^{[k + 1]} \cdot \seq{s}^1 ; 1^{[l]} \cdot \seq{s}^2)
														}
														\mathcal{S}e_{\theta}^{\seq{u'}}
									}
					\vspace{0.1cm}	\\
					=&
					\displaystyle	{	\sum	_{	\seq{u}
													\in
													sh\text{\textit{\textbf{\underline{e}}}} \, (1^{[k]} \cdot \seq{s}^1 ; 1^{[l]} \cdot \seq{s}^2 ; 1)
												}
												\mathcal{S}e_{\theta}^{\seq{u}}
										\ -
										\sum	_{	\seq{u}
													\in
													sh\text{\textit{\textbf{\underline{e}}}} \, (1 \,;\, 1^{[k]} \cdot \seq{s}^1)
												}
												\ 
												\sum	_{	\seq{u'}
															\in
															sh\text{\textit{\textbf{\underline{e}}}} \, (\seq{u} ; 1^{[l]} \cdot \seq{s}^2)
														}
														\mathcal{S}e_{\theta}^{\seq{u'}}
									}
					\vspace{0.2cm}	\\
					&
					\displaystyle	{	\ +
										(k + 1)
										\sum	_{	\seq{u}
													\in
													sh\text{\textit{\textbf{\underline{e}}}} \, (1^{[k + 1]} \cdot \seq{s}^1 ; 1^{[l]} \cdot \seq{s}^2)
														}
														\mathcal{S}e_{\theta}^{\seq{u'}}
									}
					\vspace{0.1cm}	\\
					=&
					\displaystyle	{	\sum	_{	\seq{u}
													\in
													sh\text{\textit{\textbf{\underline{e}}}} \, (1^{[k]} \cdot \seq{s}^1 ; 1^{[l]} \cdot \seq{s}^2 ; 1)
												}
												\mathcal{S}e_{\theta}^{\seq{u}}
										\ -
										\sum	_{	\seq{u}
													\in
													sh\text{\textit{\textbf{\underline{e}}}} \, (1 \,;\, 1^{[k]} \cdot \seq{s}^1 ; 1^{[l]} \cdot \seq{s}^2)
												}
												\mathcal{S}e_{\theta}^{\seq{u'}}
									}
					\vspace{0.1cm}	\\
					&
					\displaystyle	{	+
										(k + 1)
										\sum	_{	\seq{u}
													\in
													sh\text{\textit{\textbf{\underline{e}}}} \, (1^{[k + 1]} \cdot \seq{s}^1 ; 1^{[l]} \cdot \seq{s}^2)
												}
												\mathcal{S}e_{\theta}^{\seq{u'}}
									}
					\vspace{0.1cm}	\\
					=&
					\displaystyle	{	(k + 1)
										\sum	_{	\seq{u}
													\in
													sh\text{\textit{\textbf{\underline{e}}}} \, (1^{[k + 1]} \cdot \seq{s}^1 ; 1^{[l]} \cdot \seq{s}^2)
												}
												\mathcal{S}e_{\theta}^{\seq{u'}}
									} \ ,
				\end{array}
		\end{array}
	$
	where we have used the recursive definition of $\mathcal{S}e_{\theta}^{\p}$ in the first and third equality, the inductive step in the second one and finally
	the associativity and commutativity of the stuffle product\footnotemark \ in the last three equalities.
	\footnotetext	{	Let us remind that the stuffle product is the product recursively defined on
						words constructed over the alphabet $\Omega$ having a semi-group structure. For 
						two words $P = p_1 \cdots p_r$ and $Q = q_1 \cdots q_s$ constructed over the alphabet
						$\Omega$, we have:
						$$	\left \{
									\begin{array}{@{}l@{}}
											P \star \varepsilon
											=
											\varepsilon \star P
											=
											P \ \ .
											\\
											P \star Q
											=
											p_1 \big( p_2 \cdots p_r \star Q \big)
											+
											q_1 \big( P \star q_2 \cdots q_s \big)
											+
											(p_1 \cdot q_1) \big ( p_2 \cdots p_r \star q_2 \cdots q_s \big )
											\ .
									\end{array}
							\right.
						$$
						See the appendix \ref{Elements de calcul moulien debut} about mould calculus.
					}
	\\

	This concludes the proof of the first point and allows us to consider $\mathcal{S}e^\p_\theta$ for all $\theta \in \C$~.
	\\

	$2$. It is clear that for all $\theta \in \C$, $\mathcal{N}e_\theta$ is a \symmetrel mould.
	Moreover, by construction, $\mathcal{S}e_{\theta_1}^\p$ and $\mathcal{S}e_{\theta_2}^\p$ are
	also \symmetrel. Consequently, the two moulds
	$\mathcal{N}e_{\theta_2 - \theta_1}^\p \times \mathcal{S}e_{\theta_1}^\p$
	and $\mathcal{S}e_{\theta_2}^\p$ are defined over $\text{seq}(\N^*)$, \symmetrel and their
	evaluation on the sequence $1$ are equal to $\theta_2$. According to the first point, we
	therefore have 
	$	\mathcal{S}e_{\theta_1}^\p
		=
		\mathcal{N}e_{\theta_2 - \theta_1}^\p \times \mathcal{S}e_{\theta_1}^\p
	$~, which ends the proof of the lemma.
	\qed
\end{Proof}

As a direct application, there exists a unique \symmetrel extension of $\mathcal{Z}e^\p$ to $\text{seq} (\N^*)$ such that $\mathcal{Z}e^1 = 0$~. From now on to the end of this section, $\mathcal{Z}e^\p$ will denote this extension, especially in \S \ref{FHmfafmf}.

\subsection	{Trifactorization of $\mathcal{T}e^\p$ and consequences}

The extension of multitangent functions to the divergent case is more complicated than the case exposed
in the previous section. Actually, even if we fix $\mathcal{T}e^1$, one
cannot apply the shift to the right of the ones beginning an evaluation sequence because the
ones beginning or ending the sequences will be sent respectively at the end or the beginning of the
sequences. To illustrate this, we have:

$$	\label{calcul_multitangente_divergente_la_plus_simple}
	\mathcal{T}e^{1,2} (z) =	\underbrace {\mathcal{T}e^1(z)}	_{	\text{known by}
																	\atop
																	\text{hypothesis}
																}
								\ \cdot
								\underbrace {\mathcal{T}e^2(z)}	_{	\text{convergent}	
																	\atop
																	{	\text{multitangent}
																		\atop
																		\text{function}
																}	}
								-
								\underbrace {\mathcal{T}e^{2,1}(z)}	_{	\text{problem}
																		\atop
																		= \text{ unknown}
																	}
								-
								\underbrace {\mathcal{T}e^3(z)}	_{	\text{convergent}	
																	\atop
																	{	\text{multitangent}
																		\atop
																		\text{function}
																}	}
								\ .
$$

The difficulty here comes from the joint management of the two sources of divergence created at $- \infty$ and $+ \infty$~. To overcome it, we will separate the divergence at $- \infty$ from that at $+ \infty$~. To this end, we will use a mould factorization in which each term has only one source of divergence. Let us remind that we have already used and proved such a factorization when we have proved the convergence rule for the multitangent functions (see p.~\pageref{definition_des_moules_symetrEls2})~, but we give here a separate statement because of its importance:

\begin{Lemma}
	\label{factorisation moulienne}
			Let us consider the \symmetrel moulds $\mathcal{H}e_+^\p$, $\mathcal{H}e_-^\p$ and $\mathcal{C}e^\p$, with values in holomorphic functions over
			$\C - \Z$, defined by:
			\vspace{-0.15cm}
			$$	\begin{array}{lll}
					\mathcal{H}e_+^{\seq{s}} (z)
					&=&
					\hspace{-0.7cm}
					\displaystyle	{	\sum	_{0 < n_r < \cdots < n_1 < + \infty}
												\frac	{1}	{(n_1 + z)^{s_1} \cdots (n_r + z)^{s_r}}
									}
					\text { and }
					\mathcal{H}e_+^{\emptyset} (z) = 1 \ , \  \seq{s} \in \mathcal{S}_b^\star \ .
					\\
					\vspace{-0.5cm}
					\\
					\mathcal{H}e_-^{\seq{s}} (z)
					&=&
					\hspace{-0.7cm}
					\displaystyle	{	\sum	_{- \infty < n_r < \cdots < n_1 < 0}
												\frac	{1}	{(n_1 + z)^{s_1} \cdots (n_r + z)^{s_r}}
									}
					\text { and }
					\mathcal{H}e_-^{\emptyset} (z) = 1 \ , \ \seq{s} \in \mathcal{S}_e^\star\ .
					\\
					\vspace{-0.5cm}
					\\
					\mathcal{C}e^{\seq{s}} (z)
					&=&
					\hspace{-0.5cm}
					\left \{
							\begin{array}{cll}
								1								&\text{ if }&	\seq{s} = \emptyset
								\\
								\displaystyle	{\frac{1}{z^s}}	&\text{ if }&	l(\seq{s}) = 1
								\\
								0								&\text{ if }&	l(\seq{s}) > 1
							\end{array}
					\right. 
					\ , \ \seq{s} \in \text{seq}(\N^*)\ .
				\end{array}
			$$
			Then	$$\mathcal{T}e^\p = \mathcal{H}e^\p_+ \times \mathcal{C}e^\p \times \mathcal{H}e_-^\p \ .$$
\end{Lemma}

The moulds $\mathcal{H}e_+^\p$ and $\mathcal{H}e_+^\p$ are called Hurwitz multizeta functions; $\mathcal{C}e^\p$ is going to play the role of a correction because it only appears for null summation indexes in the expression of $\mathcal{T}e^\p$~. Let us remark that it is clear from Lemma
\ref{definition_des_moules_symetrEls1} and their definition that these three moulds are \symmetrel, which justifies the letter \textbf{\textit{\underline{e}}} in their names.

First of all, we mention that this trifactorization acts as we wanted: it separates the divergence sources. Secondly, it is now clear that it is sufficient to extend (with the \symmetrelity property) the definition of the two moulds of Hurwitz multizeta functions to $\text{seq}(\N^*)$ in order to obtain an extension of $\mathcal{T}e^\p$ to $\text{seq}(\N^*)$ which is also \symmetrel~. For this purpose, Lemma \ref{extension lemma} can be applied: given $(\Phi_+ ; \Phi_-) \in \mathcal{H} (\C - \Z)^2$, the moulds $\mathcal{H}e_+^\p$ and $\mathcal{H}e_-^\p$ admit a unique extension to $\text{seq}(\N^*)$ such that $\mathcal{H}e_+^1 = \Phi_+$ and $\mathcal{H}e_-^1 = \Phi_-$~. 

We complete Figure \ref{figure deux de liens entre MZV et MTGF} with the following diagram, in which we define
$$	\mathcal{H}MZV_{CV,\pm} = \text{Vect}_{\mathcal{M}ZV_{CV}}	\left (
										\mathcal{H}e_+^{\seq{s}^1}
										\mathcal{H}e_-^{\seq{s}^2}
									\right )
									_{	\seq{s}^1 \in \mathcal{S}^\star_{b}
										\atop
										\seq{s}^2 \in \mathcal{S}^\star_{e}
									} \ .
$$
From the trifactorisation, we see that $\mathcal{M}TGF_{CV}$ can be embedded in $\mathcal{H}MZV_{CV,\pm}$, which will be indicated by a curly arrow in the diagram. Recall that an arrow indicates a link between two algebras, while an arrow in dotted lines indicates a hypothetical link.

\begin{figure}[h]	\label{Figure trois de liens entre MZV et MTGF}
	\centering
	$	\xymatrix	{		\mathcal{M}ZV_{CV}	\ar@{.>}@<4pt>[ddd]	^{\text{projection}}	&&&		\mathcal{H}MZF_{+,CV}	\ar[lll]	_{\text{evaluation at 0}}
																															\ar@{^(->}[ddd]
							\\
							\\
							\\
							\mathcal{M}TGF_{CV}	\ar[uuu]			^{reduction}
												\ar@{^(->}[rrr]		^{\text{trifactorization}}	&&&		\mathcal{H}MZF_{\pm, CV}
					}
	$
	\caption	{Links between multizeta values and multitangent functions}
\end{figure}
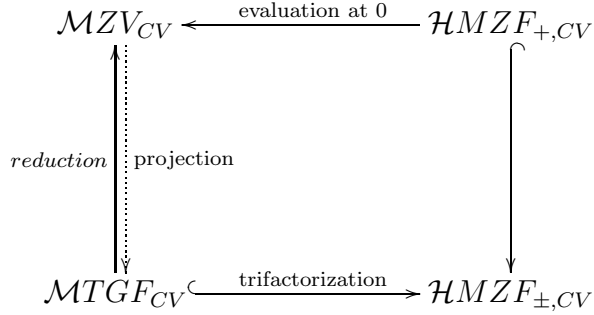

As a consequence, we obtain

\begin{Corollary} \label{corollary - definition extension}
		Let $\Phi_+$ and $\Phi_-$ be two holomorphic functions over $\C - \Z$~.	\\
		The mould $\mathcal{T}e^\p$ admits a \symmetrel extension to $\text{seq} (\N^*)$ such that:
		$$	\left \{
					\begin{array}{l}
						\forall z \in \C - \Z \ , \, \mathcal{T}e^1 (z) = \Phi_+ (z) + \displaystyle{\frac{1}{z}} + \Phi_-(z) \ .
						\vspace{0.1cm}	\\
						\forall \seq{s} \in \text{seq} (\N^*) \ ,  \,
						\mathcal{T}e^{\seq{s}} = (\mathcal{H}e^\p_+ \times \mathcal{C}e^\p \times \mathcal{H}e_-^\p)^{\seq{s}} \ .
					\end{array}
			\right.
		$$
\end{Corollary}

\subsection	{Formal Hurwitz multizeta functions and formal multitangent functions}
\label{FHmfafmf}
In order to simplify the following proof, we will work in the ring of formal power series by introducing the notion of formal Hurwitz multizeta functions and formal multitangent functions. To distinguish whether we are working analytically or formally without specifying it, we use two different notations. The formal character will be denoted by a straight capital letter while the analytic character will be denoted by a cursive capital letter (as we have always done from the beginning) . 

\subsubsection{The mould $He_+^\p(X)$}
\label{defintionHe_+}
Since the Hurwitz multizeta functions $\mathcal{H}e_+^{\seq{s}}$ are regular near $0$, we have:
\begin{Lemma}
	The Taylor series of $\mathcal{H}e^{\seq{s}}$, for all $\seq{s} \in \mathcal{S}^\star_{b,e}$, is given by:
	$$	He_+^{\seq{s}} (X) = \displaystyle	{	\sum	_{k \geq 0}
														\sum	_{	k_1 , \cdots , k_r \geq 0
																	\atop
																	k_1 + \cdots + k_r = k
																}
																\displaystyle	{	\prod	_{i = 1}
																							^r
																							\binom{s_i + k_i - 1}{k_i}
																				}
																\mathcal{Z}e^{s_1 + k_1, \cdots , s_r + k_r}
																(-X)^k
											} \ . 
	$$
\end{Lemma}

We now define formal Hurwitz multizeta functions by their Taylor expansions near $0$ and $He_+^\emptyset (X) = 1$:
$$	He_+^{\seq{s}} (X) = \displaystyle	{	\sum	_{k_1 , \cdots , k_r \geq 0}
															\displaystyle	{	\prod	_{i = 1}
																						^r
																						\binom{s_i + k_i - 1}{k_i}
																			}
															\mathcal{Z}e^{s_1 + k_1, \cdots , s_r + k_r}
															(-X)^{k_1 + \cdots + k_r}
										} \ . 
$$

Thus, $He_+^\p (X)$ is a \symmetrel mould defined over $\mathcal{S}^\star_b$ , with values in $\C [\![ X ]\!]$~.
According to Lemma \ref{extension lemma}, for all $S \in \C [\![ X ]\!]$ , $He^\p$ has a unique \symmetrel extension to $\text{seq}(\N^*)$ such that
$He^1_+ (X) = S(X)$~. We have now to define in a suitable manner $He_+^1 (X)$~. Since we extended $\mathcal{Z}e^\p$ with $\mathcal{Z}e^1 = 0$, we can set:
\begin{equation} \label{defHe+1}
	He_+^1 (X) = \displaystyle	{	\sum	_{k = 1}
											^{+ \infty}
											\mathcal{Z}e^{k + 1} (-X)^k
								} \ .
\end{equation}

This definition is natural because we want the differentiation property to be satisfied by the extension of $\mathcal{H}e^\p_+$, as well as its formal analogue. Consequently, the analytic analogue of $He^1_+$ is
$$	\mathcal{H}e_+^1 (z)
	=
	\sum	_{n \geq 1}
			\left (
					\frac	{1}	{n + z}
					-
					\frac	{1}	{n}
			\right )
	\ .
$$

Let us remind that $\mathcal{Z}e^\p$ is the unique extension of the mould $\mathcal{Z}e^\p$ to $\text{seq}(\N^*)$ satisfying $\mathcal{Z}e^1 = 0$~. So
we obtain:

\begin{Property}	\label{ExtensionOfHe+}
	The unique \symmetrel extension of $He_+^\p(X)$ to $\text{seq}(\N^*)$ satisfying \eqref{defHe+1} is given by:
	$$	He_+^{\seq{s}} (X) = \displaystyle	{	\sum	_{k_1 , \cdots , k_r \geq 0}
														\displaystyle	{	\prod	_{i = 1}
																					^r
																					\binom{s_i + k_i - 1}{k_i}
																		}
														\mathcal{Z}e^{s_1 + k_1, \cdots , s_r + k_r}
														(-X)^{k_1 + \cdots + k_r}
											} \ .
	$$
\end{Property}

\begin{Proof}
	By uniqueness of the moulds satisfying these properties, it is sufficient to prove that the mould $\widetilde{He}_+^{\seq{s}} (X)$
	defined by the right hand-side satisfies:
	\begin{enumerate}
		\item $\widetilde{He}_+^{\seq{s}} (X)$ extends the definition of $He_+^{\seq{s}} (X)$ to $\text{seq} (\N^*)$~.
		\item $\widetilde{He}_+^1 (X) = He_+^1 (X)$~.
		\item $\widetilde{He}_+^\p (X)$ is a \symmetrel mould.
	\end{enumerate}

	The third point is the only one requiring some explanations. Let us denote by $M_k^\p$ the coefficient of $X^k$ of
	$\widetilde{He}_+^{\seq{s}} (X)$. In order to prove the \symmetrelity of $\widetilde{He}_+^{\seq{s}} (X)$, we will show that:
	$$	\forall (\seq{s}^1 ; \seq{s}^2) \in \big (\text{seq} (\N^*) \big )^2,  \forall p \in \N ,  
		\displaystyle	{	\sum	_{k = 0}
									^p
									M_k^{\seq{s}^1} M_{p - k}^{\seq{s}^2}
							=
							\sum	_{\seq{\gamma} \in sh\text{\textbf{\underline{\textit{e}}}} (\seq{s}^1 ; \seq{s}^2)}
									M_p^{\seq{\gamma}}
						} .
	$$

	For $(\seq{s}^1 ; \seq{s}^2) \in \big ( \text{seq} (\N^*) \big )^2$ and $p \in \N$, we have, if we define $r = l(\seq{s}^1)$ and $r' = l(\seq{s}^2)$:
	\\
	\noindent
	$	\begin{array}{@{}l@{}l@{}l}
			\displaystyle	{	\sum	_{k = 0}
										^p
										M_k^{\seq{s}^1} M_{p - k}^{\seq{s}^2}
							}
			& =  &
			\displaystyle	{	\sum	_{k_1 + \cdots + k_{r + r'} = k}
										\begin{array}[t]{l}
												\left (
														\displaystyle	{	\prod	_{i = 1}
																					^r
																					\binom{s_i^1 + k_i - 1}{k_i}
																		}
												\right )
												\left (
														\displaystyle	{	\prod	_{i = 1}
																					^{r'}
																					\binom{s_i^2 + k_{i + r} - 1}{k_{i + r}}
																		}
												\right )
												\vspace{0.1cm}	\\
												\times
												\mathcal{Z}e^{s_1^1 + k_1, \cdots , s_r^1 + k_r}
												\mathcal{Z}e^{s_1^2 + k_{r + 1}, \cdots , s_{r'}^2 + k_{r'}}
										\end{array}
							}
			\vspace{0.1cm}	\\
			& =  &
			\displaystyle	{	\sum	_{k_1 + \cdots + k_{r + r'} = k}
										\begin{array}[t]{l}
												\left (
														\displaystyle	{	\prod	_{i = 1}
																					^r
																					\binom{s_i^1 + k_i - 1}{k_i}
																		}
												\right )
												\left (
														\displaystyle	{	\prod	_{i = 1}
																					^{r'}
																					\binom{s_i^2 + k_{i + r} - 1}{k_{i + r}}
																		}
												\right )
												\vspace{0.1cm}	\\
												\times
												\displaystyle	{	\Big (
																			\sum	_{\seq{\pmb{\gamma}} \in sh\text{\textbf{\underline{\textit{e}}}} (\seq{s}^1 + \seq{k}^{\leq r} ; \seq{s}^2 + \seq{k}^{> r})}
																					\mathcal{Z}e^{\seq{\pmb{\gamma}}}
																	\Big )
																} \ ,
										\end{array}
							}
		\end{array}
	$
	where $\seq{s}^1 + \seq{k}^{\leq r} \text{ and } \seq{s}^2 + \seq{k}^{> r}$ respectively denote
	$(s_1^1 + k_1 ; \cdots ; s_r^1 + k_r)$ and 																		\linebreak
	$(s_1^2 + k_{r + 1} ; \cdots ; s_{r'}^2 + k_{r + r'})$~.
	Two cases are possible:
	\begin{enumerate}
		\item $\seq{\pmb{\gamma}}$ is a shuffle of $\seq{s}^1 + \seq{k}^{\leq r}$ and $\seq{s}^2 + \seq{k}^{> r}$:
		\item[] Then, we can reorder if necessary the $k_i$'s such that the resulting term is $M_p^{\widetilde{\seq{\pmb{\gamma}}}}$~, where
				$\widetilde{\seq{\pmb{\gamma}}}$ is deduced from $\seq{\pmb{\gamma}}$ by setting $k_i = 0$ for all $i$.
		\item $\seq{\pmb{\gamma}}$ contains one or more contractions of $\seq{s}^1 + \seq{k}^{\leq r}$ and $\seq{s}^2 + \seq{k}^{> r}$:
		\item[] We can separate indexes which do not act on contractions from the other ones. Let us denote these by $\seq{s}_i^1$ and $\seq{s}_j^2$.
				We then obtain some sums of binomial coefficients:
				$$	\displaystyle	{	\sum	_{k_i + k_j = K}
												\binom{s_i^1 + k_i - 1}{k_i}
												\binom{s_j^1 + k_j - 1}{k_j}
										=
										\binom{s_i^1 + s_j^2 + K - 1}{K}
									} \ ,
				$$ which is exactly the expected binomial coefficient.
		\item[] Thus, the term expected is again $M_p^{\widetilde{\seq{\pmb{\gamma}}}}$~, where $\widetilde{\seq{\pmb{\gamma}}}$ is deduced from 
					$\seq{\pmb{\gamma}}$ by cancelling all the $k_i$'s.
	\end{enumerate}
			
	Let us remark that $\seq{\widetilde{\pmb{\gamma}}}$ runs over the set $sh\text{\textbf{\underline{\textit{e}}}} (\seq{s}^1 ; \seq{s}^2)$ when
	$\seq{\pmb{\gamma}}$ runs over the set $sh\text{\textbf{\underline{\textit{e}}}} (\seq{s}^1 + \seq{k}^{\leq r} ; \seq{s}^2 + \seq{k}^{> r})$~.
	We deduce from this that:
	$$	\displaystyle	{	\sum	_{k = 0}
									^p
									M_k^{\seq{s}^1} M_{p - k}^{\seq{s}^2}
							=
							\sum	_{\seq{\pmb{\gamma}} \in sh\text{\textbf{\underline{\textit{e}}}} (\seq{s}^1 ; \seq{s}^2)}
									M_p^{\seq{\pmb{\gamma}}}
						} \ .
	$$
	Consequently, $\widetilde{He}_+^{\seq{s}} (X)$ is a \symmetrel mould and is equal to $He_+^\p (X)$~.
			
	\qed
\end{Proof}

\subsubsection{The mould $He_-^\p(X)$}

Property \ref{ExtensionOfHe+} can be adapted to $He_-^\p (X)$. Thus, we can extend this mould to $\text{seq}(\N^*)$ 
for all sequences $\seq{s} \in \text{seq} (\N^*)$ by:
$$	\begin{array}{@{}lll}
		He_-^{\seq{s}} (X)
		&=&
		(-1)^{||\seq{s}||} He_+^{\overset {\leftarrow} {\underline{s}}} (-X)
		\\
		&=&
		\displaystyle	{	\sum	_{k_1 , \cdots , k_r \geq 0}
									\displaystyle	{	\hspace{0.2cm}
														\prod	_{i = 1}
																^r
																\binom{s_i + k_i - 1}{k_i}
													}
									\mathcal{Z}e_-^{s_1 + k_1, \cdots , s_r + k_r}
									(-X)^{k_1 + \cdots + k_r}
						} \ ,
	\end{array}
$$
where $\mathcal{Z}e_-^\p$ is defined from $\mathcal{Z}e^\p$ by a pseudo-parity relation:\label{definition_du_moule_ze_moins}

$$	\begin{array}{lll}
		\mathcal{Z}e_-^{s_1, \cdots, s_r}
		&=&
		(-1)^{||\seq{s}||}
		\mathcal{Z}e^{s_r, \cdots, s_1}
		=
		\displaystyle	{	\sum	_{0 < n_1 < \cdots < n_r}
									\frac	{(-1)^{||\seq{s}||}}	{{n_1}^{s_1} \cdots {n_r}^{s_r}}
						}
		\\
		&=&
		\displaystyle	{	\sum	_{p_r < \cdots < p_1 < 0}
									\frac	{1}	{{p_r}^{s_1} \cdots {p_1}^{s_r}}
						} \ .
	\end{array}
$$

Implicitly, this relation forces:
$$		\mathcal{H}e_-^1 (z)
		=
		\displaystyle	{	\sum	_{n < 0}
									\left (
											\frac	{1}	{n + z}
											-
											\frac	{1}	{n}
									\right )
						}
		\ .
$$

\subsubsection{The mould $Te^\p (X)$} 

We have seen in Corollary \ref{corollary - definition extension}
that for $\Phi_+$ and $\Phi_-$ two holomorphic function over $\C - \Z$, there exists a \symmetrel extension of $\mathcal{T}e^\p$ to $\text{seq} (\N^*)$, 
defined by $\mathcal{T}e^\p = \mathcal{H}e_+^\p \times \mathcal{C}e^\p \times \mathcal{H}e_-^\p$, such that 
$\mathcal{T}e^1 (z) = \Phi_+(z) + \displaystyle	{	\frac	{1}	{z}	} + \Phi_- (z)$ for all $z \in \C - \Z$~.

With the definition of the formal Hurwitz multizeta functions, the formal analogue $Te^\p (X)$ should be defined by:
$$Te^\p (X) = He_+^\p (X) \times Ce^\p (X) \times He_-^\p (X)\ ,$$
where
$	Ce^{\seq{s}} (X)
	=
	\left \{
			\begin{array}{ll}
					1		&	\text{, if }l(\seq{s}) = 0 
					\\
					X^{-s}	&	\text{, if }l(\seq{s}) = 1  	\ ,
					\\
					0		&	\text{, if }l(\seq{s}) \geq 2
			\end{array}
	\right.
$
and is a \symmetrel mould defined on $\text{seq} (\N^*)$ and with values in $\C(\!(X)\!)$~.

\subsection	{Properties of the extension of the mould $\mathcal{T}e^\p$ to $\text{seq} (\N^*)$}

The convergent Hurwitz multizeta functions satisfies the differentiation and parity properties, as
the convergent multitangent functions. We want their extensions to $\text{seq} (\N^*)$ to satisfy
the same properties. These depend on the choice of the functions $\mathcal{H}e_+^1$ and
$\mathcal{H}e_-^1$~. From now on, we always define these functions by:
$$	\begin{array}{lll}
		\mathcal{H}e_+^1 (z)
		=
		\displaystyle	{	\sum	_{n > 0}
									\left (
											\frac	{1}	{n + z}
											-
											\frac	{1}	{n}
									\right ) 
						}
		&\hspace{0.3cm} , \hspace{0.3cm}&
		\mathcal{H}e_-^1 (z)
		=
		\displaystyle	{	\sum	_{n < 0}
									\left (
											\frac	{1}	{n + z}
											-
											\frac	{1}	{n}
									\right ) .
						}
	\end{array}
$$

\noindent
With these definitions, the corresponding \symmetrel extensions satisfy:

\begin{Lemma}
	For all sequences $\seq{s} \in \text{seq} (\N^*)$, we have:
	\begin{enumerate}
		\item $	\displaystyle	{	\frac	{\partial \mathcal{H}e_{\pm}^{\seq{s}}}	{\partial z}
									=
									- \sum	_{i = 1}
											^{l(\seq{s})}
											s_i \mathcal{H}e_{\pm}^{\seq{s} + e_i}
								}
			$ .
		\item $\mathcal{H}e_+^{\seq{s}}(-z) = (-1)^{||\seq{s}||} \mathcal{H}e_-^{\overset {\leftarrow} {\underline{s}}} (z)$ , where $z \in \C - \Z$~.
	\end{enumerate}
\end{Lemma}

\begin{Proof}
	$1$. Let us begin by proving this for formal Hurwitz multizeta functions.
	\\
	\\
	We just have to prove the first point because the second one is the definition of $He_-^{\seq{s}}(X)$~. For positive integers $s$ and $k$, the key point
	of the following computation is:
	$$	k	\binom{s + k - 1}{k}
		=
		s	\binom{s + k - 1}{k - 1}
		\ .
	$$
	Thus, for all $\seq{s} \in \text{seq} (\N^*)$, if the derivation of $\C[\![X]\!]$ is denoted by $D$, $D (He_+^{\seq{s}}) (X)$ is successively equal to:
			
				$$	\begin{array}{l}
						\displaystyle	{	- \sum	_{k \geq 1}
													\sum	_{	k_1, \cdots ,  k_r \geq 0
																\atop
																k_1 + \cdots + k_r = k
															}
															\displaystyle	{	\sum	_{p = 1}
																						^r
																						k_p
																						\left (
																								\displaystyle	{	\prod	_{i = 1}
																															^r
																															\binom{s_i + k_i - 1}{k_i}
																												}
																						\right )
																						\mathcal{Z}e^{s_1 + k_1, \cdots , s_r + k_r}
																						(-X)^{k - 1}
																			}
										}
						\vspace{0.1cm}	\\
						=
						\displaystyle	{	- \sum	_{p = 1}
													^r
													\sum	_{k \geq 0}
															\displaystyle	{	\sum	_{	k_1, \cdots , \widehat{k_p} , \cdots , k_r \geq 0
																							\atop
																							{	k_p \geq 1
																								\atop
																								k_1 + \cdots k_r = k + 1
																							}
																						}
																						\begin{array}[t]{@{}l@{}}
																							s_p
																							\left (
																									\displaystyle	{	\prod	_{i \in \crochet{1}{r} - \{p\}}
																																\binom{s_i + k_i - 1}{k_i}
																													}
																							\right )
																							\displaystyle	{\binom{s_p + k_p - 1}{k_p - 1}}
																							\\
																							\times \mathcal{Z}e^{s_1 + k_1, \cdots , s_r + k_r}
																							(-X)^{k}
																						\end{array}
																			}
										}
						\\
					\end{array}
				$$
				$$	\begin{array}{l}
						=
						\displaystyle	{	- \sum	_{p = 1}
													^r
													\sum	_{k \geq 0}
															\displaystyle	{	\sum	_{	k_1, \cdots , k_r \geq 0
																							\atop
																							k_1 + \cdots + k_r = k
																						}
																						\begin{array}[t]{@{}l@{}}
																							s_p
																							\left (
																									\displaystyle	{	\prod	_{i \in \crochet{1}{r} - \{p\}}
																																\binom{s_i + k_i - 1}{k_i}
																													}
																							\right )
																							\displaystyle	{\binom{s_p + k_p}{k_p}}
																							\\
																							\times \mathcal{Z}e^{s_1 + k_1, \cdots , s_{p - 1} + k_{p - 1}, s_p + 1 + k_p, s_{p + 1} + k_{p + 1}, \cdots , s_r + k_r}
																							(-X)^k
																						\end{array}
																			}
										}
						\vspace{0.1cm}	\\
						=
						\displaystyle	{	- \sum	_{p = 1}
													^r
													He_+^{s_1, \cdots , s_{p - 1}, s_p + 1, s_{p + 1}, \cdots , s_r} (X)
										} \ .
					\end{array}
				$$
				\\
				\\
			$2$. From the formal case to the analytic one.
				\\
				\\
				It is well-known that $0 \leq \mathcal{Z}e^{\seq{s}} \leq 2$ (resp. $0 \leq \mathcal{Z}e_-^{\seq{s}} \leq 2$)
				for all sequences $\seq{s} \in \mathcal{S}_b^\star$ (resp. $\seq{s} \in \mathcal{S}_e^\star$).
				Thus, the formal power series  $He_+^{\seq{s}}(X)$ and $He_-^{\seq{s}}(X)$ are actually Taylor expansions. Consequently, the previous equalities 
				are valid in the analytical case, first on the disc centered in 0 and with radius $\displaystyle{\frac{1}{2}}$ and then on $\C - \Z$,
				according to the identity theorem for holomorphic functions.
				\qed	\\
\end{Proof}

These properties have immediate consequences on $\mathcal{T}e^\p$ extended to $\text{seq} (\N^*)$. These, as well as Corollary \ref{corollary - definition extension} and the definition of $\mathcal{H}e_+^1$ and $\mathcal{H}€_-^\p$ are summed up in the following theorem:

\begin{Theorem}
	There exists\label{Theo2} a \symmetrel extension of $\mathcal{T}e^\p$ to $\text{seq} (\N^*)$, valued in holomorphic functions over $\C - \Z$, such that:
	$$	\left \{
				\begin{array}{@{}l}
					\mathcal{T}e^\p = \mathcal{H}e_+^\p \times \mathcal{C}e^\p \times \mathcal{H}e_-^\p \ ,	\\
					\mathcal{T}e^1 (z) = \displaystyle	{	\frac	{\pi}	{\tan (\pi z)}	} \text{ , for all }z \in \C - \Z\  .
				\end{array}
		\right.
	$$
	Moreover, the following properties hold for all sequences $\seq{s} \in \text{seq} (\N^*)$:

	\begin{enumerate}
		\item $		\displaystyle	{	\frac	{\partial \mathcal{T}e^{\seq{s}}}	{\partial z}
										=
										- \sum	_{i = 1}
												^{l(\seq{s})}
												s_i \mathcal{T}e^{\seq{s} + \seq{e_i}}
									}
				$~, where $\seq{e_i} = (0, \cdots, 0, 1, 0, \cdots, 0)$, the $1$ being in the $i$-th position.
		\item $		\mathcal{T}e^{\seq{s}}(-z)
					=
					(-1)^{||\seq{s}||} \mathcal{T}e^{\overset {\leftarrow} {\underline{s}}} (z)
				$ , where $z \in \C - \Z$~.
	\end{enumerate}
\end{Theorem}

\begin{Proof}
		From the definition of $\mathcal{H}e_+^1$ and $\mathcal{H}e_-^1$ , we only have to prove the differentiation properties as well as the parity property.
		According to the trifactorisation, for $\seq{s} \in \text{seq} (\N^*)$ and if we denote $\seq{e}_i = (0^{[i - 1]} ; 1 ; 0^{[l(\seq{s}) - i]})$, we have
		successively\footnotemark :
		\\
		\footnotetext	{	Let us remind that, for a sequence $\seq{\pmb{\alpha}}$ and two non-negative integers $i$ and $j$ such that $i \leq j \leq r$, the
							sequences $\seq{\pmb{\alpha}}^{\leq i}$ , $\seq{\pmb{\alpha}}^{i < \cdot \leq j}$ and $\seq{\pmb{\alpha}}^{> i}$ are defined by :
							$$	\begin{array}{lllllllll}
									\seq{\pmb{\alpha}}^{\leq i}
									=
									(\alpha_1 ; \cdots ; \alpha_i)
									&\hspace{0.1cm}&,&\hspace{0.1cm}&
									\seq{\pmb{\alpha}}^{i < \cdot \leq j}
									=
									(\alpha_{i + 1} ; \cdots ; \alpha_j)
									&\hspace{0.1cm}&,&\hspace{0.1cm}&
									\seq{\pmb{\alpha}}^{> i}
									=
									(\alpha_{i + 1} ; \cdots ; \alpha_r) \ .
								\end{array}
							$$
						}
		$1$. $	\begin{array}[t]{@{}ll@{}l}
						\displaystyle	{	\frac	{\partial \mathcal{T}e ^{\seq{s}}}	{\partial z}	}
						&=\ &
						\displaystyle	{	\sum	_{\seq{s}^1 \cdot \seq{s}^2 \cdot \seq{s}^3 = \seq{s}}
													\left (
															\displaystyle	{	\frac	{\partial \mathcal{H}e_+ ^{\seq{s}^1}}	{\partial z}
																				\mathcal{C}e^{\seq{s}^2}
																				\mathcal{H}e_- ^{\seq{s}^3}
																				 +  
																				\mathcal{H}e_+ ^{\seq{s}^1}
																				\frac	{\partial \mathcal{C}e^{\seq{s}^2}}	{\partial z}
																				\mathcal{H}e_- ^{\seq{s}^3}
																				 +  
																				\mathcal{H}e_+ ^{\seq{s}^1}
																				\mathcal{C}e^{\seq{s}^2}
																				\frac	{\partial \mathcal{H}e_- ^{\seq{s}^3}}	{\partial z}
																			}
													\right )
										}
						\vspace{0.1cm}	\\
						&=\ &
						\displaystyle	{	- \sum	_{i = 1}
													^{l(\seq{s})}
													\sum	_{	\seq{s}^1 \cdot \seq{s}^2 \cdot \seq{s}^3 = \seq{s}
																\atop
																l(\seq{s}^2) = 1
															}
															s_i
															\mathcal{H}e_+ ^{\seq{s}^1 + \seq{e}_i^{\leq l(\seq{s}^1)}}
															\mathcal{C}e^{\seq{s}^2 + \seq{e}_i^{l(\seq{s}^1) < \cdot \leq l(\seq{s}^1) + l(\seq{s}^2)}}
															\mathcal{H}e_- ^{\seq{s}^3 + \seq{e}_i^{> l(\seq{s}^1) + l(\seq{s}^2)}}
															\ ,
										}
						\\
						&=\ &
						\displaystyle	{	- \sum	_{i = 1}
													^{l(\seq{s})}
													\Bigg (
															\displaystyle	{	s_i
																				\sum	_{	\seq{s}^1 \cdot \seq{s}^2 \cdot \seq{s}^3 = \seq{s} + \seq{e}_i
																							\atop
																							l(\seq{s}^2) = 1
																						}
																						\mathcal{H}e_+ ^{\seq{s}^1}
																						\mathcal{C}e^{\seq{s}^2}
																						\mathcal{H}e_- ^{\seq{s}^3}
																			}
													\Bigg )
										}
						\\
						&=\ &
						\displaystyle	{	- \sum	_{i = 1}
													^{l(\seq{s})}
													s_i
													\mathcal{T}e^{\seq{s} + \seq{e}_i}
										} \ .
				\end{array}
			$
			\\
			\\
			$2$. 
			$	\begin{array}[t]{@{}lll}
					\mathcal{T}e^{\seq{s}} (-z)
					&=&
					\displaystyle	{	\sum	_{\seq{s}^1 \cdot \seq{s}^2 \cdot \seq{s}^3 = \seq{s}}
												\mathcal{H}e_+^{\seq{s}^1} (-z) \mathcal{C}e^{\seq{s}^2} (-z) \mathcal{H}e_-^{\seq{s}^3} (-z)
									}
					\\
					&=&
					\displaystyle	{	(-1)^{||\seq{s}||}
										\sum	_{\seq{s}^1 \seq{s}^2 \seq{s}^3 = \seq{s}}
												\mathcal{H}e_+^{\overset {\leftarrow} {\underline{s}^3}} (z)
												\mathcal{C}e^{\overset {\leftarrow} {\underline{s}^2}} (z)
												\mathcal{H}e_-^{\overset {\leftarrow} {\underline{s}^1}} (z)
									}
					\\
					&=&
					(-1)^{||\seq{s}||}
					\mathcal{T}e^{\overset {\leftarrow} {\underline{s}}} (z) \ .
				\end{array}
			$
			
			\qed
\end{Proof}

\subsection	{Reduction into monotangent functions}

The only property of divergent multitangent which remains to be proved is the reduction into monotangent
functions. From now on and until the end of this section, our aim is to find and prove such a property
for convergent and divergent multitangent functions. For this, we will adapt the proof given in the
convergent case. To simplify the computation, we will apply the same technique, that is to say a partial
fraction expansion, but for the generating series $\mathcal{T}ig^\p$~. Then, we will see that the
reduction into monotangent functions can not have exactly the same expression as in the convergent case.

\subsubsection	{Preliminary reminder}

In this section, we shall use extensively generating functions of a \symmetrel mould $Me^\p$ over $\text{seq} (\N^*)$ with values in a commutative algebra $\mathbb{A}$. Such a generating function is a formal mould (see Appendix \ref{Formal Moulds}), will always be denoted by $Mig^\p$, and is defined by:
$$	\left \{
			\begin{array}{l}
					Mig^\emptyset = 1 \ .	\\
					Mig^{v_1, \cdots, v_r}
					= \displaystyle	{	\sum	_{s_1, \cdots, s_r \geq 1}
												Me^{s_1, \cdots, s_r} {v_1}^{s_1 - 1} \cdots {v_r}^{s_r - 1}
									}
						\in	\mathbb{A} [ \! [ v_1 ; \cdots ; v_r ] \! ]\ .
			\end{array}
	\right.
$$

Such a mould turns out to be automatically \symmetril, meaning that an expression close to the \symmetrelity one holds:
$$	Mig^{\seq{v}} Mig^{\seq{w}}
	=
	\displaystyle	{	\sum	_{\seq{x} \in sh\text{\textbf{\underline{\textit{i}}}} (\seq{v}  ;  \seq{w})}
								Mig ^{\seq{x}}
					} \ .
$$

The multiset $sh\text{\textbf{\underline{\textit{i}}}} (\seq{v} ; \seq{w})$ is also a quasi-shuffle product as defined in \cite{Hoffmann}, as is the stuffle. If $\seq{v}$ and $\seq{w}$ are sequences over an alphabet of indeterminates, this set is defined exactly in the same way as $sh\text{\textbf{\underline{\textit{e}}}} (\seq{v} ; \seq{w})$, but here, the contraction of the quasi-shuffle product is an abstract contraction defined over $(\N^*)^2$. The evaluation of a mould $Mig^\p$ on a sequence which has such a contraction is then done by induction and given by the formula:
$$	Mig^{\seq{v} \cdot (x \circledast y) \cdot \seq{w}}
	=
	\displaystyle	{	\frac	{Mig^{\seq{v} \cdot x \cdot \seq{w}} - Mig^{\seq{v} \cdot y \cdot \seq{w}}}
								{x - y}
					}
	\ .
$$

For examples of relations satisfied by a \symmetril mould, see \S \ref{symmetrelityAppendix} of the appendix \ref{Elements de calcul moulien debut}.

\subsubsection {A first expression of $Tig^\p (X)$}

The moulds $\mathcal{Z}e^\p$, $Te^\p (X)$, $He_+^\p (X)$, $He_-^\p (X)$ and $Ce^\p (X)$ are symmetr\textbf{\textit {\underline {e}}}l. We will consider their generating functions, respectively denoted by $\mathcal{Z}ig^\p$,\label{definition zig} $Tig^\p (X)$, $Hig_+^\p (X)$, $Hig_-^\p (X)$ and $Cig^\p (X)$. Let us remark that these moulds are valued in $\C[\![X]\!]$ or $\C(\!(X)\!)$.

We can begin with the computation of the generating functions of $He^\p_+ (X)$:

\begin{Lemma}
	The generating function of the mould $He_+^\p (X)$ , denoted by $Hig_+^\p (X)$ with values in
	$\C[\![X]\!][\![(Y_r)_{r \in \N^*}]\!] \simeq \C[\![X ; Y_1 ; Y_2 ; \cdots]\!]$, is :
	\begin{center}
		$	Hig_+^{Y_1, \cdots , Y_r} (X) = \mathcal{Z}ig^{Y_1 - X, \cdots , Y_r - X} \ .	$
	\end{center}
\end{Lemma}

Such a result has to be expected, because Hurwitz multizeta functions $\mathcal{H}e_+^\p (z)$ are nothing else than translations of multizeta values. Consequently, this should have a translation readable on the generating function.

\begin{Proof}
			Let $r \in \N^*$.
			Let us denote $D_{Y_i}$ the derivation with respect to $Y_i$ and $S(F)$ the constant term of $F \in \mathbb {A}[\![Y_1 ; \cdots ; Y_r]\!]$~.
			\\
			For all $(k_1 ; \cdots ; k_r) \in (\N^*)^r$~, we have:
			\\
			\\
			$	\displaystyle	{	\frac	{1}	{k_1! \cdots k_r!}
									S	\left (
												D_{Y_1}^{k_1}\circ \cdots \circ D_{Y_r}^{k_r}
												(\mathcal{Z}ig^{Y_1 - X, \cdots , Y_r - X})
										\right )
								}
			$
			\vspace{0.1cm}	\\
			$	\begin{array}{@{}ll}
					\hspace{1cm}
					=&
					\displaystyle	{	S	\left (
													\sum	_{p_1, \cdots ,p_r \geq 0}
															\mathcal{Z}e^{p_1 + k_1 + 1, \cdots , p_r + k_r + 1}
															\displaystyle	{	\left (
																						\prod	_{i = 1}
																								^r
																								\binom{p_i + k_i}{k_i}
																								(Y_i - X)^{p_i}
																				\right )
																			}
											\right )
									}
					\vspace{0.1cm}	\\
					\hspace{1cm}
					=&
					\displaystyle	{	\sum	_{p_1, \cdots ,p_r \geq 0}
												\mathcal{Z}e^{p_1 + k_1 + 1, \cdots , p_r + k_r + 1}
												\binom{p_1 + k_1}{k_1}
												\cdots
												\binom{p_r + k_r}{k_r}
												(- X)^{p_1 + \cdots + p_r}
									}
					\vspace{0.1cm}	\\
					\hspace{1cm}
					=&
					\displaystyle	{	\sum	_{p \geq 0}
												  
												\sum	_{	p_1 + \cdots + p_r = p	}
														\mathcal{Z}e^{p_1 + k_1 + 1, \cdots , p_r + k_r + 1}
														\binom{p_1 + k_1}{k_1}
														\cdots
														\binom{p_r + k_r}{k_r}
														(- X)^p
									}
					\vspace{0.1cm}	\\
					\hspace{1cm}
					=&
					He_+^{k_1 + 1, \cdots, k_r + 1} (X) \ .
				\end{array}
			$
			\\
			\\

			Thus, the formal Taylor formula for formal power series in several indeterminates gives:
			$	\begin{array}[t]{@{}lll}
					\mathcal{Z}ig^{Y_1 - X, \cdots , Y_r - X}
					&=&
					\displaystyle	{	\sum	_{k_1, \cdots , k_r \geq 0}
												He_+^{k_1 + 1, \cdots , k_r + 1} (X) Y_1^{k_1} \cdots Y_r^{k_r}
									}
					\\
					&=&
					Hig_+^{Y_1, \cdots , Y_r} (X) \ .
				\end{array}
			$\\
			\qed	\\
\end{Proof}

Moreover, according to the parity property, the generating function $\mathcal{Z}ig_- ^{Y_1, \cdots ,  Y_r}$ of $\mathcal{Z}e_-^{\p}$ is defined by:
$$	\mathcal{Z}ig_- ^{Y_1, \cdots ,  Y_r} = (-1)^r \mathcal{Z}ig^{- Y_r, \cdots ,  - Y_1} \ .	$$

The same computation as in the previous proof gives:
\begin{Lemma}
	The generating function of the mould $He_-^\p (X)$ , denoted by $Hig_-^\p (X)$ with values in
	$\C[\![X]\!][\![(Y_r)_{r \in \N^*}]\!] \simeq \C[\![X ; Y_1 ; Y_2 ; \cdots]\!]$, is:
	\begin{center}
		$	Hig_-^{Y_1, \cdots , Y_r} (X) = \mathcal{Z}ig_-^{Y_1 - X, \cdots , Y_r - X} .	$
	\end{center}
\end{Lemma}

These two lemmas, together with the trifactorisation, imply immediately the first expression of
$Tig^\p (X)$~. 

\begin{Property}
	Let us denote\label{premiere expression de Tig} by $\mathcal{Z}ig^\p (X)$ and $\mathcal{Z}ig_-^\p (X)$ the moulds valued in 					 
	$\C[\![X]\!][\![(Y_r)_{r \in \N^*}]\!] \simeq \C[\![X ; Y_1 ; Y_2 ; \cdots]\!]$ respectively defined by:
	
	$$	\begin{array}{l}
			\mathcal{Z}ig^{Y_1, \cdots , Y_r} (X) = \mathcal{Z}ig^{Y_1 - X, \cdots , Y_r - X}  .
			\vspace{0.2cm}	\\
			\mathcal{Z}ig_-^{Y_1, \cdots , Y_r} (X) = \mathcal{Z}ig^{Y_1 - X, \cdots , Y_r - X}_-  .
		\end{array}
	$$
	Then, in $\C (\!(X)\!)[\![(Y_r)_{r \in \N^*}]\!]$, the generating series of $Te^\p(X)$ is given by:
	$$	Tig^\p (X) = \mathcal{Z}ig^\p(X) \times Cig^\p(X) \times \mathcal{Z}ig_-^\p (X) .	$$
\end{Property}

\subsubsection	{Second expression of $Tig^\p$ and flexion markers}
\label{Second expression of Tig}
The second expression of $Tig^\p (X)$ will use some notations and notions introduced by Jean Ecalle for his study of flexion structures (see. \cite{Ecalle2}, \cite{Ecalle3}, \cite{Ecalle4} or \cite{Ecalle6}). Let us introduce these before stating the result.

\paragraph{Flexion markers}	

The four flexion markers $\lfloor$ , $\rfloor$ , $\lceil$ and $\rceil$ act on factorisation of (bi)sequences. So, let us consider two alphabets $\Omega_1$ ,  $\Omega_2$ and then their product $\Omega = \Omega_1 \times \Omega_2$~; let us also consider a sequence $ \seq{w} \in \text{seq} (\Omega)$ which can be factorized:
$$\seq{w} = \seq{w}^1 \cdots \seq{w}^r \in \text{seq} (\Omega) \ .$$

The flexion marker $\lfloor$ acts on $\seq{w}^{i}$ by subtracting the right inferior element of $\seq{w}^{i - 1}$ to each inferior element of $\seq{w}^i$ while the flexion marker $\lceil$ acts on $\seq{w}^{i}$ by adding the sum of superior elements of $\seq{w}^{i - 1}$ to the left superior element of $\seq{w}^i$. In the same way, the flexion marker $\rfloor$ acts on $\seq{w}^{i}$ by subtracting the left inferior element of $\seq{w}^{i + 1}$ to each inferior element of $\seq{w}^i$ while the flexion marker $\rceil$ acts on $\seq{w}^{i}$ by adding the sum of superior elements of $\seq{w}^{i + 1}$ to the right superior element of $\seq{w}^i$.

By the use of these flexion markers, elements of $\Omega_1$ will be added to each other while elements of $\Omega_2$ will be subtracted each other.

To clarify the definitions and the actions of the different markers, here is an example.
If	$\seq{w}	=	\cdots \seq{a} \cdot \seq{b} \cdots
				=	\cdots
					\left (
							\begin{array}{@{}c@{}c@{}c@{}}
								u_6 &,\cdots,& u_{9}	\\
								v_6 &,\cdots,& v_{9}
							\end{array}
					\right )
					\left (
							\begin{array}{@{}c@{}c@{}c@{}}
								u_{10} &,\cdots,& u_{15}	\\
								v_{10} &,\cdots,& v_{15}
							\end{array}
					\right )
					\cdots
	$, then we have:
$$	\begin{array}{@{}lll}
		\seq{a} \rfloor =	\left (
									\begin{array}{ccc}
										u_6		 &,\cdots,& u_{9}	\\
										v_{6:10} &,\cdots,& v_{9:10}
									\end{array}
							\right )
		&,&
		\seq{a} \rceil =	\left (
									\begin{array}{cccc}
										u_6 &,\cdots,& u_8, & u_{9\cdots15}	\\
										v_6 &,\cdots,& v_8, & v_{9}
									\end{array}
							\right ) \ ,
		\\
		\\
		\lfloor \seq{b} =	\left (
									\begin{array}{ccc}
										u_{11}		&,\cdots,& u_{15}	\\
										v_{11:9}	&,\cdots,& v_{15:9}
									\end{array}
							\right )
		&,&
		\lceil \seq{b} =	\left (
									\begin{array}{cccc}
										u_{6 \cdots 10}	&	,\ u_{11} &,\cdots,& u_{15}	\\
										v_{10}			&	,\ v_{11} &,\cdots,& v_{15}
									\end{array}
							\right ) \ ,
	\end{array}
$$
where $n_{i \cdots j} = n_i + \cdots + n_j$ and $n_{i:j} = n_i - n_j$ in the variables $n$.

\paragraph{Colors}

If we do not care, it is easy not to see flexion structures. But, when there are some addition or
subtraction of the variables, flexion structures are possibly present. The use of colors is a good
way to avoid passing next to them. So, we will stiffen a bit more the situation by using colors in
a temporary way. This requires to redefine our moulds as bimoulds.

First, there is the bimould of colored multizeta values defined for sequences in $\text{seq} \left ( \Q / \Z \times \N^* \right )$ by:
$$	\mathcal{Z}e	^{	\begin{scriptsize}
							\left (
								\begin{array}{@{}l@{}}
										\varepsilon_1, \cdots , \varepsilon_r	\\
										s_1, \cdots , s_r
								\end{array}
							\right )
						\end{scriptsize}
					}
	=
	\displaystyle	{	\sum	_{1 \leq n_r < \cdots < n_1}
								\frac	{{e_1}^{n_1} \cdots {e_r}^{n_r}}
										{{n_1}^{s_1} \cdots {n_r}^{s_r}}
					}
	\text{ where }
	e_k = e^{-2i\pi \varepsilon_k} \text{, for } k \in \crochet{1}{n}  .
$$

Its generating series $\mathcal{Z}ig^{\p}$ is then a \symmetril mould defined for sequences in $\text{seq} \big ( \Q / \Z \times (V_i)_{i \in \N^*} \big )$~.
This also gives us the definition of the bimould $\mathcal{Z}ig^\p_-$:
$$	\mathcal{Z}ig_- ^{	\begin{scriptsize}
							\left (
								\begin{array}{@{}c@{}c@{}c@{}}
									\varepsilon_1,	&	\cdots ,	&	\varepsilon_r	\\
									V_1,			&	\cdots ,	&	V_r
								\end{array}
							\right )
						\end{scriptsize}
					}
	=
	(-1)^r \mathcal{Z}ig	^{	\begin{scriptsize}
									\left (
										\begin{array}{@{}c@{}c@{}c@{}}
												- \varepsilon_r	&	,\cdots ,	&	- \varepsilon_1	\\
												- V_r			&	,\cdots ,	&	- V_1
										\end{array}
									\right )
								\end{scriptsize}
							} .
$$

In a similar way, we can define formal or analytical colored Hurwitz multizeta functions as well as formal or analytical colored multitangent functions. These are bimoulds valued in the algebra of holomorphic functions over $\C - \Z$ or in $\C[\![X]\!]$~. If
$	\begin{scriptsize}
		\left (
			\begin{array}{@{}l@{}}
				\varepsilon_1, \cdots, \varepsilon_r	\\
				s_1, \cdots, s_r
			\end{array}
		\right )
	\end{scriptsize}
	\in
	\text{seq} \left ( \Q / \Z \times \N^* \right ) - \{ \emptyset \}
$, with the notation $e_k = e^{-2i\pi \varepsilon_k}$, for $k \in \crochet{1}{n}$, these are respectively defined by:
$$	\begin{array}{@{}r@{}c@{}l}
		\mathcal{H}e_+	^{	\begin{scriptsize}
								\left (
									\begin{array}{@{}l@{}}
											\varepsilon_1, \cdots , \varepsilon_r	\\
											s_1, \cdots , s_r
									\end{array}
								\right )
							\end{scriptsize}
						} \hspace {-0.1cm} (z)
		&\ =\ &
		\displaystyle	{	\sum	_{0 < n_r < \cdots < n_1 < + \infty}
									\frac	{{e_1}^{n_1} \cdots {e_r}^{n_r}}
											{{(n_1 + z)}^{s_1} \cdots {(n_r + z)}^{s_r}}
						} \ , 
		\nonumber
		\\
		\\
		He_+	^{	\begin{scriptsize}
						\left (
								\begin{array}{@{}l@{}}
									\varepsilon_1, \cdots , \varepsilon_r	\\
									s_1, \cdots , s_r
								\end{array}
						\right )
					\end{scriptsize}
				} \hspace {-0.1cm} (X)
		&\ =\ &
		\displaystyle	{	\hspace{-0.35cm}
							\sum	_{	k , k_1 , \cdots , k_r \geq 0
										\atop
										k_1 + \cdots + k_r = k
									}
									\hspace{-0.5cm}
											\left [
													\displaystyle	{	\prod	_{i = 1}
																				^r
																				\binom{s_i + k_i - 1}{k_i}
																	}
											\right ]
											\mathcal{Z}e	^{	\begin{scriptsize}
																	\left (
																			\begin{array}{@{}c@{}c@{}c@{}}
																				\varepsilon_1	&, \cdots ,& \varepsilon_r	\\
																				s_1 + k_1		&, \cdots ,& s_r + k_r
																			\end{array}
																	\right )
																\end{scriptsize}
															}
											(-X)^k
						} \ , 
	\end{array}
$$
$$	\begin{array}{@{}rcl}
		\mathcal{T}e	^{	\begin{scriptsize}
								\left (
									\begin{array}{@{}l@{}}
										\varepsilon_1, \cdots , \varepsilon_r	\\
										s_1, \cdots , s_r
									\end{array}
								\right )
							\end{scriptsize}
						} \hspace {-0.1cm} (z)
		&=\ &
		\displaystyle	{	\sum	_{- \infty < n_r < \cdots < n_1 < + \infty}
									\frac	{{e_1}^{n_1} \cdots {e_r}^{n_r}}
											{{(n_1 + z)}^{s_1} \cdots {(n_r + z)}^{s_r}}
						} \ ,
		\nonumber
		\\
		\\
		Te	^{	\begin{scriptsize}
					\left (
							\begin{array}{@{}l@{}}
								\varepsilon_1, \cdots , \varepsilon_r	\\
								s_1, \cdots , s_r
							\end{array}
					\right )
				\end{scriptsize}
			}\hspace {-0.1cm} (X)
		&=\ &
		\hspace{-0.4cm}
		\displaystyle	{	\sum	_{	\begin{scriptsize}
											\left (
												\begin{array}{@{}l@{}}
													\seq{\varepsilon}^1
													\\
													\seq{s}^1
												\end{array}
											\right )
											\cdot 
											\left (
												\begin{array}{@{}l@{}}
													\seq{\varepsilon}^2
													\\
													\seq{s}^2
												\end{array}
											\right )
											\cdot 
											\left (
												\begin{array}{@{}l@{}}
													\seq{\varepsilon}^3
													\\
													\seq{s}^3
												\end{array}
											\right )
											=
											\left (
												\begin{array}{@{}l@{}}
													\seq{\varepsilon}
													\\
													\seq{s}
												\end{array}
											\right )
										\end{scriptsize}
									}
									\hspace{-0.4cm}
									He_+	^{	\begin{scriptsize}
													\left (
														\begin{array}{@{}c@{}}
															\seq{\varepsilon}^1	\\
															\seq{s}^1
														\end{array}
													\right )
												\end{scriptsize}
											}
											(X)
									 
									Ce	^{	\begin{scriptsize}
												\left (
													\begin{array}{@{}c@{}}
														\seq{\varepsilon}^2	\\
														\seq{s}^2
													\end{array}
												\right )
											\end{scriptsize}
										}
										(X)
									 
									He_-	^{	\begin{scriptsize}
													\left (
														\begin{array}{@{}c@{}}
															\seq{\varepsilon}^3	\\
															\seq{s}^3
														\end{array}
													\right )
												\end{scriptsize}
											}
											(X)
						} \ .
		\nonumber
	\end{array}
$$

Obviously, these definitions contain some divergent cases for Hurwitz multizeta functions as well as for multitangent functions: in the first case, it is when  $(\varepsilon_1 ; s_1) = (0 ; 1)$ , while it is when $(\varepsilon_1 ; s_1) = (0 ; 1)$ or $(\varepsilon_r ; s_r) = (0 ; 1)$ in the second case. In these exceptional cases, a regularization process is needed and is based, as we have done previously without colors, on the regularization of the generating series $\mathcal{Z}ig^\p$ and, so, on the following well-known lemma due to Jean Ecalle (see \cite{Ecalle3} p. $5$ and \cite{Ecalle6} p. $6$)~:

\begin{Lemma}	\label{lemme de J. Ecalle}
	Let	$	\displaystyle	{	\mu^{n_1, \cdots ,n_r}
								=
								\frac	{1}	{r_1 ! \cdots r_n !}
							}
		$ where the non-increasing sequence $\seq{n} = (n_1 ; \cdots ; n_r) \in \text{seq} (\N^*)$  contains $r_1$ times its highest value, $r_2$ times its second
	highest value, etc.	\\
	For $(u_p)_{p \in \crochet{1}{r}} \in \C^r$ and $k \in \crochet{1}{r}$ , let $e_k = e^{-2 i u_k \pi}$ .	\\
	\\
	Finally, for all $k \in \N^*$, we consider the moulds $do\mathcal{Z}ig_k^\p$ and $co\mathcal{Z}ig_k^\p$ defined for all 
	$	\begin{scriptsize}
			\left (
				\begin{array}{@{}c@{}c@{}c@{}}
					u_1	&	, \cdots ,	&	u_r	\\
					v_1	&	, \cdots ,	&	v_r
				\end{array}
			\right )
		\end{scriptsize}
		\in
		\text{seq} \big ( \Q / \Z \times (V_i)_{i \in \N^*} \big )
	$ by:
	$$	\begin{array}{@{}lll}
			do\mathcal{Z}ig_k	^{	\begin{scriptsize}
										\left (
												\begin{array}{@{}c@{}c@{}c@{}}
													u_1	&	, \cdots ,	&	u_r	\\
													V_1	&	, \cdots ,	&	V_r
												\end{array}
										\right )	
									\end{scriptsize}
								}
			&=&
			\left \{
					\begin{array}{cl}
							\displaystyle	{	\sum	_{1 \leq n_r < \dots < n_1 < k}
														\frac	{{e_1}^{n_1} \cdots {e_r}^{n_r}}
																{(n_1 - V_1) \cdots (n_r - V_r)}
											}
							&
							\text{, if } r \neq 0 \ .
							\\
							1
							&
							\text{, if } r = 0 \ .
					\end{array}
			\right. 
			\\
			\\
			co\mathcal{Z}ig_k	^{	\begin{scriptsize}
										\left (
												\begin{array}{@{}c@{}c@{}c@{}}
													u_1	&	, \cdots ,	&	u_r	\\
													V_1	&	, \cdots ,	&	V_r
												\end{array}
										\right )
									\end{scriptsize}
								}
			&=&
			\left \{
					\begin{array}{cl}
						\displaystyle	{	(-1)^r	\sum	_{1 \leq n_r \leq \dots \leq n_1 < k}
															\frac	{\mu^{n_1, \cdots , n_r}}
																	{n_1 \cdots n_r}
										}
						&
						\text{, if } \seq{u} \neq \seq{0} \text{ and }r \neq 0 \ .
						\\
						0
						&
						\text{, if }\seq{u} = \seq{0} \text{ and }r \neq 0 \ .
						\\
						1
						&
						\text{, if } r = 0 \ .
					\end{array}
			\right. 
		\end{array}
	$$
	
	Then, the mould $\mathcal{Z}ig$ admits an elementary mould ``factorisation'':
	$$	\mathcal{Z}ig^\p = \displaystyle	{	\lim	_{n \longrightarrow + \infty}
														\left (
																co\mathcal{Z}ig^\p_n \times do\mathcal{Z}ig^\p_n
														\right )
											} \ .
	$$
\end{Lemma}

Let us remark that in this ``factorisation'', the mould $do\mathcal{Z}ig^\p_k$ gives us the dominant terms of $\mathcal{Z}ig^\p$, while the mould $co\mathcal{Z}ig_k^\p$ plays the role of correcting the series to restore the convergence of the divergent series
$$	\displaystyle	{	\sum	_{1 \leq n_r < \dots < n_1 < k}
								\frac	{{e_1}^{n_1} \cdots {e_r}^{n_r}}
										{(n_1 - v_1) \cdots (n_r - v_r)}
					}
	\ .
$$

\paragraph{Some new moulds}

Let $\delta$ be the indicator function of $\{0\}$. Let us also consider the formal bimoulds $Qig^\p$ and $\delta^\p$ defined on
$\text{seq} \big ( \Q / \Z \times (V_i)_{i \in \N^*} \big )$ by:
$$	\begin{array}{ll}
		\left \{
			\begin{array}{l}
					Qig^{\emptyset}	= 0 \ .	\\
					Qig	^{	\begin{scriptsize}
								\left (
									\begin{array}{@{}c@{}}
											u_1	\\
											V_1
									\end{array}
								\right )
							\end{scriptsize}
						}
					=
					- Te^{	\begin{scriptsize}
								\left (
									\begin{array}{@{}c@{}}
										u_1	\\
										1
									\end{array}
								\right )
							\end{scriptsize}
						} (V_1) \ .
					\\
					Qig	^{	\begin{scriptsize}
								\left (
									\begin{array}{@{}c@{}}
											u_1, \cdots, u_r	\\
											V_1, \cdots, V_r
									\end{array}
								\right )
							\end{scriptsize}
						}
					=
					0 \text { , if } r \geq 2 \ .
			\end{array}
		\right. 
	\end{array}
$$

$$	\left \{
			\begin{array}{l}
					\delta^{\emptyset}	= 0 \ .
					\\
					\delta	^{	\begin{scriptsize}
								\left (
									\begin{array}{@{}c@{}c@{}c@{}}
											u_1,	&	\cdots,	&	u_r	\\
											V_1,	&	\cdots,	&	V_r
									\end{array}
								\right )
							\end{scriptsize}
						}
					=
					\left \{
							\begin{array}{ll}
									\displaystyle	{	\frac	{(i\pi)^r}	{r!}	}
									\delta(u_1) \cdots \delta(u_r)
									&
									\text{, if } r \text{ is even.}
									\\
									0
									&
									\text{, if } r \text{ is odd.}
							\end{array}
					\right.
			\end{array}
	\right. 
$$

\paragraph{Second expression of $Tig^\p$}

We will apply the previous lemma, which gives an expression of $\mathcal{Z}ig^\p$ in the first expression of $Tig^\p$. This will allow us to make a partial fraction expansion in the indeterminate $X$. We then obtain:

\begin{Theorem}
	Let $\mathcal{Q}ig^\p$ be the bimould valued in $\mathcal{H}(\C - \Z)$	\label{reduction en monotangente de Tig} and defined for all $z \in \C - \Z$ by:
	$$	\begin{array}[t]{lll}
			\mathcal{Q}ig^{y_1, \cdots , y_r} (z)
			&=&
			\left \{
					\begin{array}{ll}
						- \mathcal{T}e^1 (y_1 - z)	&	\text{, if }r = 1 \ .	\\
						0							&	\text{, otherwise.}
					\end{array}
			\right.
		\end{array}
	$$
	Then, for all $z \in \C - \Z$, we have in $\C [\![X]\!][\![(V_r)_{r \in \N^*}]\!]$:
	$$	\mathcal{T}ig^\p (z)
		=
		\displaystyle	{	\delta^{\p}
							+
							\mathcal{Z}ig^{\p \rfloor} \times \mathcal{Q}ig^{\lceil \p \rceil} (z) \times \mathcal{Z}ig_-^{\lfloor \p}
						} \ .
	$$
\end{Theorem}

\begin{Proof}
	Continuing to use the same principle for our notations, we set:
	$$	\begin{array}{lll}
			do\mathcal{Z}ig_N	^{	\begin{scriptsize}
										\left (
												\begin{array}{@{}c@{}}
													u_1, \cdots , u_r	\\
													V_1, \cdots , V_r
												\end{array}
										\right )
									\end{scriptsize}
								} (X)
			&=&
			do\mathcal{Z}ig_N	^{	\begin{scriptsize}
										\left (
											\begin{array}{@{}c@{}c@{}c@{}}
												u_1		&	,  \cdots,	& u_r	\\
												V_1 - Y	&	,  \cdots,	& V_r - Y
											\end{array}
										\right )
									\end{scriptsize}
								} \ .
		\end{array}
	$$
	$$	\begin{array}{lll}
			co\mathcal{Z}ig_N	^{	\begin{scriptsize}
										\left (
												\begin{array}{@{}c@{}}
													u_1, \cdots , u_r	\\
													V_1, \cdots , V_r
												\end{array}
										\right )
									\end{scriptsize}
								} (X)
			&=&
			co\mathcal{Z}ig_N	^{	\begin{scriptsize}
										\left (
												\begin{array}{@{}c@{}c@{}c@{}}
													u_1		&	,  	\cdots, &	u_r	\\
													V_1 - Y	&	,  	\cdots,	&	V_r - Y
												\end{array}
										\right )
									\end{scriptsize}
								} \ .
			\vspace{0.1cm}	\\
			do\mathcal{Z}ig_{-, N}	^{	\begin{scriptsize}
												\left (
														\begin{array}{@{}c@{}}
															u_1, \cdots , u_r	\\
															V_1, \cdots , V_r
														\end{array}
												\right )
										\end{scriptsize}
									} (X)
			&=&
			(-1)^r
			do\mathcal{Z}ig_{-, N}	^{	\begin{scriptsize}
												\left (
														\begin{array}{@{}c@{}}
															-u_r, \cdots, -u_1	\\
															-V_r, \cdots, -V_1
														\end{array}
												\right )
										\end{scriptsize}
									} (X)  \ .
			\vspace{0.1cm}	\\
			co\mathcal{Z}ig_{-, N}	^{	\begin{scriptsize}
												\left (
														\begin{array}{@{}c@{}}
															u_1, \cdots , u_r	\\
															V_1, \cdots , V_r
														\end{array}
												\right )
										\end{scriptsize}
									} (X)
			&=&
			(-1)^r
			co\mathcal{Z}ig_{-, N}	^{	\begin{scriptsize}
												\left (
														\begin{array}{@{}c@{}}
															-u_r, \cdots , -u_1	\\
															-V_r, \cdots , -V_1
														\end{array}
												\right )
										\end{scriptsize}
									} (X)  \ .
		\end{array}		
	$$
	\noindent
	 Let	$	\left (
						\begin{array}{@{}c@{}c@{}c@{}}
								u_1,	&	\cdots,	&	u_r	\\
								V_1,	&	\cdots,&	V_r
						\end{array}
				\right )
				\in
				\text{seq} \big ( \Q / \Z \times (V_i)_{i \in \N^*} \big )
			$.
	Applying Lemmas $\ref{premiere expression de Tig}$ and $\ref{lemme de J. Ecalle}$,
	$$	Tig^\p (X)
		=
		\displaystyle	{	\lim	_{N \longrightarrow + \infty}
									\left (
											co\mathcal{Z}ig_N^\p \times Tig_N^\p (X) \times co\mathcal{Z}ig_{-,N}^\p
									\right )
						} \ ,
	$$
	where $Tig_N^\p (X) = do\mathcal{Z}ig_N^\p (X) \times Cig^\p (X) \times do\mathcal{Z}ig_{-, N}^\p (X)$~.
	\\
			
	It is not difficult to obtain another form of the previous trifactorisation by proceeding in the same way as in the proof of the trifactorisation of $\mathcal{T}e^\p$:
	\begin{equation}\label{Tig}
		\begin{array}{l@{}l@{}l}
			\mathcal{T}ig_N	^{	\begin{scriptsize}
									\left (
											\begin{array}{@{}c@{}}
												u_1, \cdots , u_r	\\
												V_1, \cdots , V_r
											\end{array}
									\right )
								\end{scriptsize}
							}			
			&=&
			\displaystyle	{	\sum	_{-N < n_r < \cdots < n_1 < N}
										\frac	{{e_1}^{n_1} \cdots {e_r}^{n_r}}
												{(n_1 - V_1 + X) \cdots (n_r - V_r + X)}
							}
			\ .
		\end{array}
	\end{equation}
	\\
	\noindent
	Now, we can write down the partial fraction expansion in $\mathcal{T}ig_N^\p$:
	$$	\begin{array}{ll}
			Tig_N	^{	\begin{scriptsize}
							\left (
								\begin{array}{@{}c@{}}
									u_1, \cdots, u_r	\\
									V_1, \cdots, V_r
								\end{array}
							\right )
						\end{scriptsize}
					}
			\hspace{-0.2cm}
			=
			\displaystyle	{	\sum	_{k = 1}
										^r
										\hspace{0.1cm}
										\sum	_{-N < n_r < \cdots < n_1 < N}
							}
			&	\displaystyle	{	\left (
											\prod	_{	j \in \crochet{1}{r}
														\atop
														j \neq k
													}
													\frac	{{e_j}^{n_j}}
															{n_j - n_k + V_k - V_j}
									\right )
									\times
								}
			\\
			&	\displaystyle	{	\frac	{{e_k}^{n_k}}
											{n_k - V_k + X}
								}
			\ .
		\end{array}
	$$
			
	Plugging this expansion in $(\ref{Tig})$, after some computations, we obtain in $\C [\![X]\!][\![(V_r)_{r \in \N^*}]\!]$:
	$$	Tig^\p (X)
		=
		\displaystyle	{	\delta^{\p}
							+
							\mathcal{Z}ig^{\p \rfloor} \times Qig^{\lceil \p \rceil} (X) \times \mathcal{Z}ig_-^{\lfloor \p}
						} \ .
	$$
	\\
	\noindent
	It is clear that $\mathcal{Q}ig^\p (z)$ is also well-defined in a neighbourhood of $0$ in $\text{seq} (\C)$. Moreover, the generating functions
	$\mathcal{Z}ig^\p$ and $\mathcal{Z}ig_-^\p$ are actually Taylor expansions defined in $\text{seq} \big ( D(0 ; 1) \big )$~. Here, the key point is that
	$|\mathcal{Z}e^{\seq{s}}| \leq 4^r r!$ for all sequences $\seq{s} \in \text{seq} (\N^*)$ of length $r$. Such an upper bound is far from being precise, but is
	sufficient for our purpose. A proof of it comes from the shift to the right of the ones beginning an evaluation sequence and:
	$$	\sharp sh\text{\textbf{\underline{\textit{e}}}}(\seq{\pmb{\alpha}} ; \seq{\pmb{\beta}})
		=
		\displaystyle	{	\sum	_{k = 0}
									^{\min (a ; b)}
									2^k
									\binom{a}{a - k}
									\binom{b}{b - k}
						} \ .
	$$
	This gives a neighbourhood of $0$ in $\text{seq} (\C)$ where
	$\mathcal{Z}ig^{\p \rfloor} \times Qig^{\lceil \p \rceil} (X) \times \mathcal{Z}ig_-^{\lfloor \p}$ defines an analytic function. The identity theorem for
	holomorphic functions concludes the proof of this theorem.	\\
	\qed
\end{Proof}

To conclude this section, let us explain why the corrective term $\delta^\p$ is mandatory.

Let us imagine this is not the case, that is to say that we have two functions $\varphi$ and $\psi$ such that the mould $\mathcal{T}e^\p$ is extended to the divergent case by $\mathcal{T}e^\p = \mathcal{H}e^\p_{+, \varphi} \times \mathcal{C}e^\p \times \mathcal{H}e^\p_{-, \psi}$ , where $\mathcal{H}e^\p_{+, \varphi}$ and $\mathcal{H}e^\p_{-, \psi}$ are respectively the extension to the divergent case of the moulds $\mathcal{H}e^\p_+$ and $\mathcal{H}e^\p_-$ such that $\mathcal{H}e^1_+ = \varphi$ and $\mathcal{H}e^1_- = \psi$~.

Then, we would have the following identity because of the fundamental equality proved in the previous theorem, but without the corrective term:

$$	\mathcal{H}ig_{+,\varphi}^\p \times \mathcal{C}ig^\p \times \mathcal{H}ig_{-,\psi}^\p
	=
	\mathcal{T}ig^\p
	=
	\mathcal{Z}ig^{\p \rfloor} \times \mathcal{Q}ig^{\lceil \p \rceil} (z) \times \mathcal{Z}ig_-^{\lfloor \p}
	\ .
$$

In particular, we would have equality of the constant terms of these generating functions, that is, we would have $1 = 0$ $\cdots$ Consequently, we cannot find a choice of the functions $\varphi$ and $\psi$ that extend the Hurwitz multizeta functions such that there is no corrective term in the reduction into monotangent of divergent multitangent functions.

\subsubsection {Reduction into monotangent functions for divergent multitangent functions}
\label{reduction2}

Theorem \ref{reduction en monotangente de Tig} admits the following corollary which comes from a direct formal power series expansion of $\mathcal{T}ig^\p (z)$~. This corresponds exactly to the fourth point mentionned at the beginning of this section. Let us remark that, from now on, we only consider moulds and not bimoulds.

Let us recall that we have introduced the following notations (see sections \ref{section-3-1} and \ref{Second expression of Tig}):
	$$	\begin{array}{lll}
			^i B_{\seq{k}}^{\seq{s}}
			&=&
			\displaystyle	{	\left (
										\prod	_{l = 1}
												^{i - 1}
												(-1)^{k_l}
								\right )
								\left (
										\prod	_{l = i + 1}
												^r
												(-1)^{s_l}
								\right )
								\left (
										\prod	_{	l = 1
													\atop
													l \neq i
												}
												^r
												\binom{s_l + k_l - 1}{s_l - 1}
								\right )
							}\ .
			\vspace{0.1cm}	\\
			\mathcal{Z}^{\seq{s}}_{i,k}
			&=&
			\displaystyle	{	\sum	_{	k_1, \cdots ,  \widehat{k_i}, \cdots , k_r \geq 0
											\atop
											k_1 + \cdots + \widehat{k_i} + \cdots + k_r = k
										}
										{}^i B_{\seq{k}}^{\seq{s}}
										\mathcal{Z}e^{s_r + k_r, \cdots ,  s_{i + 1} + k_{i + 1}}
										\mathcal{Z}e^{s_1 + k_1, \cdots ,  s_{i - 1} + k_{i - 1}}
							}\ .
			\vspace{0.1cm}	\\
			\delta^{\seq{s}} &=&	\left \{
											\begin{array}{ll}
												\displaystyle	{	\frac	{(i \pi)^r}	{r!}	}
												&
												\text{, if } \seq{s} = 1^{[r]} \text{ et if } r \text{ is even.}
												\\
												0
												&
												\text{, otherwise.}
											\end{array}
									\right.
		\end{array}
	$$

Then, we have:

\begin{Theorem}
	\label{reduction en monotangente 2}
	\textit{(Reduction into monotangent functions, version $2$)}	\\
	For all sequences $\seq{s} \in \text{seq} (\N^*)$, we have:
	$$	\mathcal{T}e^{\seq{s}} (z)
		=
		\delta^{\seq{s}}
		+
		\displaystyle	{	\sum	_{i = 1}
									^r
									\sum	_{k = 1}
											^{s_i}
											\mathcal{Z}_{i,s_i - k}^{\seq{s}}
											\mathcal{T}e^k (z)
						} \ .
	$$
	Moreover, if $\seq{s} \in \mathcal{S}_{b,e}^\star$, the summation of $k$ begins at $2$~.
\end{Theorem}

This result is computable. One can give complete tables for divergent multitangent functions up to a fixed weight, as in the convergent case (see Table $\ref{table1}$ for the convergent case and Table $\ref{table8}$ for the divergent case)~.

\input {table8.sty}

For example, one can see that $\mathcal{T}e^{2,1}$ and $\mathcal{T}e^{1,2}$ are identically vanishing. As already said in the introduction, this remarkable fact shows that the relation of \symmetrelity
$	\mathcal{T}e^{2} \mathcal{T}e^1 
	=
	\mathcal{T}e^{2,1} + \mathcal{T}e^{1,2} + \mathcal{T}e^3
	=
	\mathcal{T}e^3(z)\label{einsenstein}
$ allows us to find in a different way (more complicated, but more general) the simplest relations between Eisenstein series.

\section	{Some explicit computations of multitangent functions}
\label{calculs explicite}

Before presenting some explicit computations of multitangent functions, let us recall a few notations. If $\seq{\pmb{\alpha}}$ is any sequence, then $\seq{\pmb{\alpha}}^{[r]}$ denotes the sequence $\underbrace {\seq{\pmb{\alpha}}   \cdot  \cdots   \cdot  \seq{\pmb{\alpha}}}_{r \text{ times}}$ , where the sequence $\seq{\pmb{\alpha}}$ is repeated $k$ times. In particular, $n^{[k]}$ is the sequence $(n ; \cdots ; n)$ where $n$ is repeated $k$ times.

\subsection	{Computation of $\mathcal{T}e^{1^{[r]}}(z)$ , for $r \in \N$}
\label{calcul_de_Te111}
For all $r \in \N$, $\mathcal{T}e^{1^{[r]}}(z)$ is the constant term of $Tig^{Y_1,\cdots,Y_r}$, so

$$	\mathcal{T}e^{1^{[r]}} (X)
	=
	\left \{
		\begin{array}{@{}c@{}l@{}}
			\displaystyle	{	\frac	{(i\pi)^r}	{r!}	}
			&
			\text{, if } r \in 2\Z
			\\
			0
			&
			\text{, if } r \not \in 2\Z
		\end{array}
	\right \}
	+
	\left (
			\sum	_{k = 0}
					^{r - 1}
					\mathcal{Z}ig^{0^{[k]}} \mathcal{Z}ig_-^{0^{[r - 1 - k]}}
	\right )
	\mathcal{T}e^1 (X) \ .
$$

We will evaluate	$	\displaystyle	{	\sum	_{k = 0}
								^n
								\mathcal{Z}ig^{0^{[k]}} \mathcal{Z}ig_-^{0^{[n - k]}}
						}
			$ for $n \in \N^*$, by considering the product $\mathcal{Z}_+ \mathcal{Z}_-$, where:
$$	\left \{
			\begin{array}{l}
				\displaystyle	{	\mathcal{Z}_+	= \sum	_{n \geq 0}
															\mathcal{Z}ig^{0^{[n]}} X^n
													= \sum	_{n \geq 0}
														\mathcal{Z}e^{1^{[n]}} X^n
								} \ .
				\\
				\displaystyle	{	\mathcal{Z}_-	= \sum	_{n \geq 0}
															\mathcal{Z}ig_-^{0^{[n]}} X^n
													= \sum	_{n \geq 0}
															\mathcal{Z}e_-^{1^{[n]}} X^n
								} \ .
			\end{array}
	\right.
$$

The mould $\mathcal{Z}e^\p$ and $\mathcal{Z}e_-^\p$ being \symmetrel, we automatically obtain the following formal differential equations
(see Property $\ref{equadif}$)~:

$$	\left \{
			\begin{array}{l}
				\displaystyle	{	\label{calcul de Ze1...1)}
									D \mathcal{Z}_+	=	\mathcal{Z}_+
														\times
														\left (
																\sum	_{n \geq 0}
																		(-1)^n \mathcal{Z}e^{n + 1} X^n
														\right )
													=	\mathcal{Z}_+
														He_+^1
								} \ .
				\vspace{0.1cm}	\\
				\displaystyle	{	D\mathcal{Z}_-	=	\mathcal{Z}_-
														\times
														\left (
																\sum	_{n \geq 0}
																		(-1)^n \mathcal{Z}e_-^{n + 1} X^n
														\right )
													=	\mathcal{Z}_-
														He_-^1
								} \ .
			\end{array}
	\right.
$$

So:	$	\begin{array}[t]{@{}lllll}
				D (\mathcal{Z}_+\mathcal{Z}_-)	&=&	\mathcal{Z}_+ He_+^1 \mathcal{Z}_- + \mathcal{Z}_+ He_-^1 \mathcal{Z}_-
												\ = \ 	\mathcal{Z}_+ \mathcal{Z}_- \big ( He_+^1 + He_-^1 \big )
				\\
												&=&	-2 \mathcal{Z}_+ \mathcal{Z}_-	\left (
																							\displaystyle	{	\sum	_{n \geq 0}
																														\mathcal{Z}e^{2n + 2} X^{2n + 1}
																											}
																					\right )
				\vspace{0.1cm}	\\
												&=&	-2 \mathcal{Z}_+ \mathcal{Z}_-	\left (
																							\displaystyle	{	\sum	_{n \geq 1}
																														\mathcal{Z}e^{2n} X^{2n - 1}
																											}
																					\right ) \,.
			\end{array}
		$
\\
\\
Letting $Exp$ be the exponential map, we obtain:
$$	\mathcal{Z}_+ \mathcal{Z}_-
	=
	Exp	\left (
				- \displaystyle	{	\sum	_{n \geq 1}
											\frac	{\mathcal{Z}e^{2n}}	{n}
											X^{2n}
									}
		\right ) .
$$

\noindent
On the other hand, in $\C(\!(X)\!)$, we have:	$$	He^1_+(X) + He^1_-(X) = Te^1(X) - X^{-1}
													=
													\displaystyle	{	\pi   \frac	{\cos (\pi X)}	{\sin(\pi X)}
																		-
																		\frac	{1}	{X}
																	}
													\ .
												$$
Indeed, this relation is valid in $\C[\![X]\!]$, so:
$$	He^1_+(X) + He^1_-(X)
	=
	D	\left (
				Log	\left (
							\displaystyle	{	\frac	{\sin (\pi X)}	{\pi X}	}
					\right )
		\right )
	\ .
$$
So that:
$$	\mathcal{Z}_+ \mathcal{Z}_-
	=
	Exp	\left (
				- \displaystyle	{	\sum	_{n \geq 1}
											\frac	{\mathcal{Z}e^{2n}}	{n}
											X^{2n}
									}
		\right )
	=
	\frac	{\sin(\pi X)}	{\pi X}
	=
	\sum	_{n \geq 0}
			(-1)^n
			\frac	{(\pi X)^{2n}}
					{(2n + 1)!} \ .
$$

\noindent
Finally, we obtain:
$$	\begin{array}{lll}
		Te^{1^{[r]}} (X)
		&=&
		\left \{
			\begin{array}{@{}c@{}l@{}}
				\displaystyle	{	\frac	{(i\pi)^r}	{r!}	}
				&
				\text{ , if } r \in 2\Z
				\\
				0
				&
				\text{ , if } r \not \in 2 \Z
			\end{array}
		\right \}
		+
		\left \{
			\begin{array}{@{}c@{}l@{}}
				0
				&
				\text{ , if } r \in 2\Z
				\\
				\displaystyle	{	\frac	{(i\pi)^{r - 1}}	{r!}	}
				&
				\text{ , if } r \not \in 2\Z
			\end{array}
		\right \}
		\times
		Te^1 (X) \ ,
	\end{array}
$$
\noindent
and the analytic equality follows for all $z \in \C - \Z$:
$$	\mathcal{T}e^{1^{[r]}} (z)
	=
	\left \{
			\begin{array}{@{}l@{}l@{}}
				(-1)^p
				\displaystyle	{	\frac	{\pi^{2p}}	{(2p)!}	}
				&
				\text{ , if } r = 2p \ .
				\vspace{0.1cm}	\\
				(-1)^p
				\displaystyle	{	\frac	{\pi^{2p}}	{(2p + 1)!}	}
				\mathcal{T}e^1 (z)
				&
				\text{ , if } r = 2p + 1 \ .
			\end{array}
	\right.
$$

\subsection	{Computation of $\mathcal{T}e^{n^{[k]}} (z)$, for $n \in \N^*$ and $k \in \N$}

We now prove Property \ref{Property2Intro}, p. \pageref{Property2Intro}, giving an explicit evaluation of all multitangent functions of the form $\mathcal{T}e^{n^{[k]}} (z)$, for $n \in \N^*$ and $k  \in  \N$,
in terms of monotangent functions and multizeta values.

We will use an elementary theory of formal power series in one indeterminate. The central point is the following lemma. This gives us a formal differential equation automatically satisfied by the generating functions of the family of multitangent functions under consideration. Then, we only have to find out a formal power series expansion of solutions of this equation.

\subsubsection	{A property linking \symmetrelity and formal differential equation}

Let us begin by proving the following general property concerning \symmetrel moulds:

\begin{Property}
	\label{equadif}
	Let us consider a commutative algebra $\mathbb {A}$, a semigroup $(\Omega ; +)$ and
	a \symmetrel mould $Se^\p \in \mathcal{M}_{\mathbb {A}}^\p (\Omega)$~.
	\\
	For all $\omega \in \Omega$ , we set:
	$$	F_\omega = \displaystyle	{	\sum	_{p = 0}
												^{+ \infty}
												Se^{\omega^{[p]}} X^p
									}\hspace{0.5cm} , \hspace{0.5cm}
		G_\omega = \displaystyle	{	\sum	_{p = 0}
												^{+ \infty}
												(-1)^p
												Se^{(p + 1)\omega}
												X^p
									}   .
	$$
	For a given $\omega \in \Omega$, the formal power series $F_\omega$ satisfies the differential equation:
	$$	D Y = Y G_\omega \ .	$$
\end{Property}

Let us point out that this property is well-known in combinatorics as the Newton relations for symmetric functions. Here, the term $F_\omega$ represents the elementary symmetric functions while the term $G_\omega$ is then the power sums.

The proof we will give here of this property is based on the shift to the right of the ones beginning an evaluation sequence $\omega \in \Omega$ of the mould $Se^\p$ ; so the proof is exactly based on the notion of symmetr\textbf{\textit {\underline {e}}}lity. This algorithm is recursively presented by the following formula:
$$	Se^{\omega^{[p]}} Se^\omega
	=
	(p + 1) Se^{\omega^{[p + 1]}}
	+
	\sum	_{k = 0}
			^{p - 1}
			Se^{\omega^{[k]},  2 \omega,  \omega^{[p - k - 1]}}  .
$$

\begin{Proof}
			Let us fix $\omega \in \Omega$ and introduce the temporary notation $u_{p,l}$ for $(p ; l) \in \N \times \N^*$:
			$$	u_{p,l} = \displaystyle	{	(-1)^l	\sum	_{k = 0}
															^p
															Se^{\omega^{[k]} , l \omega, \omega^{[p - k]}}
										}.
			$$

			Then, using the \symmetrelity property, we have for $(p ; l) \in (\N^*)^2$ :
			\\
			$	\begin{array}[t]{@{}lll}
					(-1)^l Se^{\omega^{[p]}} Se ^{l \omega}
					&=&
					(-1)^l	\displaystyle	{	\sum	_{k = 0}
														^p
														Se^{\omega^{[k]}, l \omega , \omega^{[p - k]}}
												-
												(-1)^{l + 1}	\sum	_{k = 0}
																		^{p - 1}
																		Se^{\omega^{[k]} , (l + 1) \omega , \omega^{[p -1 - k]}}
											}
					\\
					&=&
					u_{p,l} - u_{p - 1, l + 1} \ .
				\end{array}
			$
			\\
			\\
			
			This implies successively, for $p \in \N^*$:\\
			$	\begin{array}[t]{@{}lll}
					\displaystyle	{	\sum	_{l = 0}
												^{p - 1}
												(-1)^l Se^{\omega^{[p - l]}} Se^{(l + 1) \omega}
									}
					&=&
					\displaystyle	{	-\sum	_{l = 1}
												^{p}
												(-1)^l Se^{\omega^{[p - (l - 1)]}} Se^{l \omega}
									}
					\vspace{0.2cm}	\\
					&=&
					\displaystyle	{	-\sum	_{l = 1}
												^p
												\big (
														u_{p - (l - 1), l} - u_{p - l, l + 1}
												\big )
									}
					\\
					&=&
					u_{0, p + 1} - u_{p, 1}
					\\
					&=&
					(-1)^{p + 1} Se^{(p + 1) \omega} + (p + 1) Se^{\omega^{[p + 1]}}  .
				\end{array}
			$
			\\
			\\
			Then:	$	 \displaystyle	{	(p + 1) Se^{\omega^{[p + 1]}}
											=
											\sum	_{l = 0}
													^p
													(-1)^l
													Se^{\omega^{[p - 1]}} Se^{(l + 1) \omega}
										}
					$ , for all $p \in \N^*$~.
			\\
			\\
			Since the previous equality is also true for $p = 0$, we can state the following equality between formal power series:
			$$	D F_\omega = F_\omega G_\omega  \ .	$$
			\qed
\end{Proof}

Using the fact that two formal power series with the same formal derivative differ only by their
constant term, it is not difficult to see that, if $\mathbb {A}$ is a ring and if
$\varphi \in \mathbb {A}[\![X]\!]$, then the formal power series satisfying $DY = Y D\varphi$ are defined by:
$$	Y(X) = C  Exp(\varphi(X) -  \varphi(0)),  C \in \mathbb {A} \ .	$$

Here, $Exp$ refers to the exponential. The resolution of such a formal differential equation boils
down to a problem of expressing an indefinite integral. The constant $C$ is then determined
by the constant term in $Y$.
\\

Recall that	$	\displaystyle	{	\mathcal{Z}_+ = \sum	_{r \geq 0}
															\mathcal{Z}e^{1^{[r]}} X^r
								}
			$ satisfies the formal following differential equation, as we have seen in \S \ref{calcul_de_Te111} during the evaluation of $Te^{1, \cdots, 1}(X)$:
$$	D \mathcal{Z}_+
	=
	\mathcal{Z}_+	\left (
						\sum	_{p \geq 0}
								(-1)^p \mathcal{Z}e^{p + 1} X^p
					\right )
	=
	\mathcal{Z}_+ He^1_+ \ .
$$

\subsubsection	{Application to the mould $Te^\p(X)$}

Recall that the mould $Te^\p (X)$ has been extended to $\text{seq} (\N^*)$ in the previous section, in order to preserve the \symmetrelity property. Hence, the previous property applies: if we set, for $n \in \N^*$, 
$	T_n = \displaystyle	{	\sum	_{p = 0}
									^{+ \infty}
									Te^{n^{[p]}}(X) Y^p
						}
$ and	$	U_n = \displaystyle	{	\sum	_{p = 0}
											^{+ \infty}
											(-1)^p Te^{n(p + 1)} (X) Y^p
								}
		$, we then have, for all positive integer $n$, in $\C(\!(X)\!)[\![Y]\!]$:
$$	D T_n = T_n  U_n \ .$$

We just need to compute a formal indefinite integral of $U_n$ in order to compute $Te^{n^{[p]}} (X)$ for all $p \in \N$. Let us consider $V_n \in \C(\!(X)\!)[\![Y]\!]$ defined by $V_n(X ; Y) = U_n (X ; Y^n)$~. A permutation of formal summation symbols (which is a priori a non-authorized operation) , followed by a partial fraction expansion, suggests we have for all positive integer $n$:
$$						  	nY^{n - 1} V_n(X  ; Y)
							=
							\displaystyle	{	- \sum	_{k = 0}
														^{n - 1}
														e	^{(2k + 1)	\frac	{i \pi}	{n}}
														Te^1 	\left (
																		X - e^{(2k + 1)	\frac	{i \pi}	{n}} Y
																\right )
											} \ .
$$

Recall that, here, $S$ denotes the linear map that associates with each formal power series its
constant term, while $D_{(Y)}$ denotes the formal derivative relative to the indeterminate $Y$. Indeed,
the Taylor formula allows to prove this relation in the ring $\C(\!(X)\!)[\![Y]\!]$~. For $l \in \N$ ,
if we denote the right hand side of the previous equality by $W_n$, we have successively:
\\
$	\begin{array}{@{}lll}
		\displaystyle	{	\frac	{1}	{l!}
							S	\left (
										D_{(Y)}^l W_n
								\right )
						}
		&=&
		S	\left (
					\displaystyle	{	- \sum	_{k = 0}
												^{n - 1}
												e	^{(2k + 1)(l + 1)	\frac	{i \pi}	{n}	}
												Te^{l + 1}	\left (
																	X - e	^{(2k + 1) \frac	{i \pi} {n}} Y
															\right )
									}
			\right )
		\vspace{0.2cm}	\\
		&=&
		-
		\left (
				\displaystyle	{	\sum	_{k = 0}
											^{n - 1}
											e	^{(2k + 1)(l + 1)	\frac	{i \pi}	{n}	}
								}
		\right )
		Te^{l + 1}	(X)
		\vspace{0.1cm}	\\
		&=&
		\left \{
				\begin{array}{ll}
					0
					&	\text{, si } l + 1 \not \equiv 0 [n] \ .
					\\
					n(-1)^{q + 1} Te^{qn} (X)
					& 	\text{, si } l + 1 = qn \ .
				\end{array}
		\right.
	\end{array}
$

Hence:
$	\begin{array}[t]{lll}
		W_n
		&=&
		\displaystyle	{	\sum	_{l = 0}
									^{+ \infty}
									\frac	{1}	{l!}
									S	\left (
												D_{(Y)}^l W_n
										\right )
									Y^l
							=
							\sum	_{q = 1}
									^{+ \infty}
									n(-1)^{q + 1} Te^{qn}(X) Y^{qn - 1}
						}
	\end{array}
$\\
$	\begin{array}[t]{lll}
		\phantom{\text{Hence: } W_n}
		&=&
		\displaystyle	{	n Y^{n - 1}
							\sum	_{q = 0}
									^{+ \infty}
									(-1)^{q} Te^{n(q + 1)}(X) Y^{qn}
						}
		\vspace{0.1cm}	\\
		&=&
		n Y ^{n - 1} V_n(X ;Y) \ .
	\end{array}
$

In the ring $\C(\!(X)\!)[\![Y]\!]$, we therefore have:
$$	\displaystyle	{	n Y ^{n - 1} V_n(X ; Y)
						=
						-
						\sum	_{k = 0}
								^{n - 1}
								e	^{(2k + 1)	\frac	{i \pi}	{n}	}
								Te^1	\left (
												X
												-
												e	^{(2k + 1)	\frac	{i \pi}	{n}	}
												Y
										\right )
					}  \ .
$$

The ring morphism $\varphi_n : \C(\!(X)\!)[\![Y]\!] \longrightarrow \C(\!(X)\!)[\![Y^{1/n}]\!]$
defined by $$\varphi_n(Y) = Y^{1/n}$$ is a continuous one for the $I$-adic topology; we hence
observe that if $P$ is a polynomial with coefficients in $\C(\!(X)\!)$, then
$\varphi (P(X ; Y)) = P(X ; Y^{1/n})$~. This can be extended to formal power series of
$\C(\!(X)\!)[\![Y]\!]$, using the continuity of $\varphi_n$ and the density of polynomials.

Transposed in $\C(\!(X)\!)[\![Y^{1/n}]\!]$ using the morphisms $\varphi_n$, the relation expressing $V_n(X ; Y)$ becomes in
$\C(\!(X)\!)[\![Y^{1/n}]\!]$~:
$$		\displaystyle	{	U_n(X ; Y)
							=
							- \frac {1} {n}
							\sum	_{k = 0}
									^{n - 1}
									e	^{(2k + 1)	\frac	{i \pi}	{n}}
									Te^1	\left (
													X
													-
													e^{(2k + 1)	\frac	{i \pi}	{n}}
													Y^{\frac{1}{n}}
											\right )
									Y^{\frac{1}{n} - 1}
						} \ .
$$

A priori, this last equality is in $\C(\!(X)\!)[\![Y^{1/n}]\!]$, while by definition we have $U_n \in \C(\!(X)\!)[\![Y]\!]$~. We can then proceed component by component in the ring $\C(\!(X)\!)[\![Y]\!]$~.

To express $T_n$ by using the general formula of solving a first order formal differential equation, it is sufficient to determine the exponential of the formal indefinite integral (in $Y$), without constant term, of $\omega Te^1(X + \omega Y)$ in $\C(\!(X)\!)[\![Y]\!]$~.

To this purpose, let us recall that we have proved in $\C(\!(X)\!)[\![Y]\!]$ the relation
$	Te^1 (X + Y)
	=
	\displaystyle	{	\frac	{\pi}	{\tan(\pi X + \pi Y)}	}
$ . Therefore, the formal indefinite integral in $Y$ of 
$\omega Te^1 (X + \omega Y)$, for $\omega \in \C$, without constant term, is given by
$	\text{Log }	\left (
						\displaystyle	{	\frac	{\sin (\pi(X + \omega Y))}
													{\sin (\pi X)}
										}
				\right )
$. Consequently, in $\C(\!(X)\!)[\![Y^{1/n}]\!]$, the formal indefinite primitive in $Y$ without constant term of
$\displaystyle{\frac{\omega}{n}} Te^1 \Big (X + \omega Y^{\frac{1}{n}} \Big ) Y^{\frac{1}{n} - 1}$ is
$	\text{Log }	\left (
						\displaystyle	{	\frac	{\sin (\pi(X + \omega Y^{\frac{1}{n}}))}
													{\sin (\pi X)}
										}
				\right )
$. Thus, by solving the formal differential equation in $\C(\!(X)\!)[\![Y^{1/n}]\!]$, we deduce that for all positive integer $n$:
$$	\displaystyle	{	T_n =	\sum	_{p = 0}
										^{+ \infty}
										Te^{n^{[p]}}(X) Y^p
							=	\frac	{	\displaystyle	{	\prod	_{k = 0}
																		^{n - 1}
																		\sin	\left (
																					\pi	\big (
																								X - e^{(2k + 1) \frac{i \pi}{n}} Y^{\frac{1}{n}}
																						\big )
																				\right )
															}
										}
										{	\sin^n (\pi X)	}
					} \ .
$$

Let us insist on the fact that, although seeming to be a priori a relation in $\C(\!(X)\!)[\![Y^{1/n}]\!]$, this equality holds in fact in
 $\C(\!(X)\!)[\![Y]\!]$, by definition of $T_n$~.

\subsubsection	{A new formal power series expansion of $T_n$}

In order to compute $Te^{n^{[p]}} (X)$ for $(n ; p) \in (\N^*)^2$, we need a formal power series expansion of $T_n$ expressed in another way than its definition. To get this new expansion, it is convenient to expand the product of many sinus terms.

It is easily seen, by induction on $n$, that in $\C[\![X_1 ; \cdots ; X_n]\!]$:
$$	\prod	_{k = 1}
			^n
			\sin (X_k)
	=
	\frac	{(-1)^{n - 1}}
			{2^n}
	\hspace{-0.3cm}
	\sum	_{(\varepsilon_1 ; \cdots ; \varepsilon_n) \in \{+ 1 ; - 1\}^n}
			\hspace{-0.3cm}
			(-1)^{\sharp \{ k \in \crochet{1}{n}  ;   \varepsilon_k = -1\}}
			\sin^{(n - 1)}	\left (
									\sum	_{k = 1}
											^n
											\varepsilon_k X_k
							\right ) \ .
$$

Let us consider the moulds $sg^\p$ , $e^\p$ and $s^\p$ , with values in $\C$ and defined over the alphabet
$\Omega = \{ 1 ; -1 \}$ for all sequences $\seq{\varepsilon}  \in \text{seq} (\Omega)$
by:\linebreak
$sg^{\seq{\varepsilon}} = \displaystyle	{	\prod	_{k = 1}
													^n
													\varepsilon_k
										}
						= (-1)^{\sharp \{i \in \crochet{1}{n}  ;   \varepsilon_i = - 1 \}}
$,																								 
$s^{\seq{\varepsilon}} = \displaystyle	{	\sum	_{k = 1}
													^n
													\varepsilon_k
										}
$ and	$e^{\seq{\varepsilon}} = \displaystyle	{	\sum	_{k = 1}
															^n
															\varepsilon_k e^{(2k - 1) \frac{i \pi}{n}}
												}
		$.

Let $E$ be the floor function and define all for $(k ; n) \in \N \times \N^*$ the functions $t_{k,n}$ by:
$$	\forall x \in \R\ ,  \,
	t_{k,n} (x)
	=
	\left \{
			\begin{array}{ll}
				\cos^{(n - 1)}(x)	&	\text{ , if }k \text{ is odd.}	\\
				\sin^{(n - 1)}(x)	&	\text{ , if }k \text{ is even.}	\\
			\end{array}
	\right.
$$

It follows that for all $n \in \N^*$, we have successively:
\\
\\
$	\begin{array}{@{}lll}
		T_n
		&=&
		\displaystyle	{	\frac	{(-1)^{n - 1}}	{(2 \sin (\pi X))^n}
							\sum	_{\seq{\varepsilon} = (\varepsilon_1 ; \cdots ; \varepsilon_n) \in \Omega^n}
									sg^{\seq{\varepsilon}}
									 
									\sin^{(n - 1)}	\left (
															s^{\seq{\varepsilon}} \pi X
															+
															e^{\seq{\varepsilon}} \pi Y^{\frac{1}{n}}
													\right )
						}
		\\
		&=&
		\displaystyle	{	\frac	{(-1)^{n - 1}}	{(2 \sin (\pi X))^n}
							\sum	_{\seq{\varepsilon} = (\varepsilon ;\cdots ; \varepsilon_n) \in \Omega^n}
									\left (
											\displaystyle	{	sg^{\seq{\varepsilon}}
																 
																\sum	_{k = 0}
																		^{+ \infty}
																		\frac	{	(-1)^{E ( \frac{k + 1}{2} )}
																					(e^{\seq{\varepsilon}} \pi)^k
																				}
																				{k!}
																		t_{k,n} (s^{\seq{\varepsilon}} \pi X)
																		Y^{\frac{k}{n}}
															}
									\right )
						}
		\\
		&=&
		\displaystyle	{	\sum	_{k = 0}
									^{+ \infty}
									\left (
											\displaystyle	{	\frac	{(-1)^{n - 1 + E ( \frac{k + 1}{2} )} \pi^k}
																		{k! (2 \sin (\pi X))^n}
																\sum	_{	\seq{\varepsilon}
																			=
																			(\varepsilon_1; \cdots; \varepsilon_n) \in \Omega^n
																		}
																		sg^{\seq{\varepsilon}}
																		 
																		(e^{\seq{\varepsilon}})^k
																		 
																		t_{k,n} (s^{\seq{\varepsilon}} \pi X)
																		Y^{\frac{k}{n}}
															}
									\right ) .
						}
	\end{array}
$

Note that we have used in the second equality
$$	\sin^{(n - 1)} (X + Y) = \displaystyle	{	\sum	_{k = 0}
														^{+ \infty}
														\frac	{(-1)^{E ( \frac{k + 1}{2} )}}
																{k!}
														t_{k,n}(X) Y^k
											} \text{ in } \C[\![X ; Y]\!] \ ,
$$
\noindent
and that in the last sum, $\Omega^n$ is a finite set.
\\

As already indicated, $T_n \in \C(\!(X)\!)[\![Y]\!]$ by definition of $T_n$. This implies that the coefficients of $Y^{\frac{k}{n}}$ if $n \nmid k$ vanish in the previous equality. Hence:

$$	\displaystyle	{	T_n
						=
						\sum	_{k = 0}
								^{+ \infty}
								\left (
										\displaystyle	{	\frac	{(-1)^{n - 1 + E ( \frac{kn + 1}{2} )} \pi^{kn}}
																	{(kn)! (2 \sin (\pi X))^n}
															\sum	_{	\seq{\varepsilon}
																		=
																		(\varepsilon_1; \cdots; \varepsilon_n) \in \Omega^n
																	}
																	sg^{\seq{\varepsilon}}
																	(e^{\seq{\varepsilon}})^{kn}
																	t_{kn,n} (s^{\seq{\varepsilon}} \pi X)
																	Y^{k}
														}
								\right ) .
					}
$$

It follows that we have proved, for all $(n ; k) \in \N^* \times \N$, the formal equality announced
in Property \ref{Property2Intro}:
$$	Te^{n^{[k]}} =	\displaystyle	{	\frac	{(-1)^{n - 1 + E ( \frac{kn + 1}{2} )} \pi^{kn}}
												{(kn)! (2 \sin (\pi X))^n}
										\sum	_{	\seq{\varepsilon}
													=
													(\varepsilon_1; \cdots; \varepsilon_n) \in \Omega^n
												}
												sg^{\seq{\varepsilon}}
												(e^{\seq{\varepsilon}})^{kn}
												t_{kn,n} (s^{\seq{\varepsilon}} \pi X) \ .
									}
$$

To conclude this computation, we only have to justify that the analytic equality follows from the
formal one. This is obvious because each (convergent or divergent) multitangent function is a Laurent
series at $0$ which is exactly given by the expression of the associated formal multitangent function.
In the componentwise equality which has just been proved, we can thus replace the straight capital
letters by cursive capital letters to conclude the proof of Property \ref{Property2Intro}.

\subsubsection	{A few examples}

For $n = 1$, this result gives, for $k \in \N^*$ and $z \in \C - \Z$:\\
$$	\begin{array}{@{}lll}
		\mathcal{T}e^{1^{[k]}}(z)	&=&	\left \{
												\begin{array}{ll}
													\displaystyle	{	\frac	{(-1)^p \pi^{2p}}	{(2p)!}
																	}
													&	\text{, if } k = 2p \ .
													\vspace{0.1cm}	\\
													\displaystyle	{	\frac	{(-1)^p \pi^{2p}}	{(2p + 1)!}
																		\mathcal{T}e^1(z)
																	}
													&	\text{, if } k = 2p + 1 \ .
												\end{array}
										\right.
	\end{array}
$$

Also, for $n = 2$, this result gives, for $k \in \N^*$ and $z \in \C - \Z$:\\
$$	\mathcal{T}e^{2^{[k]}}(z)
	=
	\displaystyle	{	\frac	{2^{2k - 1} \pi^{2k - 2}}
								{(2k)!}
						\mathcal{T}e^2(z) .
					}
$$

Table \ref{table8-1} gives some others explicit results from this property.

\input {table81.sty}

\subsection	{About odd, even or null multitangent functions}
\label{odd, even or null mtgf}
Surprisingly, there exists convergent multitangent functions which are null (see table \ref{table1}).
The first multitangent with this property is $\mathcal{T}e^{2,1,2}$. It is easy to see: the reduction
into monotangent functions imposes on $\mathcal{T}e^{2,1,2}$ to be $\C$-linearly dependent of
$\mathcal{T}e^{2}$, hence to be an even function~; nevertheless, the parity property tells us
$\mathcal{T}e^{2,1,2}$ is also an odd function. Necessary, the multitangent function
$\mathcal{T}e^{2,1,2}$ is the null function.

In the same manner, we can state the following lemma:

\begin{Lemma}
	Let $\seq{s} \in \mathcal{S}^\star_{b,e} \cap \text{seq} \big ( \{1 ; 2\} \big )$ be a symmetric sequence 
	(i.e. $\overset {\leftarrow} {\seq{s}} = \seq{s}$), of odd weight and of length greater than $1$.\\
	Then, $\mathcal{T}e^{\seq{s}}$ is the null function.
\end{Lemma}

When we look at a table of convergent multitangent functions up to weight $18$, it seems that the converse is also true:

\begin{Conjecture}
	\label{MTGF_nulles}
	\textit{(Characterisation of null multitangent functions)}
	The null convergent multitangent functions are exactly the multitangent functions $\mathcal{T}e^{\seq{s}}$ with symmetric sequence
	$\seq{s} \in \mathcal{S}_{b,e}^\star \cap \text{seq} \big ( \{1 ; 2\} \big )$,
	of odd weight and of length greater than $1$~.
\end{Conjecture}

We mentioned that this conjecture is true for length $3$ (see \cite{Bouillot}).
\\

Let us remark that even (resp. odd) components of an odd (resp. even) multitangent function are naturaly null. The following question is then an interesting one:
``\textit	{	If $\seq{s} \in \mathcal{S}_{b,e}^\star$ (or $\text{seq}(\N^*)$)~, is there any component $\mathcal{T}e^k$, $k \in \crochet{2}{max(s_1 ; \cdots ; s_r)}$ which does not appear in the reduction into monotangent functions of $\mathcal{T}e^{\seq{s}}$ ?}''

It seems that the answer might be no, except when the multitangent function not having this component is an odd or an even function.
\\

Another question is also: ``\textit	{	If $\mathcal{T}e^{\seq{s}}$ is an odd or even function,
do we have $\overset {\leftarrow} {\seq{s}} = \seq{s}$?}'' The answer seems to be yes.
The converse is already acquired, according to the parity property.
\\

This can be summed up in the following conjecture (which obviously implies the previous one)~:

\begin{Conjecture}
	\textit{(Characterisation of odd or even multitangent functions)}	\\
	Let $\seq{s} \in \mathcal{S}^\star_{b,e}$.\\
	$1$. If the component $\mathcal{T}e^{k}$, $k \in \crochet{2}{max(s_1 ; \cdots ; s_r)}$, does not appear in the reduction into monotangent functions of
		$\mathcal{T}e^{\seq{s}}$ , then $\mathcal{T}e^{\seq{s}}$ will be of opposite parity of $k$ (and thus may be the null function)~.
	\\
	$2$. The multitangent function $\mathcal{T}e^{\seq{s}}$ is an odd or even function if and only if $\overset {\leftarrow} {\seq{s}} = \seq{s}$~.
\end{Conjecture}		

\subsection	{Explicit computation of some multitangent functions}

The reduction into monotangent functions allows us to do some explicit computations of multitangent functions. We will give a few examples in the convergent case.

In order to apply this reduction simply, here are a few elementary remarks:
\begin{enumerate}
	\item Only the indexes $i$ satisfying $s_i \geq 2$ give a contribution to the expression of the reduction into monotangent functions.
	\item If $\seq{s} \in \mathcal{S}_{b,e}^\star \cap \text{seq} (\{1;2;3\})$ is a symmetric sequence
			(ie $\overset {\leftarrow} {\seq{s}} = \seq{s}$) of even weight, only the monotangent function $\mathcal{T}e^2$
			has to be considered; this means that only the indexes $k = 2$ give a contribution to the reduction.
	\item If $\seq{s} \in \mathcal{S}_{b,e}^\star \cap \text{seq} (\{1;2;3\})$ is a symmetric sequence
			of odd weight, only the monotangent function $\mathcal{T}e^3$ has to be considered; this means that only the indexes $k = 3$ give a contribution to
			the reduction.
\end{enumerate}

Applying these remarks, a simple computation gives us the results of the table \ref{table20}. 

\input{table9.sty}
\section	{Conclusion}

In this article, we have thoroughly investigated the algebra $\mathcal{M}TGF_{CV}$ of multitangent functions, spanned as a $\Q$-vector space by the functions:
$$	\begin{array}[t]{lcll} \label{definition te}
		\mathcal{T}e^{\seq{s}} :	&	\C - \Z	&	\longrightarrow	&	\C
		\\
									&	z			&	\longmapsto	&	\displaystyle	{	\sum	_{- \infty < n_r < \cdots < n_1 < + \infty}
																								\frac	{1}
																											{(n_1 + z)^{s_1} \cdots (n_r + z)^{s_r}}
																						} \ ,
	\end{array}
$$
for sequences in	$	\mathcal{S}^\star_{b,e} =	\big \{
															\seq{s} \in \text{seq}(\N^*)
															;
															s_1 \geq 2 \text{ and } s_{l(\seq{s})} \geq 2
													\big \}
					$~.

The first properties we have proved are elementary ones and concern the \symmetrelity of the mould $\mathcal{T}e^\p$, the differentiation property and the parity property. Another seemingly easy property is in fact a deep one, namely the reduction into monotangent functions:

\begin{Theorem*}	\textit{(Reduction into monotangent functions)}	\\
	For all sequences $\seq{s} = (s_1 ; \cdots ; s_r) \in \text{seq} (\N^*)$, there exists an explicit family
	$	(z^{\seq{s}}_k)_{k \in \crochet{0}{M}}
		\in
		\mathcal{M}ZV_{CV}
		^{M + 1}
	$, with
	$M = \displaystyle	{\max	_{i \in \crochet {1}{r}} s_i	}$, such that:
	$$	\forall z \in \C - \Z\ , \ 
		\mathcal{T}e^{\seq{s}} (z)
		=
		z^{\seq{s}}_0
		+
		\displaystyle	{	\sum	_{k = 1}
									^{M}
									z^{\seq{s}}_k \mathcal{T}e^k (z)
						}\ .
	$$
	Moreover, if $\seq{s} \in \mathcal{S}^\star_{b,e}$, then $z^{\seq{s}}_0 = z^{\seq{s}}_1 = 0$ .	\\
\end{Theorem*}

We have then immediately derived that for all $p \geq 2$, we have:
$$ 	\displaystyle	{	\mathcal{M}TGF_{CV , p} \subseteq \bigoplus	_{k = 2}
																	^p
																	\mathcal{M}ZV_{CV , p - k} \cdot \mathcal{T}e^k
					} \ .
$$

Then, we have explained why the reduction into monotangent functions is such an important operation. The reason is that this process has in a certain sense a converse, namely the projection onto multitangent functions. According to Conjectures \ref{structure of the algebra MTGF, version 2} , \ref{MZV-module structure} and \ref{new conjecture on projection} and Properties \ref{prop-proj-1} and \ref {prop-proj-2}, we have proved the following:

\begin{Theorem*} \textit{(Projection onto multitangent functions)}	\\
	The following hypothetically statement are equivalent:
	\begin{enumerate}
		\item For all $p \geq 2$ ,	$	\displaystyle	{	\mathcal{M}TGF_{CV , p} = \bigoplus	_{k = 2}
																								^{p}
																								\mathcal{M}ZV_{CV , p - k} \cdot \mathcal{T}e^k
																		} \ .
													$
		\item $\mathcal{M}TGF_{CV}$ is a $\mathcal{M}ZV_{CV}$-module.
		\item For all sequence $\seq{\pmb{\sigma}} \in \mathcal{S}_e^\star$ ,
				$\mathcal{Z}e^{\seq{\pmb{\sigma}}} \mathcal{T}e^2 \in \mathcal{M}TGF_{CV,||\seq{\pmb{\sigma}}|| + 2}$ .
	\end{enumerate}
\end{Theorem*}

By a linear algebra argument, we have explained that the largest $p$ is, the stronger are the reasons to believe in the previous assertions. We have verified them up to weight $18$~.

The third important fact, which was used during the regularization process of divergent multitangent
functions, is the trifactorization of the multitangent functions: all multitangents can be expressed
as a finite product of Hurwitz multizeta functions in such a way to preserve the exponentially flat
character of multitangent functions.

\bigskip

Finally, the links between the algebra of multizeta values and the algebra of multitangent functions are summed up by these three properties and the following diagram:

$$	\xymatrix	{		\mathcal{M}ZV_{CV}	\ar@{.>}@<4pt>[ddd]	^{\text{projection}}	&&&		\mathcal{H}MZF_{+,CV}	\ar[lll]	_{\text{evaluation at 0}}
																														\ar@{^(->}[ddd]
						\\
						\\
						\\
						\mathcal{M}TGF_{CV}	\ar[uuu]			^{reduction}
											\ar@{^(->}[rrr]		^{\text{trifactorization}}	&&&		\mathcal{H}MZF_{\pm, CV}
				}
$$

As an example of the ``duality'' multizeta values/multitangent functions, we have explained that if
the hypothetical $\mathcal{M}ZV_{CV}$-module structure holds, then we have a conjecture concerning
the dimension of $\mathcal{M}TGF_{CV, p}$ which is actually equivalent to Zagier's conjecture on
multizeta values. This justifies the table \ref{conjectural dimension} of conjectural dimensions.

\begin{figure}[h]
	$$	\begin{array}{|c|c|c|c|c|c|c|c|c|c|c|c|c|c|}
			\hline
			p										&	0	&	1	&	2	&	3	&	4	&	5	&	6	&	7	&	8	&	9	&	10	&	11	&	12	\\
			\hline
		
			\text{dim } \mathcal{M}ZV_{CV , p}		&	1	&	0	&	1	&	1	&	1	&	2	&	2	&	3	&	4	&	5	&	7	&	9	&	12	\\
			\hline
		
			\text{dim } \mathcal{M}TGF_{CV , p + 2}	&	1	&	1	&	2	&	3	&	4	&	6	&	8	&	11	&	15	&	20	&	27	&	36	&	48	\\
			\hline
		\end{array}
	$$

	\caption	{The first hypothetical dimensions of multitangent vector space of weight $p + 2$.}
	\label {conjectural dimension}
\end{figure}

If these dimensions looks reasonable, this is because of the existence of many $\Q$-linear relations
between multitangent functions. For instance,
$$	4 \mathcal{T}e^{3, 1, 3} - 2 \mathcal{T}e^{3, 1, 1, 2} + \mathcal{T}e^{2, 1, 2, 2} = 0	$$

\noindent
is an interesting relation because it implies a relation between multizeta values, discussed at the end of Section \ref{algebraic properties}.

\bigskip
\null
\bigskip

Now, the remaining question is to find a new method to prove that there is no non-trivial $\Q$-linear relations between the multitangent functions which are supposed to span $\mathcal{M}TGF_{CV , p}$. To illustrate this, if we were able to prove the absence of non-trivial $\Q$-linear relations between $\mathcal{T}e^5$ , $\mathcal{T}e^{3,2}$ , $\mathcal{T}e^{2,3}$ and $\mathcal{T}e^{2}$, this would imply a well-known fact: $\zeta(3) = \mathcal{Z}e^3 \not \in \Q$~. Such a partial result would already be an important breakthrough, because such a method would certainly be generalisable to other weights, while Apery's method is not. Nevertheless, such a method would probably not give an upper bound of the irrationality measure, while Apery's method can.

Probably, the new method would come from the study of the Hurwitz multizeta functions, and more precisely from the study of algebraic relations in the algebra $\mathcal{H}MZV_{\pm, CV}$~.
\appendix
\renewcommand*{\thesection}{\Alph{section}}
\section	{Introduction to mould notations and calculus}
\label{Elements de calcul moulien debut}

For all this annex, references can be found in many text of Jean Ecalle. See for instance
\cite{Ecalle1.5} or \cite{Ecalle3} ; for other presentations of mould calculus, see
\cite{Cresson} or \cite{Sauzin}.

\subsection{Notion of moulds}
\label{AppendixDefinition}

If $\Omega$ is a set, $\text{seq}(\Omega)$ will denote in the sequel the set of (finite) sequence
of positive integers:
$$	\text{seq}(\N^*)
	=
	\displaystyle	{	\{ \emptyset \}
						\cup
						\bigcup	_{r \in \N^*}
								\{
									(s_1 ; \cdots ; s_r) \in (\N^*)^r
								\}
						\ .
					}
$$

A \textbf{mould} is a function defined over a free monoid $\text{seq} (\Omega)$
or sometimes over a subset of $\text{seq} (\Omega)$, with values in an algebra $\mathbb{A}$.
Concretely, this means that ``a mould is a function with a variable number of variables''.

Sometimes, it may be usefull to see a mould as a collection of functions $(f_0, f_1, f_2, \cdots)$, where, for all non negative integer $i$, $f_i$ is a function defined on $\Omega^i$ (and consequently, $f_0$ is a constant function)~.

Thus, moulds depend on sequences $\seq{w} = (w_1 ; \cdots ; w_r)$ of any length $r$. The empty sequence is denoted by $\emptyset$. The variables $w_i$ are elements of $\Omega$. We will often identify sequences of $\text{seq} (\Omega)$ and non-commutative polynomials over the alphabet $\Omega$.

In a general way, we will use the mould notations:
\begin{enumerate}
	\item	Sequences will always be written in bold and underlined, with an upper indexation
		if necessary. We call length of $\seq{w}$ and denote $l(\seq{w})$ the number of
		elements of $\seq{w}$. Without more precisions, we will use the letter $r$ to
		indicate the length of any sequence. We also define the weight of $\seq{w}$,
		when $\Omega$ has a semi-group structure, by: $$||\seq{w}|| = w_1 + \cdots + w_r \ .$$
	\item	For a given mould, we will prefer the notation $M^{\seq{\pmb{\omega}}}$, which indicate the evaluation of the mould $M^\p$ on the sequence $\seq{\pmb{\omega}}$ of $\text{seq} (\Omega)$, to the functionnal notation which would have been $M(\seq{\pmb{\omega}})$~.
	\item	We shall use the notation $\mathcal{M}_{\mathbb{A}}^\p (\Omega)$ to refer to the set
		of all the moulds constructed over the alphabet $\Omega$ with values in the algebra $\mathbb{A}$~.
\end{enumerate}

\subsection{Mould operations}
\label{AppendixOperations}
Moulds can be, among other operations, added, multiplied by a scalar as well as multiplied, composed, and so on. In this article, among the operations we will use, only the multiplication needs to be defined:
if	$(A^\p ; B^\p) \in	\left (
					\mathcal{M}^\p_{\mathbb{A}} (\Omega)
				\right )
				^2
	$, then, the mould multiplication $M^\p = A^\p \times B^\p$ is defined for all sequences $\underline{\pmb{\omega}} \in \text{seq} (\Omega)$ by:
$$	M^{\seq{\pmb{\omega}}}
	=
	\displaystyle	{	\sum	_{	(\seq{\pmb{\omega}}^1 ; \seq{\pmb{\omega}}^2) \in \text{seq}(\Omega)^2
						\atop
						\seq{\pmb{\omega}}^1 \cdot  \seq{\pmb{\omega}}^2 = \seq{\pmb{\omega}}
					}
					A^{\seq{\pmb{\omega}}^1}
					B^{\seq{\pmb{\omega}}^2}
				=
				\sum	_{i = 0}
					^{l(\seq{\pmb{\omega}})}
					A^{\seq{\pmb{\omega}}^{\leq i}}
					B^{\seq{\pmb{\omega}}^{> i}}
			} \ .
$$

A few explanations relative to notations used here have to be given. For a sequence
$\seq{\pmb{\omega}} = (\omega_1 ; \cdots ; \omega_r) \in \text{seq} (\Omega)$ and an integer
$k \in \crochet{0}{r}$, we write:
$$	\begin{array}{@{}lll}	\label{lecture}
		\seq{\pmb{\omega}}^{\leq k}
		=
		\left \{
				\begin{array}{l@{}l}
					\emptyset						&	\text{ , if } k = 0	\\
					(\omega_1 ; \cdots ; \omega_k)	&	\text{ , if } k > 0
				\end{array}
		\right.
		&\text{and}&
		\seq{\pmb{\omega}}^{> k}
		=
		\left \{
				\begin{array}{l@{}l}
					(\omega_{k + 1} ; \cdots ; \omega_r)	&	\text{ , if } k < r \ .	\\
					\emptyset								&	\text{ , if } k = r \ .
				\end{array}
		\right.
	\end{array}
$$

Let us remark that the two deconcatenations $\emptyset \cdot \seq{\omega}$ and $\seq{\omega} \cdot \emptyset$ intervene in the definition of the mould multiplication and refer respectively to the index $i = 0$ and $i = l(\seq{\omega})$. 
We will denote such a product in a short way by:
$$	(A^\p \times B^\p)^{\seq{\omega}}
	=
	\displaystyle	{	\sum	_{\seq{\omega}^1 \cdot \seq{\omega}^2 = \seq{\omega}}
								A^{\seq{\omega}^1} B^{\seq{\omega}^2}
					}
	\ .
$$

Finally, $	\left (
					\mathcal{M}^\p_{\mathbb{A}} (\Omega) , + , \cdot , \times
			\right )
		$ is an associative but non-commutative $\mathbb{A}$-algebra with unit, whose invertibles
are easily characterised:
$$	\left (
			\mathcal{M}^\p_{\mathbb{A}} (\Omega)
	\right )
	^\times
	=
	\{	M^\p \in \mathcal{M}^\p_{\mathbb{A}} (\Omega)
		\, ; \,
		M^\emptyset \in \mathbb{A}^\times
	\}
	\ .
$$
We will denote by $(M^\p)^{\times -1}$ the multiplicative inverse of a mould $M^\p$, when it exists.

\subsection{\Symmetrality}
\label{AppendixSymmetrality}
Let us first remind that the shuffle product of two words $P = p_1 \cdots p_r$ and $Q = q_1 \cdots q_s$ constructed over the alphabet $\Omega$ is denoted by $\shuffle$ and recursively defined by:
$$	\left \{
		\begin{array}{@{}l@{}}
			P \shuffle \varepsilon
			=
			\varepsilon \shuffle P
			=
			P \ \ ,
			\\
			P \shuffle Q
			=
			p_1 \big( p_2 \cdots p_r \shuffle Q \big)
			+
			q_1 \big( P \shuffle q_2 \cdots q_r \big)
			
			\ ,
		\end{array}
	\right.
$$
where $\varepsilon$ is the empty word. As an example, if $P = a \cdot b$ and $Q = c$, we have $P \shuffle Q = abc + acb + cab$~.

In order to have a better understanding of the shuffle, a visual representation of it can be usefull.
A word can be seen as a desk of card, then the shuffle of two words becomes the set of all the obtained result
by inserting \textit{classically} one desk of cards in the other one.
\\

The multiset $sh\text{\textbf{\underline{\textit{a}}}}(\seq{\pmb{\alpha}} ; \seq{\pmb{\beta}})$,
where $\seq{\pmb{\alpha}}$ and $\seq{\pmb{\beta}}$ are sequences of $\text{seq} (\Omega)$, is
defined to be the set of all monomials that appear in the non-commutative polynomial
$\seq{\pmb{\alpha}} \shuffle \seq{\pmb{\beta}}$, counted with their multiplicity.
\\

When $\mathbb{A}$ is an algebra, we define a \symmetral mould $Ma^\p$ to be a mould of
$\mathcal{M}_{\mathbb{A}}^\p (\Omega)$ which satisfies for all $(\seq{\pmb{\alpha}} ; \seq{\pmb{\beta}}) \in \big (\text{seq} (\Omega) \big )^2$:

\begin{equation}
	\left \{
		\begin{array}{l}
			\displaystyle	{	Ma^{\seq{\pmb{\alpha}}}
								Ma^{\seq{\pmb{\beta}}}
								=
								\sum	_{	\seq{\pmb{\gamma}}
										\in
										sh\text{\textbf{\underline{\textit{a}}}} (\seq{\pmb{\alpha}}   ;   \seq{\pmb{\beta}})
									}
									\hspace{-0.5cm}'  \hspace{0.2cm}
									Ma^{\seq{\pmb{\gamma}}}
							} \ .
			\\
			Ma^\emptyset = 1 \ .
		\end{array}
	\right.
\end{equation}

Here, the sum	$	\displaystyle	{	\hspace{-0.3cm}
										\sum	_{	\seq{\pmb{\gamma}}
													\in
													sh\text{\textbf{\underline{\textit{a}}}} (\seq{\pmb{\alpha}}   ;    \seq{\pmb{\beta}})
												}
												 \hspace{-0.5cm}'  \hspace{0.2cm}
												Ma^{\underline{\pmb{\gamma}}}
									}
				$ is a shorthand for
$	\displaystyle	{	\sum	_{\seq{\pmb{\gamma}} \in \text{seq} (\Omega)}
								\text{mult}	\binom{\seq{\pmb{\alpha}} \,;\, \seq{\pmb{\beta}}}{\seq{\pmb{\gamma}}}
								Ma^{\seq{\pmb{\gamma}}}
					}
$, where $\text{mult}	\binom{\seq{\pmb{\alpha}} \,;\, \seq{\pmb{\beta}}}{\seq{\pmb{\gamma}}}$ is the coefficient of the monomial $\seq{\pmb{\gamma}}$ in the product $\seq{\pmb{\alpha}} \shuffle \seq{\pmb{\beta}}$ and is equal to 
$\langle \seq{\pmb{\alpha}} \shuffle \seq{\pmb{\beta}} | \seq{\pmb{\gamma}} \rangle$~. From now on, we shall omit the prime on the sum:

$$	\displaystyle	{	Ma^{\seq{\pmb{\alpha}}} Ma^{\seq{\pmb{\beta}}}
						=
						\sum	_{\seq{\pmb{\gamma}} \in \text{seq} (\Omega)}
								\langle
										\seq{\pmb{\alpha}} \shuffle \seq{\pmb{\beta}} | \seq{\pmb{\gamma}}
								\rangle
								Ma^{\seq{\pmb{\gamma}}}
						=
						\sum	_{\seq{\pmb{\gamma}} \in sh\text{\textbf{\underline{\textit{a}}}}(\seq{\pmb{\alpha}} , \seq{\pmb{\beta}})}
								Ma^{\seq{\pmb{\gamma}}} \ .
					}
$$

The \symmetrality imposes, through a multitude of relations, a strong rigidity. For example, if $(x ; y) \in \Omega^2$ and $Ma^\p$ denote a \symmetral mould, then we have necessarily:
$$	\begin{array}{lll}
			Ma^{x} Ma^{y}
			&=&
			Ma^{x,y} + Ma^{y,x} \ .
			\\
			Ma^{x,y} Ma^{y}
			&=&
			Ma^{y, x, y} + 2Ma^{x, y, y} \ .
	\end{array}
$$

The definition of \symmetrelity may also apply to a mould which is defined only on a subset $D$ of $\text{seq}(\Omega)$ (in which case, we require $D$ to be stable by the shuffle).

\subsection{\Symmetrelity}
\label{AppendixSymmetrelity}
Let $(\Omega , \cdot)$ be an alphabet with a semi-group structure. Let us first remind that the stuffle product of two words $P = p_1 \cdots p_r$ and $Q = q_1 \cdots q_s$ constructed over the alphabet $\Omega$ is denoted by $\stuffle$ and defined recursively by:
$$	\left \{
		\begin{array}{@{}l@{}}
			P \stuffle \varepsilon
			=
			\varepsilon \stuffle P
			=
			P \ ,
			\\
			P \stuffle Q
			=
			p_1 \big( p_2 \cdots p_r \stuffle Q \big)
			+
			q_1 \big( P \stuffle q_2 \cdots q_s \big)
			+
			(p_1 \cdot q_1) \big ( p_2 \cdots p_r \stuffle q_2 \cdots q_s \big )
			\ ,
		\end{array}
	\right.
$$
where $\varepsilon$ is again the empty word. As an example, in $\text{seq}(\N)$, if $P = 1 \cdot 2$ and $Q = 3$, we have:
$P \stuffle Q = 1 \cdot 2 \cdot 3 + 1 \cdot 3 \cdot 2 + 3 \cdot 1 \cdot 2 + 1 \cdot 5 + 4 \cdot 2~.$

As well as for the shuffle product, a visual representation of the stuffle product can be usefull.
Seeing one more time a word as a desk of card, the stuffle of two words becomes the set of all
the obtained results by inserting \textit{magically} one desk of blue cards in a desk of red cards.
By magically, we mean that some new cards may appear: these new ones are hydrid cards, that is, one of
their sides is blue while the other is red. Such a hybrid card can only be obtained when two cards of
different colors are situated side by side in a classic shuffle of the two desks of cards. In the 
previous example, the hybrid cards are $5$, coming from $2 + 3$, and $4$, from $3 + 1$~.
\\

The multiset $sh\text{\textbf{\underline{\textit{e}}}}(\seq{\pmb{\alpha}} ; \seq{\pmb{\beta}})$, where $\seq{\pmb{\alpha}}$ and $\seq{\pmb{\beta}}$ are sequences in $\text{seq} (\Omega)$, is defined to be the set of all monomials that appear in the non-commutative polynomial $\seq{\pmb{\alpha}} \stuffle \seq{\pmb{\beta}}$, counting with their multiplicity.
\\

When the alphabet $\Omega$ is an additive semigroup and $\mathbb{A}$ an algebra, we define a \symmetrel mould $Me^\p$ to be a mould of
$\mathcal{M}_{\mathbb{A}}^\p (\Omega)$ which satisfies for all $(\seq{\pmb{\alpha}} ; \seq{\pmb{\beta}}) \in \big (\text{seq} (\Omega) \big )^2$:

$$	\left \{
		\begin{array}{l}\displaystyle	
			Me^{\seq{\pmb{\alpha}}}
			Me^{\seq{\pmb{\beta}}}
			=
			\sum	_{	\seq{\pmb{\gamma}}
					\in
					sh\text{\textbf{\underline{\textit{e}}}} (\seq{\pmb{\alpha}} \,  ; \,  \seq{\pmb{\beta}})
				}
				 \hspace{-0.5cm}'  \hspace{0.2cm}
				Me^{\seq{\pmb{\gamma}}}
			\ .
			\\
			Me^\emptyset = 1\ .
		\end{array}
	\right.
$$

Here, the sum	$	\displaystyle	{	\hspace{-0.3cm}
										\sum	_{	\seq{\pmb{\gamma}}
													\in
													sh\text{\textbf{\underline{\textit{e}}}} (\seq{\pmb{\alpha}} \,  ;   \, \seq{\pmb{\beta}})
												}
												 \hspace{-0.5cm}'  \hspace{0.2cm}
												Me^{\underline{\pmb{\gamma}}}
									}
				$ is a shorthand for
$	\displaystyle	{	\sum	_{\seq{\pmb{\gamma}} \in \text{seq} (\Omega)}
								\text{mult}	\binom{\seq{\pmb{\alpha}} \,;\, \seq{\pmb{\beta}}}{\seq{\pmb{\gamma}}}
								Me^{\seq{\pmb{\gamma}}}
					}
$, where $	\text{mult}	\binom{\seq{\pmb{\alpha}} ; \seq{\pmb{\beta}}}{\seq{\pmb{\gamma}}}$ is the coefficient of the monomial $\seq{\pmb{\gamma}}$ in the product $\seq{\pmb{\alpha}} \stuffle \seq{\pmb{\beta}}$ and is equal to 
$\langle \seq{\pmb{\alpha}} \stuffle \seq{\pmb{\beta}} | \seq{\pmb{\gamma}} \rangle$~.
From now on, we also omit the prime:

$$	\displaystyle	{	Me^{\seq{\pmb{\alpha}}} Me^{\seq{\pmb{\beta}}}
						=
						\sum	_{\seq{\pmb{\gamma}} \in \text{seq} (\Omega)}
								\langle
										\seq{\pmb{\alpha}} \stuffle \seq{\pmb{\beta}} | \seq{\pmb{\gamma}}
								\rangle
								Me^{\seq{\pmb{\gamma}}}
						=
						\sum	_{\seq{\pmb{\gamma}} \in sh\text{\textbf{\underline{\textit{e}}}}(\seq{\pmb{\alpha}} , \seq{\pmb{\beta}})}
								Me^{\seq{\pmb{\gamma}}} \ .
					}
$$

As well as the symmetr\textbf{\textit {\underline {a}}}lity, the \symmetrelity imposes a strong rigidity. For example, if $(x ; y) \in \Omega^2$ and $Me^\p$ denote a \symmetrel mould, then we have necessarily:
$$	\begin{array}{lll}
			Me^{x} Me^{y}
			&=&
			Me^{x,y} + Me^{y,x} + Me^{x + y} .
			\\
			Me^{x,y} Me^{y}
			&=&
			Me^{y, x, y} + 2Me^{x, y, y} + Me^{x + y, y} + Me^{x, 2y} \ .
	\end{array}
$$

The definition of \symmetrelity may also apply to a mould which is defined only on a subset $D$ of
$\text{seq}(\Omega)$ (in which case, we require $D$ to be stable by stuffle).

\subsection{Formal moulds}
\label{Formal Moulds}

Remind that a mould over the alphabet $\Omega$ can be seen as a collection of functions $(f_0, f_1, f_2, \cdots)$, where, for all non negative integer $i$, $f_i$ is a function defined on $\Omega^i$. Let us now define a formal mould as a collection of formal series $(S_0, S_1, S_2, \cdots)$, where, for all non negative integer $i$, $S_i$ is a formal power series in $i$ indeterminates (and consequently, $S_0$ is constant)~. The set of all formal mould with values in the algebra $\mathbb{A}$ is denoted by $\mathcal{FM}^\p_{\mathbb{A}}$~. So, by definition, it is clear that:
\begin{equation}	\label{FM subset of M}
	\mathcal{FM}^\p_{\mathbb{A}}
	\subset
	\mathcal{M}^\p_{\mathbb{A}} (X_1, X_2, X_3, \cdots)
	:=
	\underset	{n \longrightarrow + \infty}{\text{lim ind}}
	\mathcal{M}^\p_{\mathbb{A}}(X_1, \cdots, X_n)
	\ .
\end{equation}

For a mould $M^\p$, defined over an infinite alphabet of indeterminates $(X_1, X_2, \cdots)$, there is absolutely no reason that $M^{X_1, X_2}$ be related to
$M^{X_2, X_1}$~; but, for a formal mould $M^\p$, there is a link since it is the substitution of a sequences of indeterminates in a formal series. Consequently,
it's impossible to have an equality in \eqref{FM subset of M}.

Nevertheless, a formal mould is a mould. This implies in particular that we can add, multiply them for instance.

\subsection{\Symmetrility}
\label{symmetrelityAppendix}

If $Me^\p$ is a \symmetrel mould over $\text{seq} (\N^*)$~, with values in a commutative algebra $\mathbb{A}$, then its generating functions, denoted by $Mig^\p$, is a formal mould defined by:
$$	\left \{
			\begin{array}{l}
					Mig^\emptyset = 1 \ .	\\
					Mig^{v_1, \cdots, v_r}
					= \displaystyle	{	\sum	_{s_1, \cdots, s_r \geq 1}
												Me^{s_1, \cdots, s_r} {v_1}^{s_1 - 1} \cdots {v_r}^{s_r - 1}
									}
						\in	\mathbb{A} [ \! [ v_1 ; \cdots ; v_r ] \! ]\ .
			\end{array}
	\right.
$$

The mould $Mig^\p$ is then automatically a \symmetril mould, that is to say that it satisfies the relation:
$$	Mig^{\seq{v}} Mig^{\seq{w}}
	=
	\displaystyle	{	\sum	_{\seq{x} \in sh\text{\textbf{\underline{\textit{i}}}} (\seq{v}  ;  \seq{w})}
								Mig ^{\seq{x}}
					} \ .
$$

The multiset $sh\text{\textbf{\underline{\textit{i}}}} (\seq{v} ; \seq{w})$ is also a quasi-shuffle product as defined in \cite{Hoffmann}, as i the stuffle. If $\seq{v}$ and $\seq{w}$ are sequences over an alphabet of indeterminates, this set is defined exactly in the same way as $sh\text{\textbf{\underline{\textit{e}}}} (\seq{v} ; \seq{w})$, but here, the contraction of the quasi-shuffle product is an (formal) contraction defined over the set of (formal) indeterminates. The evaluation of a mould $Mig^\p$ on a sequence which has such a contraction is then done by induction and given by the formula:
$$	Mig^{\seq{v} \cdot (x \circledast y) \cdot \seq{w}}
	=
	\displaystyle	{	\frac	{Mig^{\seq{v} \cdot x \cdot \seq{w}} - Mig^{\seq{v} \cdot y \cdot \seq{w}}}
								{x - y}
					}
	\ .
$$

For example, a \symmetril mould $Mig^\p$ satisfies:
$$	\begin{array}{lll}
		Mig^X Mig^Y	&=&	Mig^{X,Y} + Mig^{Y,X} + \displaystyle	{	\frac	{Mig^X - Mig^Y}	{X - Y}
																} \ ,
		\\
		Mig^{X,Y} Mig^Z	&=&	Mig^{X,Y,Z} + Mig^{X,Z,Y} + Mig^{Z,X,Y} + \displaystyle	{	\frac	{Mig^{X,Y} - Mig^{X,Z}}	{Y - Z}
																					}
		\\				&&	+
							\displaystyle	{	\frac	{Mig^{X,Y} - Mig^{Z,Y}}	{X - Z}
											} \ .
	\end{array}
$$

To evaluate $Mig^{X,Y} Mig^Y$, we have to use derivative:
$$	Mig^{X,Y} Mig^Y	=	2Mig^{X,Y,Y} + Mig^{Y,X,Y} + D_Y Mig^{X,Y} + D_X Mig^{X,Z} \ .	$$

\subsection	{Some examples of rules}

Envisaged as a simple system of notations, the mould language already leads us to concise formulas
as well as the economy of long sequences of indexes. But its real utility resides in the different
mould operations and the rules which indicate how these affect (preserve or transform) basic symmetries.

\noindent
For example:	\begin{tabular}[t]{ll}
					$1$.	&	$\left [ \text{altern\text{\textbf{\underline{\textit{a}}}}l} \,; \text{altern\text{\textbf{\underline{\textit{a}}}}l} \right ] = \text{altern\text{\textbf{\underline{\textit{a}}}}l} $~, where $[ \cdot , \cdot]$ is a Lie bracket.
					\\
					$2$.	&	symmetr\text{\textbf{\underline{\textit{e}}}}l
							$\times$
							symmetr\text{\textbf{\underline{\textit{e}}}}l
							=
							symmetr\text{\textbf{\underline{\textit{e}}}}l~.
					\\
					$3$.	&	altern\text{\textbf{\underline{\textit{a}}}}/\text{\textbf{\underline{\textit{e}}}}l conjugated by symmetr\text{\textbf{\underline{\textit{a}}}}/\text{\textbf{\underline{\textit{e}}}}l = altern\text{\textbf{\underline{\textit{a}}}}/\text{\textbf{\underline{\textit{e}}}}l~.
					\\
					$4$.	&	exponential $\left (\text{altern\text{\textbf{\underline{\textit{a}}}}/\text{\textbf{\underline{\textit{e}}}}l} \right )$ = symétr\text{\textbf{\underline{\textit{a}}}}/\text{\textbf{\underline{\textit{e}}}}l~.
					\\
				\end{tabular}

\subsection {Some notations}
\label{AppendixNotations}

We will always write in bold, italic and underlined the vowel which indicates not only a 
symmetry of the considered moulds, but also the nature of the products of sequences which
will appear. Using this, it will become simpler to distinguish 
\symmetral , \symmetrel and \symmetril moulds as well as to distinguish the set
$sh\text{\textbf{\underline{\textit{a}}}}(\seq{\pmb{\alpha}} ; \seq{\pmb{\beta}})$,
$sh\text{\textbf{\underline{\textit{e}}}}(\seq{\pmb{\alpha}} ; \seq{\pmb{\beta}})$ and
$sh\text{\textbf{\underline{\textit{i}}}}(\seq{\pmb{\alpha}} ; \seq{\pmb{\beta}})$.

The moulds that we consider will carry in their name the vowelic alteration which immediately indicates their symmetry type. For example, the mould $\mathcal{T}e^\p (z)$ is a \symmetrel mould (see p. \pageref {definition te}), while
$\mathcal{Z}ig^\p$ is \symmetril (see p. \pageref {definition zig}).
The absence of this vowel will also indicate that the mould verifies no symmetry.
\\

Finally, if $\seq{\pmb{\alpha}} = (\alpha_1 ; \cdots ; \alpha_n)$ is a sequence constructed over an alphabet $\Omega$, we respectively denote by 
$\overset {\leftarrow} {\seq{\pmb{\alpha}}}$ and $\seq{\pmb{\alpha}}^{[k]}$, the reversed sequence $\seq{\pmb{\alpha}}$ and the sequence $\seq{\pmb{\alpha}}$ repeated $k$ times:
$$	\begin{array}{lllll}
		\overset {\leftarrow} {\seq{\pmb{\alpha}}}
		=
		(\alpha_n ; \cdots ; \alpha_1)
		&\hspace{1.2cm}&,&\hspace{1.2cm}&
		\seq{\pmb{\alpha}}^{[k]} = \underbrace {\seq{\pmb{\alpha}} \cdots \seq{\pmb{\alpha}}}_{k \text{ times}} \ .
	\end{array}
$$
Remind that we can also extract a part of the sequence $\seq{\pmb{\alpha}}$~. If $i$ and $j$ are two non-negative integers such that $i \leq j \leq r$, the
sequences $\seq{\pmb{\alpha}}^{\leq i}$ , $\seq{\pmb{\alpha}}^{i < \cdot \leq j}$ and $\seq{\pmb{\alpha}}^{> i}$ are defined by:
$$	\begin{array}{lllllllll}
		\seq{\pmb{\alpha}}^{\leq i}
		=
		(\alpha_1 ; \cdots ; \alpha_i)
		&,&
		\seq{\pmb{\alpha}}^{i < \cdot \leq j}
		=
		(\alpha_{i + 1} ; \cdots ; \alpha_j)
		&,&
		\seq{\pmb{\alpha}}^{> i}
		=
		(\alpha_{i + 1} ; \cdots ; \alpha_r) \ .
	\end{array}
$$
\label{Elements de calcul moulien fin}

\newpage
\section{Tables}
\label{Tables}

\clearpage

\bibliographystyle{model1b-num-names}

\end{document}